\documentclass{article}
\usepackage[utf8]{inputenc}
\usepackage[margin=1.2in]{geometry} 

\usepackage{amssymb}
\usepackage{mathtools}
\usepackage[autostyle=true]{csquotes}
\usepackage{mathrsfs}
\usepackage{graphicx}
\usepackage{dsfont}
\usepackage{subcaption}
\usepackage{appendix}
\usepackage{stmaryrd}
\usepackage{hyperref}

\usepackage[style=list,hyperfirst=false,toc]{glossaries}

\usepackage{enumitem}
\usepackage{accents}
\usepackage{framed}
\usepackage[amsmath,thmmarks,framed,hyperref]{ntheorem}
\usepackage{float}
\usepackage{color}
\usepackage{svg}
\usepackage[capitalise]{cleveref}
\usepackage[framemethod=TikZ]{mdframed}
\usepackage{multicol}

\hypersetup{
    colorlinks=true,
    linkcolor=blue,
    citecolor = red,
    filecolor=magenta,      
    urlcolor=cyan
}

\newcommand{\R}{\mathbb{R}}

\newcommand{\Z}{\mathbb{Z}}
\newcommand{\N}{\mathbb{N}}
\newcommand{\Q}{\mathbb{Q}}
\newcommand{\C}{\mathbb{C}}
\newcommand{\A}{\mathbb{A}}
\newcommand{\Sp}{\mathbb{S}}
\newcommand{\Di}{\mathbb{D}}
\newcommand{\M}{\mathbb{M}}
\newcommand{\I}{\mathbb{I}}
\newcommand{\T}{\mathbb{T}}
\newcommand{\V}{\mathbb{V}}
\newcommand{\U}{\mathbb{U}}
\newcommand{\W}{\mathbb{W}}
\newcommand{\B}{\mathbb{B}}

\newcommand{\intset}[2]{\llbracket #1 , #2 \rrbracket}
\newcommand{\croset}[1]{\left\lbrace #1 \right\rbrace}

\newcommand{\symp}[2]{\mathrm{Symp}^{#1}(#2)}
\newcommand{\sympan}[2]{\mathrm{Symp}_{#1}^{\omega}(#2)}
\newcommand{\sympo}[2]{\mathrm{Symp}^{#1}(#2; \partial\check{#2})}
\newcommand{\sympc}[2]{\mathrm{Symp}_c^{#1}(#2; \partial\check{#2})}
\newcommand{\symps}[2]{\mathrm{Symp}_{c}^{#1}(#2; \partial #2)}
\newcommand{\sympcomp}[2]{\mathrm{Symp}_c^{#1}(#2)}

\newcommand{\diffeo}[2]{\mathrm{Diff}_c^{#1}(#2)}

\newcommand{\hamc}[2]{\mathrm{Ham}_{c}^{#1}(#2)}
\newcommand{\hams}[2]{\mathrm{Ham}_{c}^{#1}(#2; \partial #2)}

\newcommand{\dilo}[4]{\mathrm{Dilo}_{#1}^{#2}(#3,#4)}
\newcommand{\diro}[4]{\mathrm{Diro}_{#1}^{#2}(#3,#4)}
\newcommand{\emb}[2]{\mathrm{Emb}^{#1}(#2)}

\newcommand{\isum}[2]{\displaystyle\sum_{#1}^{#2}}
\newcommand{\inint}[2]{\displaystyle\int_{#1}^{#2}}
\newcommand{\lconv}[2]{\underset{#1 \rightarrow #2}{\lim}}
\newcommand{\rconv}[2]{\underset{#1 \rightarrow #2}{\longrightarrow}}

\newcommand{\conj}[2]{#1^{-1}\circ #2 \circ #1}
\newcommand{\iconj}[2]{#1\circ #2 \circ #1^{-1}}

\newcommand{\rest}[2]{{
  #1 
  |_{#2} 
  }}

\DeclareMathOperator*{\diam}{diam}
\DeclareMathOperator{\Leb}{Leb}

\DeclareMathOperator*{\id}{id}
\DeclareMathOperator*{\supp}{supp}

\makeglossaries

\theorempreskip{0.2cm}
\theorempostskip{0.2cm}
\theoremstyle{plain}
\theoremseparator{.}
\newtheorem{theorem}{Theorem}[section]
\newtheorem{lemma}[theorem]{Lemma}

\crefname{sublemma}{Sub-lemma}{Sub-lemmas}
\newtheorem{prop}[theorem]{Proposition}
\crefname{prop}{Proposition}{Propositions}
\newtheorem{coro}[theorem]{Corollary}
\newtheorem{defin}[theorem]{Definition}
\newtheorem{remark}[theorem]{Remark}
\newtheorem{fact}[theorem]{Fact}
\newtheorem{notation}[theorem]{Notation}
\newtheorem{example}[theorem]{Example}
\newtheorem*{theorem-no}{Theorem}
\newtheorem*{prop-no}{Proposition}
\crefname{equation}{equation}{equations}
\crefname{fact}{Fact}{Facts}

\theoreminframepreskip{0cm}
\theoreminframepostskip{0cm}
\newtheorem{mainthm}{Theorem}

\renewenvironment{theorem-no}[1][Theorem]{%
\vspace{10pt}
  \par\noindent\textbf{Theorem #1.}\enspace  
  \itshape
}{\par 
\vspace{10pt} \normalfont}

\theoremstyle{nonumberplain}
\theoremheaderfont{\itshape}
\theorembodyfont{\normalfont}
\theoremsymbol{\ensuremath{\square}}
\newtheorem{proof}{Proof}

\numberwithin{equation}{section}

\renewenvironment{proof}[1][Proof]{%
\vspace{0.1cm}
  \par\noindent\textit{#1.}\enspace  
}{%
  \hfill\ensuremath{\square}\par  
  \vspace{0.2cm}
}

\title{Analytic symplectomorphisms displaying minimal ergodicity on the sphere, cylinder and disk}
\author{Yann Delaporte\footnote{IMJ-PRG, CNRS, Sorbonne University, Université Paris Cité, partially supported by the ERC project 818737 Emergence of wild differentiable dynamical systems.}}

\begin{document}

\maketitle
\date{}

\begin{abstract}

We construct analytic symplectomorphisms on the sphere, the disk and the cylinder which are minimally ergodic (only 3 ergodic measures). To achieve this, we apply and generalize a principle introduced by Berger, based on the Approximation by Conjugacy method of Anosov-Katok. 

\end{abstract}

\tableofcontents

\section{Introduction}

The main goal of this article is to produce analytic symplectomorphisms on the sphere, the disk and the cylinder which are \emph{minimally ergodic}: they have a minimal number of ergodic measures. 
The first example of such symplectomorphisms has been obtained by Shnirelman in \cite{shnirelman_example_1930}, this example is a symplectomorphism of the open disk with exactly two ergodic measures, it is further discussed in \cite{fayad_constructions_2004}. 
Later on, \cite{fathi_existence_1977} obtained the existence of smooth minimal symplectomorphism (only one ergodic measure) on manifold endowed with a free circle action. 
In our case, we will reach three ergodic measures on each surface (sphere, cylinder and disk). 
We will justify this number at the beginning of \cref{sec:finerg_cyl}.\\
To obtain such symplectomorphisms, we will base this work on Berger's generalization of the Approximation by Conjugacy (AbC) method of Anosov Katok. 
This method has been introduced in \cite{anosov_new_1970} and gives the existence of smooth, ergodic, and volume-preserving dynamics on a compact manifold endowed with a $\R/\Z$ action, such as the cylinder, the sphere, or the disk. 
The AbC method allows in particular to produce \emph{pseudo-rotations}, which are symplectomorphisms (area and orientation preserving maps) of surfaces with a finite number of periodic points (see \cite{beguin_pseudo-rotations_2006}). 
Note that a minimally ergodic symplectomorphism of the cylinder, the sphere or the disk is a pseudo-rotation as a periodic point gives rise to an ergodic measure.\\

The AbC method has already been used to produce minimally ergodic dynamics. 
In \cite{fayad_constructions_2004}, they produced smooth minimally ergodic dynamics on the disk, the sphere and the cylinder. 
Later on, in \cite{hofer_anosovkatok_2022} they also use the Anosov-Katok method to obtain smooth Hamiltonian pseudo-rotations with a minimal number of ergodic measures on toric symplectic manifolds. 
As pointed out \cite{fayad_constructions_2004}, to obtain minimal ergodicity one needs to control the behaviour of \emph{all} orbits, which is harder than ergodicity where one only needs to control the behaviour of \emph{almost every} orbit. 
Therefore, both articles obtain minimal ergodicity by showing the weak-$\star$ convergence of \emph{every} empirical measures into the convex hull of the desired ergodic measures.\\

However the AbC method does not fit naturally with the analytic setting. 
In \cite[§7.1]{fayad_constructions_2004}, they point out the difficulties to obtain analyticity with the AbC method as the convergence radius shrinks when we compose conjugacies. 
Some analytic examples have been obtained on tori with \enquote{block-slide} maps and Fourier series in \cite{banerjee_real-analytic_2019} or with a straightforward analytic construction in \cite[Theorem 2.1]{furstenberg_strict_1961}. 
Also, Fayad and Katok managed to use Anosov-Katok method to obtain analytic uniquely ergodic pseudo-rotations on the odd sphere in \cite{fayad_analytic_2014}. 
Recently, Berger tackled the case of the cylinder in \cite{berger_analytic_2022} where he works with entire functions to address the issue of the convergence radius. 
As a result, Berger disproved a conjecture of Birkhoff
\cite{birkhoff_unsolved_1941} by showing the existence of analytic pseudo-rotations of the cylinder which are ergodic. 
In addition, on the contrary, he shows the existence of analytic pseudo-rotations which exhibit high local emergence.\\

Later on, to address the question of the sphere and the disk, Berger introduced the \emph{AbC Principle}, which allows to perform analytic constructions with the AbC method on surfaces such as the disk, the cylinder and the sphere \cite{berger_analytic_2024}. 
It allowed him to produce analytic transitive pseudo-rotations of the sphere, the disk and the cylinder. Moreover, as heraised the problem, we could obtain analytic pseudo-rotations which are either ergodic or with high emergence on the sphere and the disk. 
Basically, the Principle states that if a property $(\mathcal{P})$ is realizable by an \emph{AbC scheme} in a $C^r$-topology, for $0 \leq r \leq \infty$, then there exists an analytic symplectomorphism satisfying $(\mathcal{P})$. 
See \cref{def:abc} for the AbC scheme's definition.\\

Where the original principle was enough for ergodicity and emergence since they are properties concerning \emph{almost every} orbits, as it is shown in \cite{delaporte_analytic_2025}, it faces difficulties when it comes to minimal ergodicity which concerns \emph{every} orbits. 
In fact, in Berger's scheme, at every step of the construction, we do not control the behaviour of an infinite amount of orbits. 
For instance, on the cylinder, the AbC principle does not control the dynamics nearby the boundary of the cylinder at each step, while such a neighbourhood contains an infinite amount of $\R/\Z$ orbits. 
This issue comes from the approximation Theorem used to prove the Principle. This approximation Theorem from \cite[Theorem 1.8]{berger_analytic_2022} allows to approach smooth Hamiltonian time-one maps by analytic ones. 
Yet, to obtain the analytic approximation in the theorem, Berger uses Runge's Theorem which brings a no-control area, which is defined as a neighbourhood of the boundary of the cylinder.\\

Then, to address this issue, we generalize the AbC scheme to the \emph{AbC$^\star$ scheme} stated in \cref{sec:star_scheme}. This AbC$^\star$ scheme also comes with an \emph{AbC$^\star$ Principle} (\cref{thm:star_princ}) which states that if a property $(\mathcal{P})$ is realizable by an AbC$^\star$ scheme in a $C^r$-topology, for $0 \leq r \leq \infty$, then there exists an analytic symplectomorphism satisfying $(\mathcal{P})$. This Principle is proved in \cref{sec:proof_princ} by using an approximation Theorem where the no-control area can be any union of two disjoint essential bands (see Subsection \ref{an:approx}), and we use the approximation Theorem \cite[Theorem 1.8]{berger_analytic_2022} to prove this one. Therefore, defining the no-control area by \enquote{winding} annuli allows to consider every orbit in the construction and reach minimal ergodicity, leading to our main theorem.

\begin{mainthm}\label{thm:an_finerg_schem}
There exist minimally ergodic analytic symplectomorphisms on the cylinder, the sphere and the disk.
\end{mainthm}

We obtain this result by proving in \cref{sec:finerg_cyl} that we can realize minimal ergodicity on the sphere, the disk, and the cylinder with an AbC$^\star$ scheme. This proof only involves the $C^0$-topology, yet, the Principle implies the existence of analytic minimally ergodic symplectomorphisms on these surfaces.\\

In this article, we start in \cref{sec:scheme} by presenting AbC schemes which enable us to build symplectomorphisms satisfying desired properties and are defined in a $C^r$-topology, for $0\leq r \leq +\infty$. Then, we introduce the AbC Principle in \cref{thm:abc_princ} which provides analytic symplectomorphisms satisfying a property realized by an AbC scheme. There is hidden behind the proof of this principle an approximation theorem whose trade-off is a no-control area due to the use of Runge's theorem in the approximation. In the AbC scheme, these areas are defined by neighbourhoods of the extremal parts of the surfaces (poles on the sphere, boundary on the cylinder, and the boundary and the center on the disk) and prevent us from realizing minimally ergodic symplectomorphisms with an AbC scheme. Therefore, we generalize this scheme as the AbC$^\star$ scheme, where these areas are allowed to be \enquote{winding} annuli on the surface in order to encompass minimal ergodicity. We will then introduce the  AbC$^\star$ Principle in \cref{thm:star_princ}, which gives the existence of analytic symplectomorphisms built by these schemes, therefore satisfying the desired properties. Moreover, we obtain from Proposition \ref{prop:gen_scheme} that the AbC$^\star$ Principle implies the AbC Principle.\\

Then, in Section \cref{sec:finerg_cyl}, we prove the $C^0$-AbC$^\star$ realization of minimal ergodicity. First, we set up the construction  from \cite[§3]{fayad_constructions_2004} of minimal ergodicity into a $C^0$-AbC$^\star$ scheme. Next, we prove that this scheme realizes minimal ergodicity in \cref{sec:finerg_scheme}. Where Fayad and Katok used the rough definition of the weak-$\star$ convergence to obtain minimal ergodicity, we use the Kantorovich distance which induce this convergence and allows us to quantify how far a symplectomorphism is from being minimally ergodic. Therefore, we obtain the $C^0$-AbC$^\star$ realization of minimal ergodicity on the sphere, the cylinder, and the disk and we can use the AbC$^\star$ Principle to obtain \cref{thm:an_finerg_schem}.\\

Finally, in Section \ref{sec:proof_princ}, we prove the AbC$^\star$ Principle by adapting the proof of Berger's AbC Principle. We first obtain an approximation Theorem in \cref{an:thm_main_approx} to approach analytically symplectomorphisms of the interior of the cylinder. After introducing the notion of complexification and complex and symplectic structures in \cref{sec:complex}, we use this previous theorem to obtain in \cref{sec:approx_mod} the approximation modulo deformation \cref{an:thm_main_approx}. It allows to approach a symplectomorphism by another one which slightly deforms the complex and symplectic structure of a complexification of the cylinder. We finish the section by plugging this approximation result into the AbC$^\star$ scheme on the cylinder, where the structure is simpler, then on the sphere and the disk. Therefore, since at every step of the scheme we just slightly deform our complex and symplectic structure, we obtain at the limit a complex and symplectic structure which remain invariant under the limit map. In particular, this map built by the scheme is an analytic symplectomorphism for these limit structures, which allows to conclude the proof of the Principle.\\

\textbf{Acknowledgement.} {\it I am very thankful to Pierre Berger who introduced me to the Anosov Katok method through several papers, especially his own about the AbC Principle which lead to its generalization in this article. I also thanks him for the supervision on this work and the several re-readings of this paper. I am grateful to Raphaël Krikorian for his encouragements on my work.}

\begin{center}
\textbf{An index of notations can be found page \pageref{index}.}
\end{center}
\vspace{-0.5cm}
\addcontentsline{toc}{section}{\bfseries{Notations}}
\section*{Notations}

\paragraph{Surfaces and spaces}\ \\

In this paper we will work on the following surfaces:

\newglossaryentry{01A}{sort = {1Space}, name = {$\A$}, description = {Cylinder}}
\newglossaryentry{01Di}{sort = {1Space}, name = {$\Di$}, description = {Disk}}
\newglossaryentry{01Sp}{sort = {1Space}, name = {$\Sp$}, description = {Sphere}}
\begin{enumerate}
\item \gls{01A} $:= \T \times \I$ the cylinder, where $\T$ is the torus $\R / \Z$ and $\I = [-1;1]$.
\item \gls{01Di} $:= \lbrace x \in \R^2 ; x_1^2 + x_2^2 \leq 1 \rbrace$ the unit closed disk.
\item\gls{01Sp}$:= \lbrace x \in \R^3, x_1^2 +x_2^2 +x_3^2 =1 \rbrace$ the sphere. 
\end{enumerate}

The circle $\T$ acts canonically by rotation on the disk and the cylinder, it also acts by rotation of axis $x_3$ on the sphere. We denote by $R_\alpha$ such a rotation for $\alpha \in \T$.\\

\newglossaryentry{02M}{sort = {1Space}, name = {$\M$}, description = {An element of $\croset{\A,\Sp,\Di}$}}
We consider \gls{02M}$\in \lbrace \A; \Sp; \Di \rbrace$ endowed with its canonical symplectic form $\Omega$. The form $\Omega$ defines a probability measure on the surface denoted $\Leb$. For $1 \leq r \leq +\infty$, we denote $\symp{r}{\M}$ the space of $C^r$-symplectomorphisms endowed with the $C^r$ topology. For $r=0$, $\symp{0}{\M}$ denotes the space of symplectic homeomorphisms, defined as the closure of $\symp{1}{\M}$ for the $C^0$ topology; in dimension 2 it coincides with the space of area-preserving homeomorphisms \cite[Theorem I]{oh_c0-coerciveness_2006}. Let $\symp{\omega}{\M}$ denote the set of analytic symplectomorphisms of $\M$.\\
When $0 \leq r \leq +\infty$, we denote by $d_{C^r}$ a distance inducing the $C^r$ topology.

We will work on these symplectic surfaces:

\newglossaryentry{03A_c}{sort = {1Space}, name = {$\check{\A}$}, description = {Cylinder without the boundary}}
\newglossaryentry{03Di_c}{sort = {1Space}, name = {$\check{\Di}$}, description = {Disk without the boundary and the center}}
\newglossaryentry{03Sp_c}{sort = {1Space}, name = {$\check{\Sp}$}, description = {Sphere without the North and South poles}}
\begin{enumerate}
\item \gls{03A_c}$:= \T \times (-1;1)$,
\item \gls{03Di_c}$:= \croset{x \in \Di \ : \ 0< x^2_1 + x_2^2  <1}$,
\item \gls{03Sp_c}$:= \croset{x \in \Sp \ : \ x \neq (0,0,\pm 1)}$. 
\end{enumerate}

We may observe that these surfaces are symplectomorphic via

\begin{equation}\label{eq:pi}
\begin{array}{rrcl}
\pi : &\check{\A} &\rightarrow &\check{\M}\\
&(\theta,y) &\mapsto &\left\lbrace
\begin{array}{cl}
(\theta,y) &\text{if } \M = \A\\
\left( \sqrt{1-y^2}\cos(2\pi \theta) , \sqrt{1-y^2}\sin (2\pi \theta ),y \right) &\text{if } \M = \Sp\\
\sqrt{\frac{1+y}{2}} \left( \cos (2\pi \theta) , \sin(2\pi\theta) \right) &\text{if } \M = \Di
\end{array}\right.
\end{array}.
\end{equation}

\newglossaryentry{04pi}{sort = {1Space}, name = {$\pi$}, description = {Surjective coninuous map from $\A$ to $\M$}}

We use the same notation $\pi$ for the three surfaces $\M \in \croset{\A; \Sp ; \Di}$. Note that \gls{04pi} might be extended to a surjective continuous map from $\A$ to $\M$.\\

\newglossaryentry{05M_c}{sort = {1Space}, name = {$\check{\M}$}, description = {An element of $\croset{\check{\A},\check{\Sp},\check{\Di}}$}}
\newglossaryentry{1Symp_c}{sort = {3Func Space}, name = {$\sympc{r}{\M}$}, description = {$C^r$ symplectomorphisms compactly supported in $\check{\M}$}}
\newglossaryentry{1Symp_o}{sort = {3Func Space}, name = {$\sympo{r}{\M}$}, description = {$C^r$ symplectomorphisms equal to the identity on $\partial \check{\M}$}}

Then, for $r \geq 0$, we consider the following subspaces of $\symp{r}{\M}$:
\begin{itemize}
\item the subspace \glsadd{1Symp_c}formed by the elements compactly supported in \gls{05M_c}:
$$\sympc{r}{\M} = \croset{f \in \symp{r}{\M} \, : \, \overline{\mathrm{supp}(f)} \subset \check{\M}},$$
\item the subspace \glsadd{1Symp_o}formed by elements equal to the identity on $\partial \check{\M} = \M \setminus \check{\M}$:
$$\sympo{r}{\M} =  \croset{f \in \symp{r}{\M} \, : \, \rest{f}{\partial \check{\M}} = \id }.$$
\end{itemize}

\newglossaryentry{06A_eta}{sort = {1Space}, name = {$\A(\eta)$}, description = {Sub-open-cylinder of $\A$}}
\newglossaryentry{06M_eta}{sort = {1Space}, name = {$\M(\eta)$}, description = {Image of $\A(\eta)$ by $\pi$}}
\glsadd{06A_eta}
\glsadd{06M_eta}

We consider for $\eta \in [0;1)$ the surfaces
$$\A(\eta) = \T \times \I(\eta) \; \text{ and } \; \M(\eta) = \pi(\A(\eta)),$$
where $\I(\eta) = (-1+\eta ; 1-\eta)$. In particular $\check{\M} = \M(0)$, and for $f \in \sympc{r}{\M}$ there exists $\eta >0$ such that the support of $f$ is included in $\M(\eta)$.

\paragraph{Kantorovich-Wasserstein distance}\ \\

\newglossaryentry{1M(M)}{sort = {3Measure}, name = {$\mathcal{M}(\M)$}, description = {Probability measures on $\M$}}
\newglossaryentry{2d_K}{sort = {3Measure}, name = {$d_K$}, description = {Kantorovich-Wassertein distance on $\mathcal{M}(\M)$}}
We will work with the space of probability measures on $\M$, denoted by \gls{1M(M)}, endowed with the Kantorovich-Wassertein distance: \glsadd{2d_K}
$$\forall \mu_1, \mu_2 \in \mathcal{M}(\M), \; d_K(\mu_1,\mu_2) := \inf \croset{\inint{\M^2}{} d(x_1,x_2) d\mu \ : \ \mu \in \mathcal{M}(\M^2) \text{ s.t. } {p_i}_*\mu = \mu_i \, \forall i \in \croset{1;2}}$$
 
where for $i \in \croset{1;2}$, $p_i : (x_1,x_2) \in \M^2 \mapsto x_i \in \M$ is the canonical projection. This distance induces the weak $\star$-topology on $\mathcal{M}(\M)$, which is compact. Moreover, by the duality theorem of Kantorovich and Rubinstein, we have the dual formulation:
$$d_K(\mu_1,\mu_2) = \sup \croset{\inint{\M}{}f d(\mu_1 - \mu_2) \; : \; f: \M \rightarrow \R \text{ 1-Lipschitz}}.$$
\newglossaryentry{2e^f}{sort = {3Measure}, name = {$e^f_n$}, description = {Empirical measures}}
\glsadd{2e^f}
Given a symplectomorphism $f$ of $\M$, we consider the empirical measures for $x$ in $\M$ and $n \geq 1$:
$$e^f_n (x) := \frac{1}{n} \isum{k=1}{n}\delta_{f^k(x)}.$$

\section{Schemes and Principles}\label{sec:scheme}

\subsection{AbC Principle}\label{sec:abc_scheme}

In this section we present the AbC principle introduced by P. Berger in \cite{berger_analytic_2024}.\\

\newglossaryentry{2T^r0}{sort = {3Func Space}, name = {$\mathcal{T}^r$}, description = {$C^r$ topology of the convergence on compact subsets of $\check{\M}$}}
First we define what an AbC scheme is, where AbC stands for Approximation by Conjugacy. Let us denote by \gls{2T^r0} the pull-back by the restriction $\rest{\cdot}{\check{\M}} : \sympo{r}{\M} \mapsto \symp{r}{\check{\M}}$ of the compact-open $C^r$-topology on $\symp{r}{\check{\M}}$, for $\M \in \croset{\A ; \Sp; \Di}$ and $0\leq r \leq +\infty$. A basis of neighbourhoods of $f \in \sympo{r}{\M}$ is:
$$\croset{g \in \sympo{r}{\M} \, : \, d_{C^s} (\rest{f}{\M(\eta)}, \rest{g}{\M(\eta)}) < \epsilon },$$
taken among $\eta,\epsilon >0$ and finite $s \leq r$.

\begin{defin}[AbC scheme]\label{def:abc}
A \underline{$C^r$-AbC scheme} is a map:
$$(h,\alpha) \in \sympo{r}{\M} \times \Q/\Z \mapsto \left (U(h,\alpha), \nu(h,\alpha)\right ) \in \mathcal{T}^r \times (0,\infty)$$
such that:
\begin{enumerate}[label = \alph*)]
\item Each open set $U(h,\alpha)$ contains a map $\hat{h}\in \sympo{r}{\M}$  satisfying
$$ \iconj{\hat{h}}{R_\alpha} = \iconj{h}{R_\alpha}.$$

\item  Given any sequences $(h_n)_n \in \sympo{r}{\M}^\N$ and $(\alpha_n)_n \in \left (\Q / \Z \right )^\N$ satisfying :
$$h_0 = id \, , \;\; \alpha_0 = 0,$$
and for $n \geq 0$:
$$ \iconj{h_{n+1}}{R_{\alpha_n}} = \iconj{h_n}{R_{\alpha_n}},$$
$$h_{n+1} \in U(h_n,\alpha_n ) \, \text{ and } \, 0< \lvert \alpha_{n+1} - \alpha_n \rvert < \nu(h_{n+1}, \alpha_n )$$
then the sequence $f_n := \iconj{h_n}{R_{\alpha_n}}$ converges to a map $f$ in $\symp{r}{\M}$.
\end{enumerate}
The map $f$ is said \underline{constructed from the $C^r$-scheme} $(U,\nu)$.
\end{defin}

The aim of such AbC schemes is to construct maps which satisfy a property $(\mathcal{P})$:

\begin{defin}
A property ($\mathcal{P}$) on $\symp{r}{\M}$ is an \underline{AbC realizable $C^r$-property} if:
\begin{enumerate}
\item The property ($\mathcal{P}$) is invariant by $\symp{r}{\M}$-conjugacy: if $f \in \symp{r}{\M}$ satisfies ($\mathcal{P}$) then so does $\iconj{h}{f}$ for any $h \in \symp{r}{\M}$.
\item There exists a $C^r$-AbC scheme such that any map $f\in \symp{r}{\M}$ constructed from it satisfies \nolinebreak($\mathcal{P}$).
\end{enumerate}
\end{defin}

For instance, Berger proved that transitivity is an AbC $C^0$-realizable property. It also has been shown that ergodicity or high emergence are realizable \cite{delaporte_analytic_2025}.\\

The use of open sets in the scheme allows to use density results, so that a realizable property is also realizable for stronger regularity as stated in the two next results.

\begin{prop}[\textnormal{\cite[Prop 2.6]{berger_analytic_2024}}]
For $+\infty \geq s \geq r \geq 0$, if $(\mathcal{P})$ is an AbC $C^r$-realizable property then it is an AbC $C^s$-realizable property.
\end{prop}

In addition of this regularization result, Pierre Berger obtained in \cite[Theorem D]{berger_analytic_2024} that any realizable property can be achieved analytically, as stated in his following main theorem.

\begin{theorem}[AbC principle]\label{thm:abc_princ}
For every $0 \leq r \leq \infty$, any $AbC$ realizable $C^r$-property $(\mathcal{P})$ is satisfied by a certain $f \in \symp{\omega}{\M}$.
\end{theorem}

Where this principle seems suitable for properties defined almost everywhere, it appears to face difficulties when it comes to property defined on every orbit such as minimal ergodicity. Therefore, answering the problem posed in \cite[Problem 2.10]{berger_analytic_2024}, we generalize this principle to the AbC$^\star$ principle in the next section to encompass minimal ergodicity. In particular, we will obtain with Proposition \ref{prop:gen_scheme} that the AbC$^\star$ principle implies the AbC principle.\\

\subsection{AbC$^\star$ Principle}\label{sec:star_scheme}

In order to generalize the above AbC principle, we first have to introduce the notion of bicurves on the surface $\M \in \croset{\Sp, \Di, \A}$. In fact, when we work with the conjugacy of a rotation $\iconj{h}{R_\alpha}$, its orbits are the images by $h$ of the orbits of $R_\alpha$. Therefore to control every orbit of the conjugacy, $h$ must be controlled on every $\R/\Z$-orbit of $\M$. However, in the previous scheme, the topology used was the one on $\check{\M}$ and so $h$ was not controlled on a neighbourhood of $\M \setminus \check{\M}$ which contains an infinite amount of $\R/\Z$ orbits. To overcome this difficulty, we introduce the notion of bicurves to work with winding bands built such that, for any $\R/\Z$ orbit, the measure of its intersection with a band uniformly goes to zero as the band is getting tighter.

\begin{defin}[Bicurve]\label{def:bicurve}
A \underline{bicurve} on the cylinder is the union $\gamma$ of two disjointed loops $\gamma_+$ and $\gamma_-$ embedded in $\A$ such that there exists a diffeotopy $(f_t)_t$ in $\diffeo{\infty}{\A}$ from $f_0 = \id$ to $f_1 \in \diffeo{\infty}{\A}$ which sends a pair of circles of the form $\T \times \croset{y_-, y_+}$ to $\gamma$, with $-1 \leq y_- < y_+ \leq 1$: 
$$ f_1(\T \times \croset{y_+} ) = \gamma_+ \quad \text{and} \quad  f_1(\T \times \croset{y_-} ) = \gamma_-.$$

We say that the bicurve is \underline{diffeotopic} to the straight bicurve $\T \times \croset{y_-,y_+} $.\\

We denote $\Gamma_2(\A)$ the set of bicurves on $\A$.\\
Then the set $\Gamma_2(\M)$ of bicurves on $\M$ is formed by the images by $\pi$ of those on $\A$ :
$$\Gamma_2(\M) = \croset{\pi (\gamma) \, : \, \gamma\in \Gamma_2(\A)}.$$
\end{defin}
\newglossaryentry{1gam}{sort = {2bicurve}, name = {$\gamma$}, description = {Bicurves}}
\newglossaryentry{2gam+-}{sort = {2bicurve}, name = {$\gamma_+$, $\gamma_-$}, description = {Upper and lower parts of a bicurve}}
\newglossaryentry{3Gam}{sort = {2bicurve}, name = {$\Gamma_2(\M)$}, description = {Set of bicurves in $\M$}}
\glsadd{1gam}\glsadd{2gam+-}\glsadd{3Gam}

\begin{remark}
In the case $\M = \Sp$, due to the non injectivity of $\pi$, a bicurve might be a pair of points: $\pi(\T \times \croset{-1;1}) = \croset{N,S}$.
\end{remark}

With this definition of bicurves we introduce the different interesting parts of $\M$ for a given bicurve $\gamma$. (see Figure \ref{fig:bicurve})
\newglossaryentry{4B(gam,kap)}{sort = {2bicurve}, name = {$B(\gamma,\kappa)$}, description = {$\kappa$-neighbourhood in $\M$ of a bicurve $\gamma$}}
\glsadd{4B(gam,kap)}
\newglossaryentry{5A(gam)}{sort = {2bicurve}, name = {$A(\gamma)$}, description = {Component of $\M\setminus \gamma$ between $\gamma_-$ and $\gamma_+$}}
\glsadd{5A(gam)}
\newglossaryentry{5A(gam,kappa)}{sort = {2bicurve}, name = {$A(\gamma,\kappa)$}, description = {Points of $A(\gamma)$ which are $\kappa$-distant from $\gamma$}}
\glsadd{5A(gam,kappa)}
\newglossaryentry{5O(gam)}{sort = {2bicurve}, name = {$O(\gamma)$}, description = {Components of $\M\setminus \gamma$ above $\gamma_+$ and below $\gamma_-$}}
\glsadd{5O(gam)}
\newglossaryentry{5O(gam,kappa)}{sort = {2bicurve}, name = {$O(\gamma,\kappa)$}, description = {Points of $O(\gamma)$ which are $\kappa$-distant from $\gamma$}}
\glsadd{5O(gam,kappa)}

\begin{notation}\label{not:bic}\ 
\begin{itemize}
\item The band $B(\gamma,\kappa) := \croset{x \in \M \, : \, d(x,\gamma) < \kappa}$ for any $\kappa >0$.
\item $A(\gamma)$ is the component of $\M \setminus \gamma$  whose boundary is $\gamma$, and $A(\gamma,\kappa) = A(\gamma) \setminus B(\gamma,\kappa)$.
\item $O(\gamma) := \M \setminus (\gamma \cup A(\gamma))$ and $O(\gamma,\kappa) = O(\gamma) \setminus B(\gamma,\kappa)$.
\end{itemize}
\end{notation}

\begin{figure}
\centering
\includegraphics[scale=0.7]{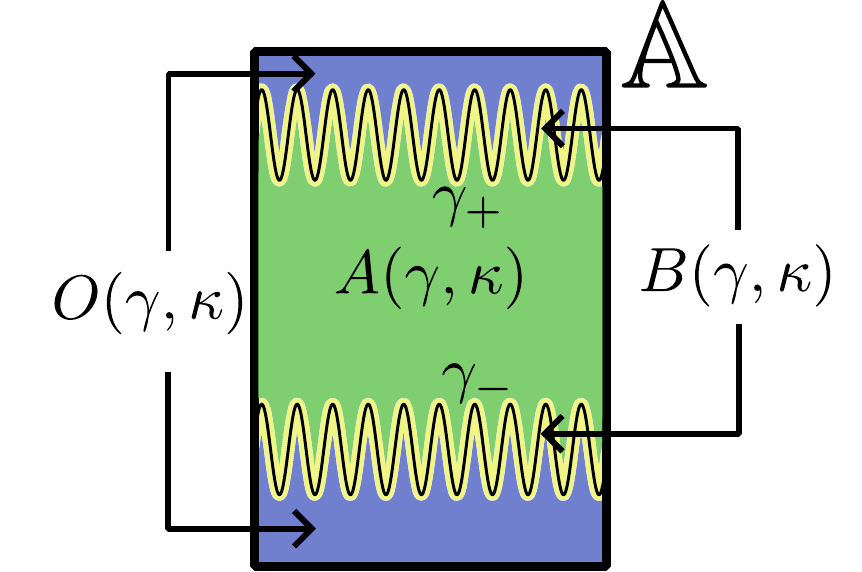}
\caption{Bicurve on the cylinder}
\label{fig:bicurve}
\end{figure}

\newglossaryentry{2T^r1_gam}{sort = {3Func Space}, name = {$\mathcal{T}_\gamma^r$}, description = {$C^r$ topology of the convergence on compact subsets of $\M\setminus\gamma$}}
\glsadd{2T^r1_gam}
\newglossaryentry{2T_r2Gam}{sort = {3Func Space}, name = {$\mathcal{T}_\Gamma^r$}, description = {Union of the $\mathcal{T}_\gamma^r$ over the bicurves of $\M$}}
\glsadd{2T_r2Gam}
Then to define our scheme, we define for $0 \leq r \leq +\infty$ a set $\mathcal{T}_\Gamma^r$ formed by open subsets of $\sympo{r}{\M}$. The set $\mathcal{T}_\Gamma^r$ is the union of the sets $\mathcal{T}^r_\gamma$ among $\gamma \in \Gamma_2(\M)$, where each $\mathcal{T}^r_\gamma$ is the topology of $\sympo{r}{\M}$ whose basis of open neighbourhoods of $f\in \sympo{r}{\M}$ is given by:
$$\croset{ g \in \sympo{r}{\M}, \, d_{C^s}( \rest{g}{\M \setminus B(\gamma,\kappa)},\rest{f}{\M \setminus B(\gamma,\kappa)} ) < \epsilon}$$
taken among $\epsilon >0$, finite $s \leq r$ and $\kappa > 0$.\\

\begin{remark}
Observe that the topology generated by the union of $\mathcal{T}_\gamma^r$ is the canonical $C^r$-topology of $\sympo{r}{\M}$, since taking the intersection of the previous open sets for two disjointed bicurves $\gamma$ and $\gamma'$, and $\kappa$, $\kappa'$ small enough gives the usual $\epsilon$-neighbourhood of $f$ in $\sympo{r}{\M}$. Hence, contrary to $\mathcal{T}^r$, the set $\mathcal{T}_\Gamma^r$ is not a topology.
\end{remark}

\begin{remark}
Taking $\gamma_0 := \pi ( \T \times \croset{-1;1} )$ gives us $\M \setminus \gamma_0 = \check{\M}$. i.e. $\mathcal{T}^r_{\gamma_0}$ is the topology $\mathcal{T}^r$ from the latter AbC scheme. Thus we can obtain the generalization of the AbC scheme by the in coming AbC$^\star$ scheme.
\end{remark}

Then let us introduce the AbC$^\star$ scheme.

\begin{defin}[AbC$^\star$ scheme]\label{def:abc_star}
A \underline{$C^r$-AbC$^\star$ scheme} is a map:
$$(h,\alpha) \in \sympo{r}{\M} \times \Q/\Z \mapsto \left (U(h,\alpha), \nu(h,\alpha)\right ) \in \mathcal{T}_\Gamma^r \times (0,\infty)$$
such that:
\begin{enumerate}[label = \alph*)]
\item Each open set $U(h,\alpha)$ contains a map $\hat{h}$ satisfying
$$ \iconj{\hat{h}}{R_\alpha} = \iconj{h}{R_\alpha}.$$

\item There exists $\gamma \in \Gamma_2(\M)$ invariant under $R_\alpha$ such that $U(h,\alpha) \in \mathcal{T}^r_\gamma$  and $\rest{\hat{h}}{O(\gamma)} = \rest{h}{O(\gamma)}$.

\item  Given any sequences $(h_n)_n \in \sympo{r}{\M}^\N$ and $(\alpha_n)_n \in \left (\Q / \Z \right )^\N$ satisfying :
$$h_0 = \id \, , \;\; \alpha_0 = 0,$$
and for $n \geq 0$:
$$ \iconj{h_{n+1}}{R_{\alpha_n}} = \iconj{h_n}{R_{\alpha_n}},$$
$$h_{n+1} \in U(h_n,\alpha_n ) \, \text{ and } \, 0< \lvert \alpha_{n+1} - \alpha_n \rvert < \nu(h_{n+1}, \alpha_n )$$
then the sequence $f_n := \iconj{h_n}{R_{\alpha_n}}$ converges to a map $f$ in $\symp{r}{\M}$.
\end{enumerate}
The map $f$ is said \underline{constructed from the $C^r$-scheme} $(U,\nu)$.
\end{defin}

Since this scheme aims to construct maps which satisfies a property $(\mathcal{P})$. We also define the following.

\begin{defin}[AbC$^\star$ realizable]
A property ($\mathcal{P}$) on $\symp{r}{\M}$ is an \underline{AbC$^\star$ realizable $C^r$-property} if:
\begin{enumerate}
\item The property ($\mathcal{P}$) is invariant by $\symp{r}{\M}$-conjugacy: if $f \in \symp{r}{\M}$ satisfies ($\mathcal{P}$) then so does $\iconj{h}{f}$ for any $h \in \symp{r}{\M}$.
\item There exists a $C^r$-AbC$^\star$ scheme such that every map $f\in \symp{r}{\M}$ constructed from it satisfies ($\mathcal{P}$).
\end{enumerate}
\end{defin}

In order to show \cref{thm:an_finerg_schem}, we will prove in \cref{sec:finerg_cyl} the following.

\begin{prop}\label{prop:real_minerg}
Being minimally ergodic is an AbC$^\star$ realizable $C^0$-property.
\end{prop}

Then, as precise above, we can generalize the AbC scheme by this one.

\begin{prop}\label{prop:gen_scheme}
A $C^r$-AbC scheme $(U,\nu)$ is a $C^r$-AbC$^\star$ scheme, in particular if a property $(\mathcal{P})$ is an AbC realizable $C^r$-property, then it is an AbC$^\star$ realizable $C^r$-property.
\end{prop}

\begin{proof}
Let $(U,\nu)$ be a $C^r$-AbC scheme. Observe that the topology $\mathcal{T}^r$ is the topology $\mathcal{T}^r_{\gamma_0}$, with $\gamma_0 = \pi (\partial \A)$. Then conditions a) and c) of \cref{def:abc_star} are conditions a) and b) of \cref{def:abc}. Moreover, condition b) from \cref{def:abc_star} is satisfied because $\gamma_0$ is invariant under every rotations and the set $O(\gamma_0)$ is empty. Therefore $(U,\nu)$ defines an AbC$^\star$ scheme.
\end{proof}

The use of open sets in the scheme allows us to use density results, so that a realizable property is also realizable for stronger regularity as stated below.

\begin{prop}\label{prop:ext_real}
For $+\infty \geq s > r \geq 0$, if $(\mathcal{P})$ is an AbC$^\star$ $C^r$-realizable property then it is an AbC$^\star$ $C^s$-realizable property.
\end{prop}
We prove this proposition below likewise Berger did in \cite[Proposition 2.6]{berger_analytic_2024}.\\

Finally we can generalize the previous AbC Principle into the AbC$^\star$ principle to bring analyticity.

\begin{theorem}[AbC$^\star$ Principle]\label{thm:star_princ}
For every $0 \leq r \leq \infty$, any AbC$^\star$ realizable $C^r$-property is satisfied by a map $f$ in $\symp{\omega}{\M}$.
\end{theorem}

To obtain this principle, we will improve an approximation theorem in \cref{an:approx} to encompass the bicurves. Then it remained to adapt Berger's work to this theorem and obtain our proof in the rest of the \cref{sec:proof_princ}.\\

As an application of the AbC$^\star$ Principle, we have:

\begin{proof}[Proof of Theorem \ref{thm:an_finerg_schem}]
It is a direct consequence of the AbC$^\star$ Principle and the realization of minimal ergodicity given by \cref{prop:real_minerg}.

\end{proof}

\begin{proof}[Proof of \cref{prop:ext_real}]
Let $(U,\nu)$ be a $C^r$-AbC$^\star$ scheme realizing property $(\mathcal{P})$. Let us define a $C^s$-AbC$^\star$ scheme $(\widetilde{U},\widetilde{\nu})$ which also realizes property $(\mathcal{P})$.\\
Since the space $\symp{s}{\M}$ is dense in $\symp{r}{\M}$ by \cite{zehnder_note_1977} and \cite{oh_c0-coerciveness_2006}, we define the map $\widetilde{U}$ as:
$$ \widetilde{U} : (h,\alpha) \in \sympo{s}{\M} \times \Q/\Z \mapsto U(h,\alpha) \cap \sympo{s}{\M} \in \mathcal{T}_\Gamma^s.$$
First, let us show that $\widetilde{U}$ satisfies a) and b) from \cref{def:abc_star}.\\

Let $h$ be in $\sympo{s}{\M}$ and $\alpha = \frac{p}{q} \in \Q/\Z$, with $p\wedge q = 1$. By \cref{def:abc_star} a) and b) for the AbC$^\star$ scheme $(U,\nu)$, there exists a bicurve $\gamma \in \Gamma_2(\M)$ invariant under $R_\alpha$ and $\hat{h}_0 \in U(h,\alpha)$ such that:
\begin{enumerate}[label = \arabic*)]
\item $\hat{h}_0$ satisfies $ \iconj{\hat{h}_0}{R_\alpha} = \iconj{h}{R_\alpha}$,
\item $U(h,\alpha)\in \mathcal{T}^r_\gamma$ and $\hat{h}_0$ coincides with $h$ on $O(\gamma)$.
\end{enumerate}

Note that 1) implies that the symplectomorphism $h^{-1} \circ \hat{h}_0$ commutes with $R_\alpha$ and by 2) it leaves invariant $A(\gamma) := \M \setminus (O(\gamma) \cup \gamma)$, for it coincides with the identity on $\partial O(\gamma)\cup (\M \setminus \check{\M})$ wich contains $\gamma$.

We can observe that, by definition of the topology $\mathcal{T}^s_\gamma$ and 2), the set $\widetilde{U}(h,\alpha) := U(h,\alpha)\cap \sympo{s}{\M}$ belongs to $\mathcal{T}^s_\gamma$. Therefore we aim to find $\hat{h}_1 \in \widetilde{U}(h,\alpha)$ such that:
\begin{enumerate}[label = \arabic*')]
\item $\hat{h}_1$ satisfies $ \iconj{\hat{h}_1}{R_\alpha} = \iconj{h}{R_\alpha}$,
\item $\hat{h}_1$ coincides with $h$ on $O(\gamma)$.
\end{enumerate}
This would provide a) and b) of \cref{def:abc_star} for $\widetilde{U}$. Let us use the following fact, proved below, to define $\hat{h}_1$.

\begin{fact}\label{fact:dens}
Let $S\subset \M$ be an open cylinder invariant under $R_\alpha$ and $h_0 \in \symp{r}{S}$ commuting with $R_\alpha$, with ${0 \leq r \leq +\infty}$. Then for any neighbourhood $\mathcal U$ of $h_0$ in the $C^r$ compact-open topology of $\mathrm{Symp}^r(S,\M)$, there exists $h_1 \in \sympcomp{\infty}{S} \cap \mathcal U$ commuting with $R_\alpha$.
\end{fact}

We obtain later with \cref{an:lemma_symp_bi} that $A(\gamma)$ is symplectomorphic, up to composition with $\pi$, to an open cylinder $\T \times (y_-,y_+)$. Therefore, we shall apply this fact with the open cylinder $S := A(\gamma)$ (which is $R_\alpha$-invariant likewise $\gamma$), $h_0 := \rest{h^{-1} \circ \hat{h}_0}{S}$ and $\mathcal U := \croset{\rest{(h^{-1}\circ g)}{S} \, , \, g \in U(h,\alpha)}$ ($\mathcal U$ is indeed a neighbourhood of $h_0$ in the $C^r$ compact-open topology of $\mathrm{Symp}^r(S,\M)$ by definition of the topology $\mathcal T^r_\gamma$). \cref{fact:dens} yields $h_1$ in $\sympcomp{\infty}{S}\cap \mathcal U$ commuting with $R_\alpha$. 

As $h_1$ is compactly supported in $S$, we extend it canonically as a symplectomorphism $\tilde{h}_1$ in $\sympo{\infty}{\M}$ by the identity on $\M \setminus S$. Note that $\tilde{h}_1$ also commutes with $R_\alpha$. Let $\hat{h}_1 := h\circ \tilde{h}_1$. By construction of $\mathcal U$, the map $\hat{h}_1$ belongs to $U(h,\alpha)$ and so:
$$\hat{h}_1  \in U(h,\alpha)\cap \sympo{s}{\M} = \widetilde{U}(h,\alpha).$$
Moreover, $\hat{h}_1$ satisfies 1') and 2'). Therefore the map $\widetilde{U}$ satisfies a) and b) from \cref{def:abc_star}.\\

Finally, to obtain c), we can notice that $\sympo{s}{\M}$ is a Fréchet space for a distance $d$. Then it suffices to reduce $\nu(h,\alpha)$ and define a new map $\widetilde{\nu}$ to ensure that, for every $\widetilde{\alpha}= \frac{\widetilde{p}}{\widetilde{q}}$ with $\widetilde{p}\wedge \widetilde{q} = 1$ such that $0 < \lvert \widetilde{\alpha} - \alpha \rvert < \widetilde{\nu}(h,\alpha)$, we have $\widetilde{q} > 2q$ and
$$d(\iconj{h}{R_\alpha},\iconj{h}{R_{\widetilde{\alpha}}}) \leq \frac{1}{q}.$$
Then any sequence $(h_n,\alpha_n)_n$ satisfying the scheme $(\widetilde{U},\widetilde{\nu})$ constructs a map $f\in \symp{s}{\M}$ which is the limit of the Cauchy sequence $(\iconj{h_n}{R_{\alpha_n}})_n$ for the distance $d$. Moreover, since this scheme is included in the scheme $(U,\nu)$ ($\widetilde{U} \subset U$ and $\widetilde{\nu} \leq \nu$), then $f$ is also constructed from the $C^r$-scheme $(U,\nu)$, therefore it satisfies $(\mathcal{P})$. 
\end{proof}

\begin{proof}[Proof of \cref{fact:dens}]
Let $S\subset \M$ be an open cylinder invariant under $R_\alpha$, $h_0 \in \symp{r}{S}$ commuting with $R_\alpha$ and let $\mathcal U$ be a neighbourhood of $h_0$ in the $C^r$ compact-open topology of $\mathrm{Symp}^r(S,\M)$. 
First, observe that as $h_0$ induces a symplectomorphism on the open cylinder $S/R_\alpha$, by building the symplectomorphism on this surface and then pull it back to $S$, we can assume that $\alpha =0$. Next, by \cite{zehnder_note_1977} and \cite{oh_c0-coerciveness_2006}, the space $\symp{\infty}{S}$ is dense is $\symp{r}{S}$. 
Moreover, as $S$ is symplectomorphic to the open cylinder $\check{\A}$, the density of $\hamc{\infty}{\check{\A}}$ in $\symp{\infty}{\check{\A}}$ put together with the previous density result provides for any $\eta >0$, $t\leq r$ finite and any compact subset $K$ of $S$ the existence of $h_1 \in \sympcomp{\infty}{S}$ which is $\eta$ $C^t$-close to $h$ on $K$. Then we choose $\eta$, $t$ and $K$ such that $h_1$ belongs to the neighbourhood $\mathcal U$.

\end{proof}

\section{Minimal ergodicity}\label{sec:finerg_cyl}

In this section, we aim to prove \cref{prop:real_minerg} which states that being minimally ergodic is an AbC$^\star$-realizable $C^0$ property.\\

We first recall the following definition of minimal ergodicity.

\begin{defin}
Let $f$ be a symplectomorphism of $\M \in \croset{\A ; \Sp ; \Di}$. It is said to be \underline{minimally ergodic} if $f$ has exactly $3$ ergodic measures.
\end{defin}

Here, $3$ is in fact the minimal number of ergodic measures on our surfaces. The lower bound on the cylinder is natural since we have at least one ergodic measure on each components of the boundary and one on an ergodic component of the interior. 
For the disk, we have at least one on an ergodic component of the interior and one on the boundary; moreover, as we work with area-preserving diffeomorphism, Brouwer plane translation theorem implies\footnote{As the open disk is homeomorphic to the plane $\R^2$, the homeomorphism $h$ of the open disk is conjugated to an homeomorphism $g$ of $\R^2$. And if $h$, has no fix point, then so does $g$ and by Brouwer plane translation theorem it has domains of translation, which yield wandering domains for $h$. However an area preserving diffeomorphism of the disk does not have wandering domains, hence $h$ has a fix point.} that we have at least a fix point in the interior, which makes our third ergodic measure. 
On the sphere, as every symplectomorphism is homotopic to the identity\footnote{As $\symp{}{\Sp}$ admits $SO(3)$ for strong rotation retract by \cite{smale_diffeomorphisms_1959}, the space is connected and every symplectomorphism is path connected to the identity} we have at least a fix point by Lefschetz fixed-point theorem, this gives one ergodic measure; in addition, the sphere without this point is symplectomorphic to the interior of the disk which carries at least two ergodic measures by what precedes, which makes $3$.\\
Moreover, \cite{fayad_constructions_2004} constructs smooth symplectomorphisms on any $\M \in \croset{\Di, \A, \Sp}$ with exactly $3$ ergodic measures, which justifies this definition.\\

We notice that any realization of a $C^0$-AbC$^\star$ scheme belongs to the following subset $\mathcal{F}$ of $\sympo{0}{\M}$:
$$\mathcal{F} := \overline{\croset{\iconj{h}{R_\alpha} \, : \, h \in \sympo{0}{\M} \, , \, \alpha \in \Q/\Z}}^{C^0}.$$

Observe that every symplectomorphism $f$ in $\mathcal{F}$ leaves invariant the measures $\Leb_\M$ and $\mu_{\pm 1}= \pi_* (\Leb_{\T \times \croset{\pm 1}} )$, with $\pm \in \croset{-,+}$. In particular we have the following inclusion:
\begin{equation}\label{eq:incl_conv}
 Conv(\Leb_\M, \mu_1, \mu_{-1}) \subset \mathcal{M}(f),
\end{equation}
where $\mathcal{M}(f)$ is the set of invariant probability measures of $f$, and $Conv$ stands for the convex hull. Moreover, we can show that this inclusion is an equality if and only if the symplectomorphism is minimally ergodic.

\begin{prop}\label{prop:minerg_inc}
A symplectomorphism $f \in \mathcal{F}$ is minimally ergodic if and only if
$$Conv(\Leb_\M, \mu_1, \mu_{-1}) = \mathcal{M}(f).$$
\end{prop}

\begin{proof}
First, if we have the equality, then by the ergodic decomposition theorem the $3$ ergodic measures of $f$ are the extremal points of $\mathcal{M}(f)$ which are $\Leb_\M$, $\mu_1$, and $\mu_{-1}$. Then $f$ has exactly $3$ ergodic measures and is minimally ergodic. Conversely, let $\nu_{-1}$, $\nu_0$, and $\nu_1$ be the ergodic measures of $f$. By definition of $\mathcal{F}$,
on each component $C$ of $\pi(\T \times \croset{\pm 1})$ which is a circle (otherwise it is a fixed point), the restriction $f|C$ is a rigid rotation. By minimal ergodicity this rotation must be irrational. From this we deduce that $\mu_{- 1}$ and $\mu_{ 1}$ are ergodic. Then we can assume that $\nu_{\pm1} = \mu_{\pm 1}$. Next, by the ergodic decomposition theorem, we have $\Leb_\M\in Conv(\mu_{-1}, \nu_0, \mu_1)$. Yet, as $\Leb_\M (\sup(\mu_{\pm 1})) = 0$, it follows that $\nu_0 = \Leb_\M$. 
We therefore recover by the ergodic decomposition theorem that $Conv(\Leb_\M, \mu_1, \mu_{-1}) = \mathcal{M}(f)$.
\end{proof}
\begin{remark}
We proved that   the ergodic measures of a minimally ergodic $f\in \mathcal{F}$ are $\Leb_\M, \mu_1, \mu_{-1}$.
\end{remark}
Next, by the ergodic decomposition theorem, we can quantify how far the symplectomorphism $f$ is to be minimally ergodic with the following quantity:
\newglossaryentry{3delt_merg}{sort = {3Measure}, name = {$\Delta_{merg}(f)$}, description = {Distance of a symplectomorphism $f$ of $\M$ from being minimally ergodic}}
\glsadd{3delt_merg}
\begin{equation}
\Delta_{merg}(f) := \sup_{\mu \in \mathcal{M}_{erg}(f)}  d_K \left (\mu,Conv(\Leb_\M, \mu_1, \mu_{-1})\right ),
\end{equation}
where $\mathcal{M}_{erg}(f)$ is the set of ergodic measures of $f$.\\

Then, by the latter proposition and the ergodic decomposition theorem, $f\in \mathcal{F}$ is minimally ergodic if and only if $\Delta_{merg}(f) = 0$.\\

Using the AbC$^\star$ scheme, we are going to construct a sequence $f_n = \iconj{h_n}{R_{\alpha_n}}$ which converges to a certain $f$ and so that:
$$\Delta_{merg}(f) = \lconv{n}{\infty}\Delta_{merg}(f_n) = 0.$$

As a matter of fact $f$ is minimally ergodic. Interestingly, $f_n$ is conjugate to a rational rotation, and so has infinity many ergodic invariant measures, but $\Delta_{merg}(f_n)$ is small when $n$ is large.\\

Then, to define a scheme that constructs minimally ergodic symplectomorphisms, we first use the construction of \cite[Theorem 3.3]{fayad_constructions_2004} in \cref{sec:lem_minerg} to build symplectomorphisms $g$ that pushes forward the measures $(\mu_y)_{y\in \I}$ close to the convex hull $Conv(\Leb_\M, \mu_1, \mu_{-1})$, with $\mu_y = \pi_* \Leb_{\T\times\croset{y}}$. Then, in \cref{sec:finerg_scheme}, we define an AbC$^\star$ scheme using these symplectomorphisms $g$. The scheme will ensure that the conjugacies of rational rotations with high denominator are close to being minimally ergodic according to $\Delta_{merg}$.\\

\subsection{Lemma for minimal ergodicity}\label{sec:lem_minerg}

\indent We consider $C := Conv ( \Leb_\M, \mu_1, \mu_{-1} )$.\\
Let us adapt the construction from Fayad and Katok with convex hulls, so that, when we build the scheme, we can use the quantification $\Delta_{merg}$ of the \enquote{distance} to being minimally ergodic. Moreover, in order to use this to define an AbC$^\star$ scheme in the next section, we shall use bicurves in this construction (see \cref{def:bicurve} and \cref{not:bic}).

\begin{lemma}\label{lemma:finerg}
Let $q \in \N^*$ and $\epsilon > 0$, there exist $g \in \sympc{\infty}{\M}$ and $\gamma \in \Gamma_2(\M)$ such that:
\begin{enumerate}[label = \roman*)]
\item $g$ commutes with $R_{\frac{1}{q}}$ and $\gamma$ is invariant under $R_{\frac{1}{q}}$,
\item for every $y \in [-1;1]$, $d_K(g_*\mu_y,C) < \epsilon$,
\item and $\rest{g}{O(\gamma)} = \id$ and $\sup_{y \in \I}(\mu_y(B(\gamma,\kappa))) \rconv{\kappa}{0}  0$.
\end{enumerate}
\end{lemma}

\begin{figure}[h]
\centering
\includegraphics[scale=0.4]{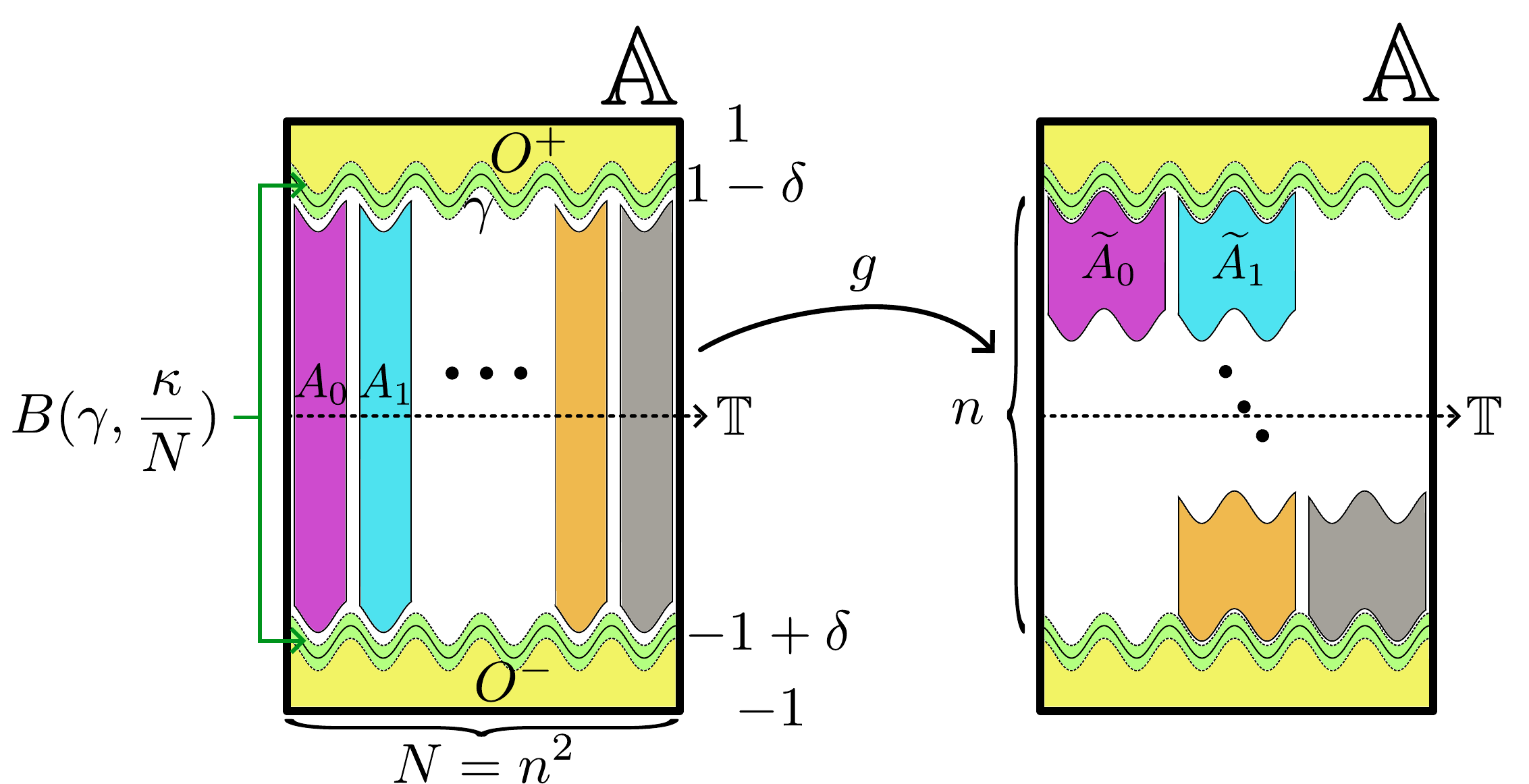}
\caption{Lemma for minimal ergodicity}
\label{fig:finerg}
\end{figure}

\begin{proof}[Sketch of proof]
To obtain i), it suffices to construct such a $g$ and $\gamma$ for $q=1$. Then, using the invariance of $\Leb_\M$, $\mu_1$ and $\mu_{-1}$ under $R_{\frac{1}{q}}$, we generalize to any $q$ up to using a finite covering.\\

Next, to obtain iii) we construct a symplectomorphism $g$ supported in $A(\gamma)$, for a bicurve $\gamma \in \Gamma_2(\M)$. As $g$ is supported in $A(\gamma)$, we obtain the first part of iii) : $\rest{g}{O(\gamma)} = \id$. Then, by defining this bicurve by the graphs of scaled cosines on the cylinder, thus providing a bicurve on any $\M \in \croset{\A, \Di, \Sp}$, every longitude $\pi(\T \times \croset{y})$ intersects the bicurve in a finite set. Hence we recover the second part of property iii) stating that the proportion of a longitude $\pi(\T \times \croset{y})$ which intersects the band $B(\gamma,\kappa)$ of width $\kappa$ tends to $0$ as $\kappa$ tends to $0$, uniformly in $y\in \I$.\\

Finally, to obtain ii), we consider $N$ vertical boxes $(A_l)_{1\leq l \leq N}$ in $A(\gamma)$ spread along the $\T$-axis, and shaped so that every longitude $\pi(\T\times \croset{y})$ has a large proportion either in $O(\gamma)$ or in the boxes $(A_l)_l$ (see \cref{fig:finerg}). Then, we aim to send these boxes by a symplectomorphism into boxes of same volume $(\widetilde{A}_l)_{1 \leq l \leq N}$ which are distributed uniformly in $A(\gamma)$. Such a symplectomorphism $g$ is provided by the following Folkloric theorem (see for instance \cite[§4.1]{berger_analytic_2022}) of symplectic geometry which is a consequence of Moser's trick:
\begin{theorem}\label{thm:folklo_symp}
Let $(D_l)_{1 \leq l\leq N}$ and $(D_l')_{1 \leq l\leq N}$ be two families of disjoint smooth disks in the interior of a symplectic surface $(S,\omega)$ of finite volume, such that for every $l$, $\Leb(D_l) = \Leb(D'_l)$. Then there exists $f \in \sympcomp{\infty}{S}$ such that $f(D_l) = D'_l$ for every $l$.
\end{theorem}

Then, we can split most of each longitude $\pi(\T \times \croset{y})$ into two subsets (one being possibly empty):
\begin{enumerate}
\item The intersection of $\pi(\T \times \croset{y})$ with $A(\gamma)$, whose image by $g$ is close to be uniformly distributed into $A(\gamma)$.
\item The intersection of $\pi(\T \times \croset{y})$ with $O(\gamma)$, on which $g=id$ and which is close to be uniformly spread into $\pi(\T \times \croset{y})$. 
\end{enumerate}
By taking the bicurve $\gamma$ close to $\pi(\T \times \croset{-1,1})$, the equidistribution on $A(\gamma)$ is close to an equidistribution w.r.t. $\Leb_\M$, while the second set -- whenever nonempty -- is close to be  equidistributed w.r.t. $\mu_1$ or $\mu_{-1}$.
Therefore, the map $g$ pushes forward each measure $\mu_y$ to a measure close to a convex combination of $\Leb_\M$, $\mu_1$ and $\mu_{-1}$.
\end{proof}

\begin{proof}[Proof of \cref{lemma:finerg}]
To obtain i), as explained in the sketch of proof it suffices to construct such a $g$ and $\gamma$ for $q=1$.

First, let us define the bicurve. Let $n \in \N^*$ be such that $N := n^2 \gg \frac{1}{\epsilon^4}$, let $ 1 > \delta \gg \delta' >0$ be such that $\delta \ll \epsilon$ and set:
$$v: \theta \in \T \mapsto \delta' \cos (N \theta) \in~[-\delta',\delta'].$$
We define the bicurve $\gamma := \gamma_+\sqcup \gamma_-\in \Gamma_2(\M)$\footnote{Observe that $\gamma$ is indeed a bicurve by considering as diffeotopy the time one flow of the Hamiltonian $H: (\theta,y) \mapsto \chi(y)\tfrac{\delta'}{N}\sin(N\theta)$, where $\chi$ is a bump function near $1-\delta$ and $-1+\delta$.} by:
\begin{equation}\label{eq:bic}
\gamma_+ := \pi \left( \croset{(\theta,v(\theta) + 1-\delta) \; , \; \theta \in \T} \right)\quad \text{and}\quad 
\gamma_- := \pi \left( \croset{(\theta,v(\theta) - (1-\delta)) \; , \; \theta \in \T} \right)
\end{equation}
We also consider:
$$\gamma_+^v : \theta\in \T \mapsto v(\theta) + 1-\delta \quad \text{and}\quad \gamma^v_- : \theta\in \T \mapsto  v(\theta) - (1-\delta),$$
so that $\gamma_\pm = \pi ( \croset{(\theta,\gamma_\pm^v(\theta)) \, , \, \theta\in \T})$.

As $\gamma$ is defined with graphs of a scaled cosine on the cylinder, it follows that every longitude $\pi(\T \times \croset{y})$ intersects $\gamma$ in a finite set. Hence we recover the second part of iii) : 
\begin{equation}\label{eq:mes_band_0}
\sup_{y \in \I}(\mu_y(B(\gamma,\kappa))) \rconv{\kappa}{0}  0.
\end{equation}
For the first part of iii), it is a direct consequence of the fact that $g$, which is defined later, is supported in $A(\gamma)$.\\

Next, let us use the map $v$ to define the boxes $(A_l)_{0\leq l < N}$, which will be moved by $g$, and other relevant sets. For $l \in \intset{0}{N-1}$ and $\kappa>0$ such that $\kappa \ll \delta'$, we consider:
\begin{align*}
I_l &:= \left[ \frac{l+\kappa}{N},\frac{l+1-\kappa}{N} \right] \subset \T\\
A_l &:= \pi \left( \croset{(\theta,y)\in \A \, : \; \theta \in I_l, \gamma_-^v(\theta) + \frac{2\kappa}{N} \leq y\leq \gamma_+^v(\theta) - \frac{2\kappa}{N}} \right) \subset A(\gamma)\\
O^+ &:= \pi \left(\croset{(\theta,y)\in \A \, : \; \theta \in \T, y\geq \gamma_+^v(\theta) + \frac{\kappa}{N}} \right) \subset O(\gamma)\\
O^- &:= \pi \left(\croset{(\theta,y)\in \A \, : \; \theta \in \T, y\leq \gamma_-^v(\theta) - \frac{\kappa}{N}} \right) \subset O(\gamma).
\end{align*}
See \cref{fig:finerg}.\\

Now, we define the boxes $(\widetilde{A}_l)_{0\leq l < N}$ where we want to send the $(A_l)_l$ by $g$. We consider a bijection $l \in \intset{0}{N-1} \mapsto (l_1,l_2) \in \intset{0}{n-1}^2$ and we define for $l \in \intset{0}{N-1}$:
$$\widetilde{A}_l := \pi \left( \croset{(\theta,y)\in \A \, : \, \theta \in \widetilde{I}_{l_1}, \gamma_+^v(\theta) - \frac{\kappa}{N} -\frac{l_2+1}{n}H + \frac{\kappa}{n^3} \leq y \leq \gamma_+^v(\theta) - \frac{\kappa}{N} -\frac{l_2}{n}H - \frac{\kappa}{n^3} } \right) \subset A(\gamma)$$

where $\widetilde{I}_l = [\frac{l_1 + \kappa}{n}; \frac{l_1 +1 - \kappa}{n}]$ and $H = \gamma_+^v(\theta) - \gamma_-^v(\theta) - 2\kappa/N =  2(1 -\delta - \kappa/N)$.\\

Then, in order to use \cref{thm:folklo_symp} to send the boxes $(A_l)_l$ onto the boxes $(\widetilde{A}_l)_l$, we need to check that they have same area. We denote by $\Leb_\I$ the probability measure on $\I = [-1,1]$ induced by the Lebesgue measure of $\R$, i.e. for $[a,b] \subset I$ we have $\Leb_\I [a,b] = \frac{b-a}{2}$. Hence we have $\Leb_\A = \Leb_\T \otimes \Leb_\I$ and we have the following areas:
\begin{equation}\label{eq:leb_Al}
\Leb_\M A_l = \Leb_\T I_l \cdot \Leb_\I[-1+\delta + 2\frac{\kappa}{N}, 1- \delta -2\frac{\kappa}{N}] = \frac{1-2\kappa}{N}\frac{H-\frac{2\kappa}{N}}{2}
\end{equation}
and
$$\Leb_\M \widetilde{A}_l = \Leb_\T \widetilde{I}_l \cdot \frac{\frac{H}{n} - 2\frac{\kappa}{n^3}}{2} = \frac{1-2\kappa}{n}\frac{\frac{H}{n} - 2\frac{\kappa}{n^3}}{2} = \frac{1-2\kappa}{N}\frac{H - 2\frac{\kappa}{N}}{2} = \Leb(A_l).$$
Then, these boxes have same area and lie in $A(\gamma)$. Therefore, there exist disjoint disks $(D_l)_{0\leq l < N}$ and disjoint disks $(\widetilde{D}_l)_{0\leq l < N}$ in $A(\gamma)$ of the same volume such that $A_l \subset D_l$ and $\widetilde{A}_l \subset \widetilde{D}_l$ for every $l$. 
Moreover we can impose that the diameters of $\widetilde{D}_l$ are $\Theta (\frac{1}{\sqrt{n}})$ (since it is the case for the $\widetilde{A}_l$), where the square root on $n$ comes from $\pi$ being $\tfrac12$-Hölder (see Page \pageref{eq:pi} for the expression of $\pi$). 
This property will allow to obtain the equidistribution of the disks $(\widetilde{D}_l)_l$ in $A(\gamma)$ because $n \gg \tfrac{1}{\epsilon^2}$.\\
Hence, by \cref{thm:folklo_symp}, there exists $g$ in $\sympc{\infty}{\M}$ which send $D_l$ to $\widetilde{D}_l$ for every $l$ and is equal to the identity on $O(\gamma)$ (this gives iii) ).\\
See Figure \ref{fig:finerg} for the construction.\\

Now let us prove that $g$ satisfies ii) of \cref{lemma:finerg}. Let $A := \sqcup_l A_l$, we will use that most of each longitude $\pi(\T \times \croset{y})$ lies in $A \sqcup O^+ \sqcup O^-$ where we control the behaviour of $g$ to obtain that $g_*\mu_y$ is close to $Conv(\Leb_\M,\mu_{-1}, \mu_1)$.\\
We notice that $\mu_y(O^+ \sqcup A \sqcup O^-)$ is close to $1$ uniformly in $y$, since the $(A_l)_l$ are $\frac{2\kappa}{N}$ vertically spaced apart, by \cref{eq:mes_band_0} and because $\kappa \ll 1$. In particular we have $1-\sup_y \mu_y(O^+ \sqcup A \sqcup O^-) \leq \tfrac{\epsilon}{2\mathrm{diam}(\mathcal{M}(\M))}$ because $\kappa \ll \epsilon$. Then, by a result on convex combinations on measures from \cref{prop:kant_conv2}, we have for every $y$:
\begin{equation}\label{eq:munu_y}
d_K(g_*\mu_y, g_*\nu_y )\leq \frac{\epsilon}{2},
\end{equation}
where $\nu_y := \rest{\mu_y}{O^+ \cup A \cup O^-}$. Here, for $E \subset \M$ such that $\mu_y(E) \neq 0$, the measure $\rest{\mu_y}{E}$ denotes the probability measure obtained by restricting $\mu_y$ to $E$: it is defined by $\rest{\mu_y}{E}(E') = \tfrac{\mu_y(E\cap E')}{\mu_y(E)}$. For $E\subset \M$ such that $\mu_y(E) = 0$, $\rest{\mu_y}{E}$ is the null measure. Then we have the following convex combinations:

\begin{equation}\label{eq:mu_conv}
g_* \nu_y = \nu_y(A)g_* \left( \rest{\nu_y}{A}\right) + \nu_y(O^+)g_* \left(\rest{\nu_y}{O^+}\right)+ \nu_y(O^-) g_* \left( \rest{\nu_y}{O^-}\right)
\end{equation}

Let us show that $g_* \left( \rest{\nu_y}{A}\right)$, $g_* \left(\rest{\nu_y}{O^+}\right)$ and $g_* \left( \rest{\nu_y}{O^-}\right)$ are respectively $\frac{\epsilon}{2}$-close to $\Leb_\M$, $\mu_{1}$ and $\mu_{-1}$ when $\nu_y(A)$, $\nu_y(O^+)$ and $\nu_y(O^-)$ are respectively non zero.\\

Let $\pm$ be in $\croset{-,+}$ and $y$ be such that $\nu_y(O^\pm)$ is non zero. Let us show that $g_* \left( \rest{\nu_y}{O^\pm}\right)$ is $\tfrac{\epsilon}{2}$-close to $\mu_{\pm1}$. 
First, since $g$ is the identity on $O^\pm \subset O(\gamma)$, it follows that $g_* \left( \rest{\nu_y}{O^\pm}\right) = \rest{\mu_y}{O^\pm}$. 
Then, regarding the Kantorovich distance, since the map $v$ used to define $\gamma$ is $N$-periodic and $N\gg \tfrac{1}{\epsilon^4}$, the mass of $\rest{\mu_y}{O^\pm}$ is uniformly spread on $\pi(\T\times \croset{y})$ and $\rest{\mu_y}{O^\pm}$ is $\tfrac{\epsilon}{4}$-close to $\mu_y$. 
Moreover $\nu_y(O^\pm)$ is non zero if and only if $\pm y> 1-\delta - \delta'+\tfrac{\kappa}{N}$, it follows that $\rest{\mu_y}{O^\pm}$ is $\tfrac{\epsilon}{4}$-close to $\mu_{\pm 1}$ because $\delta, \delta',\kappa \ll \epsilon$. This yields:
\begin{equation}\label{eq:munu_dk}
d_K(\mu_{\pm 1}, g_* \left( \rest{\nu_y}{O^\pm}\right)) \leq \frac{\epsilon}{2} \quad \text{when} \quad \nu_y(O^\pm) \neq 0.
\end{equation}

Next, for $y$ such that $\nu_y(A) \neq 0$, let us use \cref{an:dK_unif} to obtain that $g_* \left( \rest{\nu_y}{A}\right)$ is $\tfrac{\epsilon}{2}$-close to $\Leb_\M$. 
First, observe that $g_* \left( \rest{\nu_y}{A}\right) = \tfrac{1}{N}\sum_{l=1}^{N}\rest{(g_*\mu_y)}{g(A_l)}$, we therefore work with the families of disjoints sets $(g(A_l))_l$ and of measures $(\rest{(g_*\mu_y)}{g(A_l)})_l$. 
By \cref{eq:leb_Al}, we have $\Leb_M( g(A)) \geq 1-\tfrac{\epsilon}{4\diam(\M)}$ since $\kappa,\delta,\delta' \ll \epsilon$, we also have that $\Leb_\M(g(A_l)) = \Leb_\M(g(A_k))$ for every $k,l$. 
Moreover, because the $(g(A_l))_l$ are contained in the $(D_l)$ whose have diameters of $\Theta(\tfrac{1}{\sqrt{n}})$ and $\sqrt{n} \gg \tfrac{1}{\epsilon}$, it follows that $\diam(g(A_l)) \leq \tfrac{\epsilon}{4}$ for every $l$. 
Therefore, by applying \cref{an:dK_unif} to $(\rest{g_*\mu_y}{g(A_l)})_l$ we obtain:
\begin{equation}\label{eq:lebnu_dk}
d_K(\Leb_\M, g_* \left( \rest{\nu_y}{A}\right)) \leq \frac{\epsilon}{2} \quad \text{when} \quad \nu_y(A) \neq 0.
\end{equation}

Then, from \cref{eq:mu_conv,eq:munu_dk,eq:lebnu_dk}, it follows by a result on the Kantorovich distance on convex combinations given in \cref{prop:kant_conv1}:
$$\forall y \in \I, \;d_K \left( g_*\nu_y, C\right) \leq \epsilon/2,$$
Where we recall that $C = Conv(\Leb_\M, \mu_{-1}, \mu_1)$.\\
Then from the above equation and \cref{eq:munu_y}, we conclude that the measure $g_*\mu_y$ is $\epsilon$-close to the convex hull $C$ for every $y$. \\

This gives ii) of \cref{lemma:finerg} and concludes the proof.

\end{proof}

\subsection{AbC$^\star$ scheme}\label{sec:finerg_scheme}

In this section, we use the previous construction to define the $C^0$-AbC$^\star$ which realizes minimal ergodicity.\\

Let $\alpha = \frac{p}{q} \in \Q$ and $h \in \sympo{0}{\M}$. We define
\begin{equation}
f := \iconj{h}{R_\alpha}
\end{equation}
and consider 
\begin{equation}\label{eq:def_eps}
\epsilon := \sup_{y \in \I} d_K (h_* \mu_y,C).
\end{equation}
Let us first define $U(h,\alpha) \in \mathcal{T}_\Gamma^0$.

\begin{prop}\label{prop:finerg_const}
There exist a bicurve $\gamma \in \Gamma_2(\M)$ invariant under $R_\alpha$ and an open set $U(h,\alpha) \in \mathcal{T}^0_{\gamma}$ such that $U(h,\alpha)$ contains a symplectomorphism $\hat{h}$ that satisfies $f = \iconj{\hat{h}}{R_\alpha}$ and $\rest{\hat{h}}{O(\gamma)} = \rest{h}{O(\gamma)}$. Moreover it holds that:
\begin{equation}\label{eq:H_U}
\forall H \in U(h,\alpha) \, : \, \sup_{y \in \I}d_K(H_*\mu_y,C) \leq \epsilon/2.
\end{equation}
\end{prop}
\begin{proof}
Given $\epsilon' >0$ (to be specified later), there exist $g \in \sympo{\infty}{\M}$ and $\gamma\in \Gamma_2(\M)$ satisfying \cref{lemma:finerg} with parameters $\epsilon'$ and $q$. In particular $\gamma$ is invariant under $R_\alpha$ by \cref{lemma:finerg}.\\
We define $\hat{h} = h \circ g \in \sympo{0}{\M}$. By i) of Lemma \ref{lemma:finerg} we have $f = \iconj{\hat{h}}{R_\alpha}$. 
Next, by iii) of the lemma, we also have $\rest{\hat{h}}{O(\gamma)} = \rest{h}{O(\gamma)}$, giving the first part of \cref{prop:finerg_const}.

We now construct the open set $U(h,\alpha)$ as a neighbourhood of $\hat{h}$ in $\mathcal{T}^0_\gamma$. We require the neighbourhood to be such that, for any $H\in U(h,\alpha)$, $H_*\mu_y$ is uniformly $\epsilon/2$-close to $C$ in $y$. 
First, let us obtain this property for $\hat{h}$. 
Consider the map $P:\mu \mapsto h_*\mu$ which is uniformly continuous on the compact set $\mathcal{M}(\M)$ by \cref{coro:kanto}, and choose $\epsilon'$ sufficiently small such that the modulus of continuity $\omega$ of the map $P$ satisfies $\omega(\epsilon') \leq \epsilon/4$. Then, by ii) of \cref{lemma:finerg} and since $P$ preserves $C$, we have
\begin{equation}\label{eq:hat_C}
\sup_{y\in \I} d_K (\hat{h}_* \mu_y,C) = \sup_{y\in \I} d_K (P(g_* \mu_y),P(C)) \leq \epsilon/4.
\end{equation}

Now, let us use the following fact proved below to define $U(h,\alpha)$.

\begin{fact}\label{fact:cont_sup}
The following map is continuous:
$$\begin{array}{lcrc}
DC : &(\sympo{0}{\M},\mathcal{T}^0_\gamma) &\longrightarrow &(\R_+^*,\left| \cdot \right|)\\
&H &\longmapsto &\sup_{y\in \I} d_K(H_*\mu_y,C).
\end{array}$$
\end{fact}
By the above fact, if we define $U(h,\alpha) := DC^{-1}((0;\epsilon/2))$, then $U(h,\alpha)$ belongs to $\mathcal{T}^0_\gamma$, contains $\hat{h}$ by \cref{eq:hat_C}, and satisfies \cref{eq:H_U}. 

\end{proof}

\begin{proof}[Proof of \cref{fact:cont_sup}]
Let us prove that $DC$ is continuous by sequential characterization. Let $H$ be in $\sympo{0}{\M}$ and $(H_n)_n$ be a sequence in $\sympo{0}{\M}$ converging to $H$ for the topology $\mathcal{T}^0_\gamma$. By definition of this topology (see Page \pageref{def:bicurve}), there exists a sequence $(\kappa_n)_n$ of positive numbers which converges to $0$ such that:
$$d_{C^0}(\rest{H}{\M \setminus B(\gamma,\kappa_n)},\rest{H_n}{\M \setminus B(\gamma,\kappa_n)}) \rconv{n}{\infty} 0.$$
In particular we have by \cref{prop:dk_c0}:
\begin{equation}\label{eq:conv_Hn}
\sup_{y\in \I} d_K \left( {H_n}_* \left( \rest{\mu_y}{\M \setminus B(\gamma,\kappa_n)} \right),H_* \left( \rest{\mu_y}{\M \setminus B(\gamma,\kappa_n)} \right) \right) \rconv{n}{\infty} 0
\end{equation}
Observe then that, for every $y \in \I$, we have the following inequality by \cref{prop:kant_conv1}:
\begin{equation}\label{eq:ineq_Hn}
\begin{array}{rcl}
d_K ({H_n}_* \mu_y, H_*\mu_y) &\leq &c_{\kappa_n}(y)  d_K \left( {H_n}_* \left( \rest{\mu_y}{B(\gamma,\kappa_n)} \right),H_* \left( \rest{\mu_y}{B(\gamma,\kappa_n)} \right) \right)\\
& & +(1-c_{\kappa_n}(y) ) d_K \left( {H_n}_* \left( \rest{\mu_y}{\M \setminus B(\gamma,\kappa_n)} \right),H_* \left( \rest{\mu_y}{\M \setminus B(\gamma,\kappa_n)} \right) \right)
\end{array},
\end{equation}
with $c_{\kappa_n} := y\in \I \mapsto \mu_y (B(\gamma,\kappa_n))$.\\
Moreover, by iii) of \cref{lemma:finerg}, the functions $(c_{\kappa_n})_n$ converge uniformly to $0$. Therefore, beause $\mathcal{M}(\M)$ has bounded diameter, it follows from \cref{eq:conv_Hn,eq:ineq_Hn} that:
$$\sup_{y\in \I} d_k({H_n}_*\mu_y,H_* \mu_y) \rconv{n}{\infty}0.$$
This yields the continuity of $DC$ by continuity of the distance to the compact set $C$.
\end{proof}

Next, by defining $\nu(h,\alpha)>0$ sufficiently small, we obtain the following.

\begin{prop}\label{prop:minerg_nu}
There exists $\nu(\alpha,h)>0$ such that for every $\hat{\alpha} \in (\alpha -\nu(h,\alpha), \alpha + \nu(h,\alpha))$, $\hat{\alpha} \neq \alpha$, it holds:
\begin{samepage}
\begin{enumerate}[label = \roman*)]
\item $\Delta_{merg}(\hat{f}) = \sup_{x \in \M} d_K (e^{\hat{f}}(x), C) \leq 2\epsilon$, where $\hat{f} = \iconj{h}{R_{\hat{\alpha}}}$,
\item $\sup_{x \in \M} d_{K}(e^{f}_k(x) , e^{\hat{f}}_k(x)) \leq \epsilon/2$ for every $k$ less than the denominator $q$ of $\alpha$,
\item $d_{C^0}(\hat{f},f) \leq \epsilon/2$,
\item if $\hat{\alpha}$ is rational, then its denominator is greater than $q$.
\end{enumerate}
\end{samepage}
\end{prop}

\begin{proof}
First, iv) is immediate because we take $\nu(h,\alpha)$ small. It is also the case for ii) and iii) because $f = \iconj{h}{R_\alpha}$ is close to $\hat{f} = \iconj{h}{R_{\hat{\alpha}}}$ when $\alpha$ is close to $\hat{\alpha}$.\\
For i), as $\hat{f}$ is conjugate to a rotation, every of its empirical measures is defined and ergodic. More precisely, for $x = h(\pi(\theta,y)) \in \M$, we have:
$$e^{\hat{f}}(x) = h_* \lconv{n}{\infty} \frac{1}{n}\isum{k=1}{n}\delta_{\pi(\theta+k\hat{\alpha},y)}.$$
If $\hat{\alpha}$ is irrational, then we have $e^{\hat{f}}(x) = h_*\mu_y$, so we have i) by \cref{eq:def_eps}. Otherwise, if $\hat{\alpha}$ is rational and $\nu(h,\alpha)$ is small enough, then the denominator of $\hat{\alpha}$ is large enough so that $e^{\hat{f}}(x)$ is $\epsilon$-close to $h_*\mu_y$; that is, i) holds by \cref{eq:def_eps}.

\end{proof}

\subsection{Realization of minimal ergodicity}

The previous section defines the following map:

$$\begin{array}{lccc}
(U,\nu): &\sympc{0}{\M} \times \Q/\Z &\rightarrow &\mathcal{T}_\Gamma^0 \times (0,\infty)\\
&(h,\alpha) &\mapsto &(U(h,\alpha),\nu(h,\alpha)).
\end{array}$$

We now show that the map $(U,\nu)$ defines a $C^0$-AbC$^\star$ scheme realizing minimal ergodicity. In particular, we prove \cref{prop:real_minerg}.

\begin{proof}[Proof of \cref{prop:real_minerg}]
First, observe that \cref{prop:finerg_const,prop:minerg_nu} provide a well-defined $C^0$-AbC$^\star$ scheme. Indeed, conditions a) and b) of \cref{def:abc_star} are satisfied with \cref{prop:finerg_const}: for $(h,\alpha) \in \sympc{0}{\M} \times \Q/\Z$, there exist $\hat{h} \in U(h,\alpha)$ and a bicurves $\gamma \in \Gamma^2(\M)$, which is invariant under $R_\alpha$, such that $\iconj{h}{R_\alpha} = \iconj{\hat{h}}{R_\alpha}$, $\rest{h}{O(\gamma)} = \rest{\hat{h}}{O(\gamma)}$ and $U(h,\alpha) \in \mathcal{T}^0_\gamma$.\\
For c), let us consider sequences $(h_n)_{n\in \N}$ and $(\alpha_n)_{n \in \N}$ constructed by the scheme $(U,\nu)$. That is, we have $h_0 = id$, $\alpha_0 = 0$, and for $n \in \N$ the map $h_{n+1}$ lies in $U(h_n,\alpha_n)$ and $0 < \lvert \alpha_n - \alpha_{n+1} \rvert < \nu(h_{n+1},\alpha_n)$. In addition we have:
$$f_n := \iconj{h_n}{R_{\alpha_n}} = \iconj{h_{n+1}}{R_{\alpha_n}}.$$
We write $\alpha_n = \frac{p_n}{q_n}$, with $p_n \wedge q_n =1$.\\
By considering the sequence $(\epsilon_n)_{n\in \N}$ defined by 
\begin{equation}
\epsilon_n = \sup_{y\in \I}d_K({h_n}_*\mu_y,C),
\end{equation}
we deduce from \cref{prop:finerg_const,prop:minerg_nu} that for every $n \geq 0$ by :

\begin{enumerate}[label = \Roman*)]
\item $\epsilon_{n+1} \leq \frac{\epsilon_{n}}{2}$ and so $\epsilon_{n+1} \leq \tfrac{\epsilon_0}{2^{n+1}}$,
\item $d_{C^0}(f_{n+1} , f_{n})\leq \tfrac{\epsilon_{n}}{2}$.
\end{enumerate}

Then, condition c) of \cref{def:abc_star} is satisfied since $(f_n)_n$ is a Cauchy sequence for $d_{C^0}$ by I) and II). We denote by $f$ its limit in $\symp{0}{\M}$.\\

It remains to prove that $(U,\nu)$ realizes minimal ergodicity, i.e. that $f$ satisfies $\Delta_{merg}(f) = 0$, hence is minimally ergodic. So let $\mu$ be an ergodic invariant measure of $f$. By Birkhoff's ergodic theorem, there exists $x\in \M$ such that $e^f(x) = \mu$. 
We show that $e^f(x) \in C$ by working with the $e^{f_n}(x)$. First, by i), ii), and iv) of \cref{prop:minerg_nu} applied with $0 < \left| \alpha_{n+1} - \alpha_n \right| < \nu(h_{n+1},\alpha_n)$ we have for $n\geq 0$:
\noindent \begin{samepage}
\begin{enumerate}[resume,label = \Roman*)]
\item $d_K (e^{f_{n+1}}(x), C) \leq \Delta_{merg}(f_{n+1}) \leq 2\epsilon_{n+1}$,
\item $d_{K}(e^{f_{n+1}}_k(x) , e^{f_{n}}_k(x)) \leq \tfrac{\epsilon_{n+1}}{2}$ for every $k \leq q_{n}$,
\item and $q_{n+1} > q_{n}$.
\end{enumerate}
\end{samepage}
Then, for every $m\geq n \ge 1$, V) also implies that $q_n \leq q_j$, $\forall j \in \intset{n}{m}$, hence we have by I) and IV):
$$d_{K} (e^{f_m}_{q_n}(x),e^{f_n}_{q_n}(x)) \leq \frac{\epsilon_0}{2^{n+1}}.$$

Next, letting $m \to \infty$ in this inequality and using that $f_m \to f$,  we have for every $n \in \N^*$:
$$d_{K} (e^{f}_{q_n}(x), e^{f_n}_{q_n}(x)) \leq \frac{\epsilon_0}{2^{n+1}}.$$
Yet, we have by I) and III) that:
$$d_K (e^{f_n}_{q_n}(x), C) = d_K (e^{f_n}(x), C) \leq \frac{\epsilon_0}{2^{n-1}}.$$
Then we deduce from the two previous inequalities that:
$$d_{K} (e^f_{q_n}(x),C) \leq \frac{\epsilon_0}{2^{n-2}}.$$

Finally, letting $n \to \infty$ and using that $e^f_n(x) \to e^f(x)$, we obtain that the ergodic measure $\mu = e^f(x)$ belongs to $C$. Hence $\Delta_{merg}(f) = 0$. In particular, we have the following inclusions by the ergodic decomposition theorem:
$$\mathcal{M}(f) \subset C = Conv(\Leb_\M, \mu_{-1}, \mu_1) \subset \mathcal{M}(f).$$
We conclude that $f$ is minimally ergodic by \cref{prop:minerg_inc}.\\

Moreover, minimal ergodicity is invariant under $\symp{0}{\M}$-conjugacy, which completes the proof.
\end{proof}

\section{Proof of the AbC$^\star$ principle}\label{sec:proof_princ}

In this section, we aim to prove the AbC$^\star$ principle. To this end, we adapt the proof of Berger's Principle in \cite{berger_analytic_2024}[§6]. 
We first establish an approximation theorem in \cref{an:approx}. In \cref{sec:complex} we introduce the complexifications of our surfaces and results on these complexifications used to prove the principle. 
Then, we generalize the theorem of approximation modulo deformation in \cref{sec:approx_mod}, so that it aligns with our scheme using bicurves. 
Finally we obtain the analytic realization of the scheme in  \cref{sec:real}.\\

To precise, let us recall some notions.\\

\begin{defin}
A \underline{smooth structure of symplectic surface} $(\M, \Omega)$ is a maximal $C^\infty$-atlas $\mathcal{A}^\infty = (\phi_\alpha)_\alpha$ such that each chart $\phi_\alpha : U \subset M \rightarrow V \subset \R^2$ pulls back the symplectic form $dx \wedge dy$ to $\rest{\Omega}{U}$.\\

A \underline{real analytic structure of symplectic surface} $(\M, \Omega)$ is a maximal atlas $\mathcal{A}^\omega = (\phi_\beta)_\beta$ such that the coordinates changes are real analytic and each chart $\phi_\beta : U \subset \M \rightarrow V \subset \R^2$ pulls back the sympletic form $dx \wedge dy$ to $\rest{\Omega}{U}$.\\
We say that the structure $\mathcal{A}^\omega$ is compatible with $\mathcal{A}^\infty$ if $\mathcal{A}^\omega \subset \mathcal{A}^\infty$.
\end{defin}

By \cref{prop:ext_real}, establishing the principle on the sphere — whose structure is simpler due to the absence of boundary — reduces to proving the following theorem.

\begin{theorem}\label{thm:princ_strctr}
For any AbC$^\star$ realizable $C^\infty$-property $(\mathcal{P})$ on $\symp{\infty}{\Sp}$, there exists a real analytic symplectic structure $\mathcal{A}^{\omega'}$ on $(\Sp, \Omega)$ which is compatible with the canonical $C^\infty$-structure and for which $(\mathcal{P})$ is realizable by an analytic symplectomorphism $f$. 
\end{theorem}

We prove this theorem in \cref{sec:conc_proof}.\\

However, this theorem provides the result for an arbitrary $C^\infty$-compatible structure, but we aim to recover the result for the canonical real analytic structure $\mathcal{A}^\omega$ of the sphere induced by the inclusion $\Sp \subset \R^3 \subset \C^3$. To this end, we use the following theorem from \cite{kutzschebauch_real_2000}.

\begin{theorem}\label{thm:str_sph}
Any symplectic manifold possesses a unique (up to isomorphism) real analytic symplectic structure.
\end{theorem}
 Hence, if $f$ is an analytic symplectomorphism for a structure $\mathcal{A}^{\omega'}$ provided by \cref{thm:princ_strctr}, then there exists a map $\Phi : (\Sp, \mathcal{A}^\omega ) \rightarrow (\Sp , \mathcal{A}^{\omega'})$ analytic and smooth with respect to the canonical structure, and preserving $\Omega$. 
In particular $\Phi$ belongs to $\symp{\infty}{\Sp}$. 
Consequently, $\conj{\Phi}{f}$ is a $C^\omega$-symplectomorphism that satisfies $(\mathcal{P})$, since $(\mathcal{P})$ is invariant under $\symp{\infty}{\Sp}$-conjugacy.\\

Now, for $\M = \A$ or $\Di$, we need to adjust the result for manifolds with boundary. To this end we have to work with $K$-analytic structure on this manifold where $K$ is a Lie group. In particular with $K = \Z_2$ in our cases.\\

Indeed, we can glue two copies of $\M$ along its boundary to form a boundaryless surface $\tilde{\M}$, making $\M$ the quotient of $\tilde{\M}$ by a reflection.\\

For instance $(\A,\Omega)$ is identified with the quotient of the torus $\tilde{\A} := (\T \times \R/4\Z, \Omega)$ by the involution
$$\tau : (\theta,y) \in \T \times \R/4\Z \mapsto (\theta, -y +2 \mod 4).$$

Likewise, we will introduce in \cref{sec:complex} the identification of $(\Di , \Omega)$ as a quotient of a symplectic sphere $(\tilde{\Di}, \Omega)$ by an involution, also denoted $\tau$.\\

Then, this involution induces a smooth action of the Lie group $\Z_2$ onto $\tilde{\M}$ satisfying $\tau^* \Omega = -\Omega$. So that we introduce the following definition.

\begin{defin}
A \underline{real analytic $\Z_2$-structure on a symplectic surface} $(\M , \Omega)$ is a maximal real analytic symplectic atlas $\mathcal{A}^\omega = (\phi_\beta)_\beta$ such that each chart $\phi_\alpha : U_\alpha \rightarrow \R\times \R^+$ can be lifted to an open set $\widetilde{U_\alpha} \subset \tilde{\M}$ forming a $\Z_2$-equivariant chart $\tilde{\phi}_\alpha$ of $\tilde{\M}$. i.e., with $\tau_0(x_1,x_2) := (x_1,-x_2)$, we have 
$$\tilde{\phi}_\alpha \circ \tau = \tau_0 \circ \tilde{\phi}_\alpha.$$

\end{defin}

The adaptation of \cref{thm:princ_strctr} is given by the following theorem for $\M \in \croset{\A , \Di}$.

\begin{theorem}\label{thm:princ_strctr_equi}
For any AbC$^\star$ realizable $C^\infty$-property $(\mathcal{P})$ on $\symp{\infty}{\M}$, there exists a real analytic symplectic $\Z_2$-structure $\mathcal{A}^{\omega'}$ on $(\M, \Omega)$ which is compatible with the canonical $C^\infty$-structure and for which $(\mathcal{P})$ is realizable by a $\mathcal{A}^{\omega'}$-analytic symplectomorphism $f$. 
\end{theorem}

We also prove this theorem in \cref{sec:conc_proof}.\\

As in the case of the sphere, we aim to recover the canonical real analytic symplectic $\Z_2$-structure induced by the inclusion 
$$\A \subset \R/\Z \times \R/4\Z \subset \C/\Z \times \C/4\Z \text{  and  } \Di \subset \R^2 \subset \C^2.$$

Thus, we want the equivalent of \cref{thm:str_sph} on the cylinder and the disk. It is given by the following adaptation of a result in \cite[§4]{kutzschebauch_real_2000}.

\begin{theorem}[\cite{berger_analytic_2024} Theorem 3.3]\label{thm:strct_equiv}
Let $K$ be a compact Lie group acting analytically on a compact surface $\widetilde{M}$. Let $\widetilde{\Omega}$ be a symplectic form on $\widetilde{M}$ and $\Gamma : K \rightarrow \croset{-1;1}$ be a Lie group morphism such that for all $g\in K$ we have $g^*\widetilde{\Omega} = \Gamma(g)\widetilde{\Omega}$. Then $(\widetilde{M},\widetilde{\Omega})$ has a unique real analytic $K$-structure up to analytic symplectomorphism.
\end{theorem}

Then, as for the sphere, this theorem concludes to obtain the compatibility with the canonical analytic structure.\\

\begin{remark}
Nevertheless, we must be careful with the implications of this theorem. 
In fact, the $K$-equivariant symplectomorphism $\widetilde{\Psi}$ from $\widetilde{M}$ to $\widetilde{M}$, which switches the analytic structure, does not necessarily descend to a symplectomorphism on the quotient $M := \widetilde{M}/K$ (for $\widetilde{M} = \tilde{\A} = \R/4\Z \times \T$ we have $K = \croset{\id,\tau}$ and $\tau^*\Omega = - \Omega$, the symplectomorphism $(y,\theta) \mapsto (y+2,\theta)$ on $\widetilde{\A}$ descends to the anti-symplectic map $(y,\theta) \mapsto (-y,\theta)$ on $\A$). 
However, in our case ($\widetilde{M} = \tilde{\A}$ or $\tilde{\Di}$), the projection $\widetilde{M} \rightarrow M$ is symplectic or anti-symplectic outside preimages of $\partial M$, and an element of $K$ is either symplectic or anti-symplectic. 
Thus, if $\widetilde{\Psi}$ is the symplectomorphism which switches the analytic structures, its projection $\Psi$ onto $\A$ or $\Di$ is either symplectic or anti-symplectic. 
Therefore conjugation by $\Psi$ preserves symplecticity.
\end{remark}

The remaining of this section will be devoted to proving \cref{thm:princ_strctr,thm:princ_strctr_equi}.\\
The first step of the proof is to adapt an approximation theorem to obtain an analytic approximation outside bicurves.

\subsection{Approximation theorem}\label{an:approx}

In this first subsection, we generalize the approximation theorem of Pierre Berger in \cite{berger_analytic_2022} to approximate symplectomorphisms outside bicurves on the cylinder (see \cref{def:bicurve} and \cref{not:bic}), so that it can be used to compute approximations modulo deformation.\\

In fact, Berger's Theorem provides an approximation outside a neighbourhood of the boundary of $\A$. However, we would like to shift the uncontrolled domain. 
Thus, we first directly obtain from this theorem the approximation outside two horizontal annuli. 
Then, by straightening a bicurve using symplectomorphisms, we can shift the uncontrolled domain to a neighbourhood of the bicurve. This is the general approach taken in this section.\\

First, we recall some notations that will be involved in the main theorem.
\newglossaryentry{3Symp_c_an0}{sort = {3Func Space}, name = {$\sympan{}{\A_\infty}$}, description = {Biholomorphic symplectomorphism of $\A_\infty$}}\glsadd{3Symp_c_an0}
\newglossaryentry{3Symp_c_an1_rho}{sort = {3Func Space}, name = {$\sympan{\rho}{\A_\infty}$}, description = {Biholomorphic symplectomorphism of $\A_\infty$ close to $\id$ next to $\A$}}\glsadd{3Symp_c_an1_rho}
\begin{notation}\label{not:k_ham}
First, we recall that $\symps{\infty}{\A}$ denotes the space of smooth symplectomorphism of $\A$ which are compactly supported in $\A \setminus \partial\A$.\\
Let $\sympan{}{\A_\infty}$ be the space of biholomorphism of $\A_\infty := \C/\Z \times \C$ leaving invariant $\Omega_0$ and $\T \times \R$, where $\Omega_0$ is the holomorphic extension of $\Omega = \tfrac{1}{2}d\theta\wedge dy$ to $\A_\infty$.\\
Then, we define:
$$K_\rho: = \T_\rho \times Q_\rho,$$
with $\T_\rho := \T + i [-\rho , \rho]$ and $Q_\rho := [-\rho-1 ; -1] \sqcup [1;\rho] +  i [-\rho ; \rho]$.\\
For $\rho > 1$ we define: 
$$\sympan{\rho}{\A_\infty}: = \croset{F \in \sympan{}{\A_\infty} , \; \sup_{x \in K_\rho } \lvert F(x) - x \rvert < \rho^{-1}}.$$
\\
For $\eta>0$, we consider $K_{\rho,\eta}$ and $Q_{\rho,\eta}$ the $\eta$-neighbourhoods of $K_\rho$ and $Q_\rho$:
$$K_{\rho,\eta} = \T_{\rho+\eta} \times Q_{\rho+\eta} \quad \text{and} \quad Q_{\rho,\eta} := \Big( [-\rho-1-\eta ; -1+\eta] \sqcup [1-\eta;\rho+\eta]\Big) +  i [-\rho - \eta ; \rho + \eta].$$
Finally we consider:
$$\sympan{\rho,\eta}{\A_\infty}: = \croset{F \in \sympan{}{\A_\infty} , \; \sup_{x \in K_{\rho,\eta} } \lvert F(x) - x \rvert < \rho^{-1}}.$$

\end{notation}

Here is our main approximation theorem:

\begin{theorem}[Main Approximation Theorem]\label{an:thm_main_approx}
Let $\gamma \in \Gamma_2(\A)$ and $f \in \symps{\infty}{\A}$ be such that $\rest{f}{O(\gamma)} = \id$. Then, for any $\rho >1$ and any neighbourhood $\mathcal{U}$ of $\rest{f}{\A \setminus \gamma}$ in the smooth compact-open topology of $C^\infty(\A \setminus \gamma,\T\times\R)$, there exists a map $F \in \sympan{\rho}{\A_\infty}$ such that $\rest{F}{\A \setminus \gamma}\in \mathcal{U}$.
\end{theorem}

\begin{remark}\label{rk:gam_bord}
By definition of the smooth compact-open topology on $C^\infty(\A \setminus \gamma,\T\times\R)$, proving that $\rest{F}{\A \setminus \gamma}$ is in $\mathcal{U}$ reduces to proving that
$$d_{C^r}(\rest{F}{\A \setminus B(\gamma,\kappa)},\rest{f}{\A \setminus B(\gamma,\kappa)}) < \epsilon$$
for a sufficiently large and finite $r \geq 0$, and $\kappa,\epsilon >0$ small enough.

It follows that, if $\gamma'$ is a bicurve contained in $B(\gamma,\kappa)$, there exists a neighbourhood $\mathcal{U}'$ of $\rest{f}{\A \setminus \gamma'}$ in the smooth compact-open topology of $C^\infty(\A \setminus \gamma',\T\times\R)$ such that for any $F: \T\times \R \to \T \times \R$:
$$\rest{F}{\A \setminus \gamma'} \in \mathcal{U}' \Longrightarrow \rest{F}{\A \setminus \gamma}\in\mathcal{U}.$$

Moreover, since $f$ is compactly supported in $\A \setminus \partial \A$ and $\supp(f) \subset \overline{A(\gamma)}$, the set 
$$B(\gamma,\kappa) \setminus (\partial\A \cup \supp(f))$$ is non-empty.

Choosing $\gamma'$ inside this set so that $\gamma' \subset \A \setminus \partial \A$ and $f$ is compactly supported in $A(\gamma')$ allows to reduce the proof of the theorem to bicurves of this form. 
\end{remark}

\begin{remark}
Observing that the lift of $\C/\frac{1}{q}\Z$ to $\C/\Z$ is holomorphic, if in the latter theorem $f$ commutes with $R_{1/q}$ and if $\gamma$ is $R_{1/q}$-invariant, we can assume that the approximation also commutes with $R_{1/q}$.
\end{remark}

In view of \cref{rk:gam_bord}, we restrict our attention to the case $\gamma \subset \A \setminus \partial \A$ and $\supp(f) \subset A(\gamma)$.\\

We begin by recalling Berger's approximation theorem. The central idea of this result is to decompose the Hamiltonian map into a sequence of Hamiltonian maps. Each of these maps is then approximated by compositions of twists via Fourier-series expansions. Finally, these twists are approximated analytically by means of the Runge Theorem, at the expense of introducing a region near the boundary where the approximation cannot be controlled.
 
\begin{theorem}[\cite{berger_analytic_2022} Thm. 1.8]\label{thm:approx_berger}
Let $0 < \eta < 1$, $f$ in $\symps{\infty}{\A}$ and let $\mathcal{U}$ be a neighbourhood of the restriction $\rest{f}{\A(\eta)}$ in the smooth compact-open topology of $C^\infty(\A(\eta),\A)$. Then, for any $\rho > 1$, there exists $F \in \sympan{\rho}{\A_\infty}$ such that the restriction $\rest{F}{\A(\eta)}$ is in $\mathcal{U}$.
\end{theorem}

\begin{remark}\label{rk:dens}
Berger originally states this theorem for a symplectomorphism $f \in \hams{\infty}{\A}$, where $\hams{\infty}{\A}$ denotes the space of smooth symplectomorphisms of $\A$ which are compactly supported in $\check{\A}=\A \setminus \partial\A$ and isotopic to $\id$ via an isotopy compactly supported in $\A \setminus \partial\A$. Yet, this space is dense in the space $\symps{\infty}{\A}$ for the smooth compact-open topology of $C^\infty(\check \A, \T\times \R )$. Therefore we can assume that $f$ belongs to $\symps{\infty}{\A}$ in the theorem.
\end{remark}

Observe that this theorem provides the case $\gamma = \partial \A$ of the main approximation theorem.\\

Then, up to conjugacy by analytic diffeomorphisms, we deduce from the following corollary from \cref{thm:approx_berger}.

\begin{coro}\label{an:thm_approx}
Let $-1 < y_- < y_+ < 1$ and $\gamma = \T \times \croset{y_-,y_+}$, and let $f$ be in $\symps{\infty}{\A}$ such that $f$ is compactly supported in $A(\gamma)$. Let $\mathcal{U}$ be a neighbourhood of $\rest{f}{\A\setminus \gamma}$ in the smooth compact-open topology of $C^\infty(\A\setminus \gamma,\T \times \R)$ and let $\eta \in (0,d(\gamma,\partial \A))$. Then, for any $\rho >1$, there exists $F \in \sympan{\rho,\eta}{\A_\infty}$ such that $\rest{F}{\A \setminus \gamma}\in \mathcal{U}$.
\end{coro}

\begin{proof}
Let $g: \A_\infty \to \A_\infty$ be an affine map which sends the bicurve $\gamma$ to $\partial \A$, in particular it sends $A(\gamma)$ to $\A \setminus \partial \A$. 
Hence, the map $\tilde{f} := \iconj{g}{f}$ belongs to  $\symps{\infty}{\A}$ as $f$ is compactly supported in $A(\gamma)$. 
Then, observe that conjugation by $g$ establishes a bijection between the smooth compact-open topologies of $C^\infty(\A\setminus \gamma,\T \times \R)$ and $C^\infty(\check{\A} \cup g(O(\gamma)),\T \times \R)$, thus we consider $\widetilde{\mathcal{U}}$ to be the neighbourhood of $\rest{\tilde{f}}{\check{\A} \cup g(O(\gamma))}$ obtained by conjugating $\mathcal{U}$ by $g$.

Next, since $g(O(\gamma))$ is contained in $K_{\tilde \rho}$ for a large $\tilde \rho >1$ and $\rest{\tilde{f}}{g(O(\gamma))}=\id$, establishing that a biholomorphic map $\widetilde{F}$ is $C^\infty$-close to $\tilde{f}$ on $g(O(\gamma))$ reduces to  establishing that $\widetilde{F}$ is $C^0$-close to $\id$ on $K_{\tilde \rho}$ by Cauchy's estimates. 
Hence, let $\delta \in (0,1)$, $\mathcal{V}$  be a neighbourhood of $\rest{\tilde f}{\A(\delta)}$ in the smooth compact-open topology of $C^\infty(\A(\delta),\A)$, and let $\tilde{\rho} >1$ be sufficiently large such that the following holds: 
\begin{equation}\label{eq:V_coro}
\forall \widetilde{F}\in \sympan{\tilde \rho}{\A_\infty} \, , \, \rest{\widetilde{F}}{\A(\delta)} \in \mathcal{V} \Longrightarrow  \rest{\widetilde{F}}{\check{\A} \cup g(O(\gamma))} \in \widetilde{\mathcal{U}}.
\end{equation}

Now, by \cref{thm:approx_berger}, there exists $\widetilde{F} \in \sympan{\tilde \rho}{\A_\infty}$ such that the restriction $\rest{\widetilde{F}}{\A(\delta)} \in \mathcal{V}$, in particular the restriction $\rest{\widetilde{F}}{\check{\A} \cup g(O(\gamma))}$ is in $\widetilde{\mathcal{U}}$ by \cref{eq:V_coro} and the map $F:= \conj{g}{\widetilde{F}}$ restricted to $\A \setminus \gamma$ belongs to $\mathcal{U}$ since $\widetilde{\mathcal{U}} = g \circ \mathcal{U} \circ g^{-1}$. 
It remains to show that $F$ belongs to $\sympan{\rho,\eta}{\A_\infty}$. 
First, $F$ is a biholomorphism in $\sympan{}{\A_\infty}$ since $\widetilde{F}$ is one and $g$ is an affine map.

Next, since $\eta$ belongs to $(0,d(\gamma,\partial\A))$, observe that $\gamma$ belongs to $\A(\eta)$. Hence, since $g$ sends $\gamma$ to $\partial \A$ and $\tilde{\rho}$ is large, it sends $K_{\rho,\eta}$ into $K_{\tilde \rho}$ where $\widetilde{F}$ is $1/\tilde{\rho}$-close to $\id$. Thus, $F = \conj{g}{\widetilde{F}}$ is $1/\rho$-close to $\id$ on $K_{\rho,\eta}$ and belongs to $\sympan{\rho,\eta}{\A_\infty}$.
\end{proof}

The next step to establish \cref{an:thm_main_approx} is to straighten a bicurve so as to make the preceding corollary applicable. This is provided by the following proposition, which constructs a symplectomorphism that maps a horizontal bicurve to a prescribed bicurve $\gamma$. This symplectomorphism is obtained by applying Moser's trick to an appropriately chosen diffeotopy of the cylinder.

\begin{prop}\label{an:lemma_symp_bi}
Let $\gamma$ be in $\Gamma_2(\A)$, there exist a bicurve $\gamma' := \T \times \croset{y_-,y_+}$ and a symplectomorphism $\phi\in\mathrm{Symp}_c^\infty(\A,\partial\A)$ such that
$$\phi \left (\gamma'\right ) = \gamma.$$
\end{prop}

This proposition is proven in \cref{an:sec_bic}. We now prove the main approximation theorem, and we start by giving a sketch of proof.

\begin{proof}[Sketch of proof of Theorem \ref{an:thm_main_approx}]
Let $\gamma \in \Gamma_2(\A)$ and $f \in \symps{\infty}{\A}$ be such that $\gamma \subset \A \setminus \partial\A$ and $\supp(f) \subset A(\gamma)$. 
Let $\mathcal{U}$ be a neighbourhood of $\rest{f}{\A \setminus \gamma}$ in the smooth compact-open topology of  $C^\infty(\A \setminus \gamma,\T\times \R)$.

The first step is to consider a symplectomorphism $\phi$ that sends a bicurve $\gamma' := \T \times \croset{y_-,y_+}$ to $\gamma$, as provided by \cref{an:lemma_symp_bi} (see \cref{fig:conjugacy}). 
We then use \cref{an:thm_approx} to approximate $\phi$ outside a bicurve $\gamma^\eta$ by a biholomorphism $\Psi$, and we consider the conjugate $g := \conj{\Psi}{f}$ (see \cref{fig:approx_phi}). 
Next, we again use \cref{an:thm_approx} to approximate $g$ outside the bicurve $\gamma'$ by a biholomorphism $G$, and we define the desired biholomorphism $F := \iconj{\Psi}{G} \in \sympan{}{\A_\infty}$ (see \cref{fig:an_ap_conj}). 
It remains to prove that $\rest{F}{\A \setminus \gamma}$ belongs to $\mathcal{U}$. This follows from the approximations of $g$ and $\phi$. 
It also remains to prove that $F$ belongs to $\sympan{\rho}{\A_\infty}$. This is achieved using the following result \cite[Lemma 3.15]{berger_analytic_2022} on the composition of biholomorphisms:
\begin{lemma}\label{lem:conj_ham}
For any $\rho >1$ and $\eta >0$, let $\rho_k' := 2^k \max \croset{2/\eta , \rho}$ for $k \geq 1$. Then, for any sequence $(F_k)_{k\geq 1}$ of maps $F_k \in \symp{\omega}{\A_\infty}$ such that $\sup_{x \in K_{\rho, \eta/2}} d(F_k(x),x) < 1/\rho_k'$, the following holds:
$$F_1 \circ \cdots \circ F_n \in \sympan{\rho}{\A_\infty} \; \text{for every} \; n\geq1.$$
\end{lemma}
To use this result, we first choose a suitable $\eta>0$ after defining $\phi$. 
Then, when we apply \cref{an:thm_approx}, we consider $\Psi \in \sympan{\rho_1,\eta}{\A_\infty}$ and $G \in \sympan{\rho_2,\eta}{\A_\infty}$ with $\rho_1$ and $\rho_2$ sufficiently large to apply the lemma.

This proves the case $\gamma \subset \A \setminus \partial \A$ and $\supp(f) \subset A(\gamma)$. The other cases, are consequences of \cref{rk:gam_bord}.
\end{proof}

\begin{proof}[Proof of Theorem \ref{an:thm_main_approx}]
\textbf{\underline{Case $\gamma \subset \A \setminus \partial \A$ and $\supp(f) \subset A(\gamma)$:}}

Let $\gamma \in \Gamma_2(\A)$ and $f \in \symps{\infty}{\A}$ be such that $\gamma \subset \A \setminus \partial\A$ and $\supp(f) \subset A(\gamma)$. 
As $f$ is compactly supported in $\A \setminus \partial\A$, we also denote its canonical extension to $\T \times \R$ by $f$.

Let $\mathcal{U}$ be a neighbourhood of $\rest{f}{\A \setminus \gamma}$ in the smooth compact-open topology of  $C^\infty(\A \setminus \gamma,\T\times \R)$ and let $\rho>1$. Let us build the desired biholomorphism $F$.

First, we conjugate $f$ by a map $\phi$ that straightens $\gamma$. By \cref{an:lemma_symp_bi}, we choose $\phi \in \symps{\infty}{\A}$ and $\gamma' := \T \times \croset{y_-,y_+} \in \Gamma_2(\A)$ such that $\phi(\gamma') = \gamma$. Let $\eta>0$ be such that $\phi$ is compactly supported in $\A(3\eta)$ and $3\eta < d(\gamma,\partial\A)$, in particular both $\gamma$ and $\gamma'$ are contained in $\A(3\eta)$. See \cref{fig:conjugacy}.\\

\begin{figure}[htb]
    \centering
    \includegraphics[width = 0.55\textwidth]{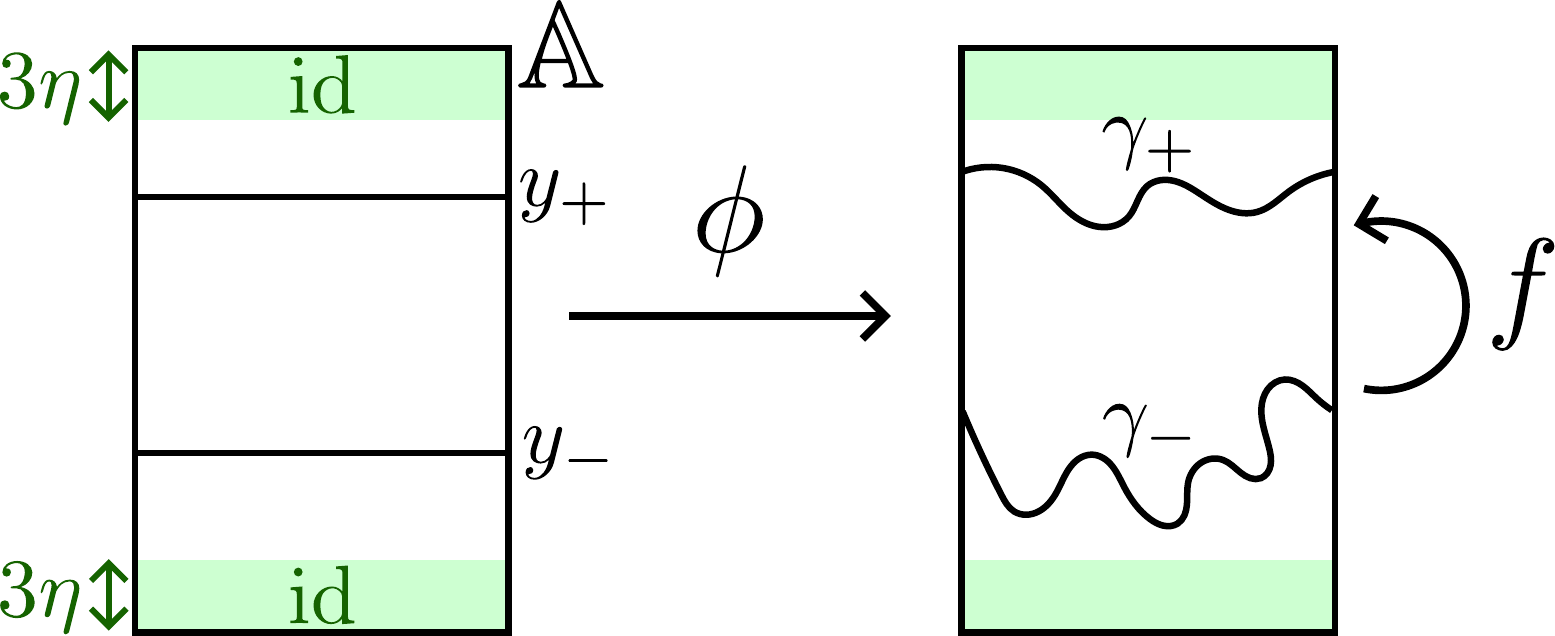}
    \caption{Symplectomorphism $\phi$ straightening the bicurve.}
    \label{fig:conjugacy}
   \end{figure}
Next, let us approximate $\phi$ by an analytic map. Define $\gamma^\eta := \T \times \croset{-1+2\eta,1-2\eta}$ and observe that $\phi$ is compactly supported in $A(\gamma^\eta)$ and $\eta < d(\gamma^\eta,\partial\A)$. Then, for any $\mathcal{V}$ small neighbourhood of $\rest{\phi}{\A\setminus \gamma^\eta}$ in the smooth compact-open topology of $C^\infty(\A \setminus \gamma^\eta,\T\times \R)$, and any $\rho_1 > 1$ large, by \cref{an:thm_approx} there exists a biholomorphism $\Psi \in \sympan{\rho_1,\eta}{\A_\infty}$ such that $\rest{\Psi}{\A\setminus \gamma^\eta}$ belongs to $\mathcal{V}$. Let us define the conjugate
$$g := \conj{\Psi}{f}.$$
Now, we would like to apply \cref{an:thm_approx} on $g$. To this end we need $g$ to be compactly supported in $A(\gamma')$, and we know that $f$ is compactly supported in $A(\gamma)$. Therefore, because $\gamma' \subset \A \setminus \gamma^\eta$ and $\phi(A(\gamma')) = A(\gamma)$, we choose the neighbourhoods $\mathcal{V}$ of $\rest{\phi}{\A\setminus \gamma^\eta}$ to be sufficiently small so that $\supp(f) \subset \Psi(A(\gamma'))$. Hence $g$ is compactly supported in $A(\gamma')$. See \cref{fig:approx_phi}.

\begin{figure}
\begin{subfigure}{0.48\textwidth}
    \centering
    \includegraphics[width = \linewidth]{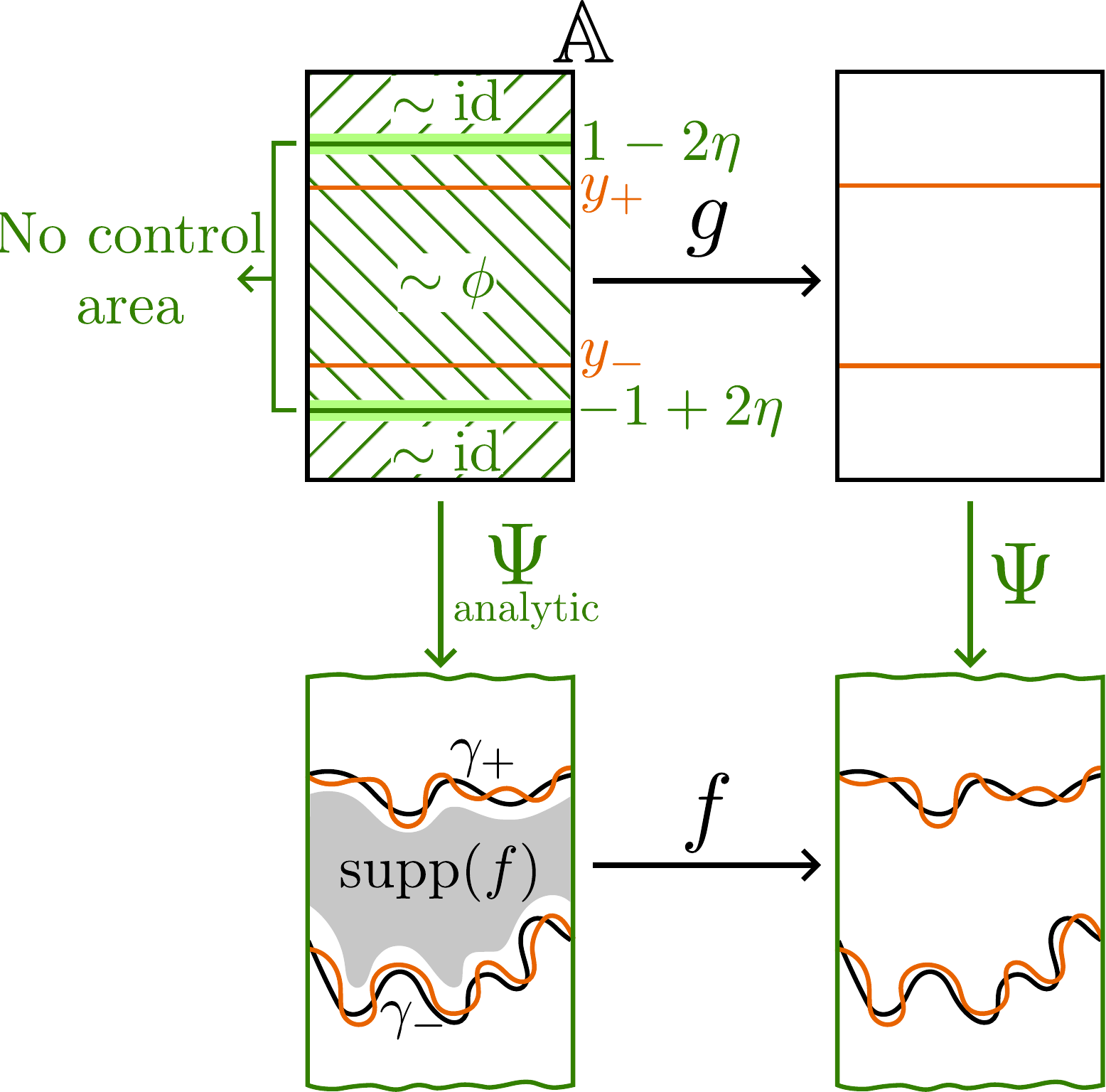}
    \caption{Analytic approximation of $\phi$ by $\Psi$.}
    \label{fig:approx_phi}
\end{subfigure}
 \begin{subfigure}{0.48\textwidth}
    \centering
    \includegraphics[width = \linewidth]{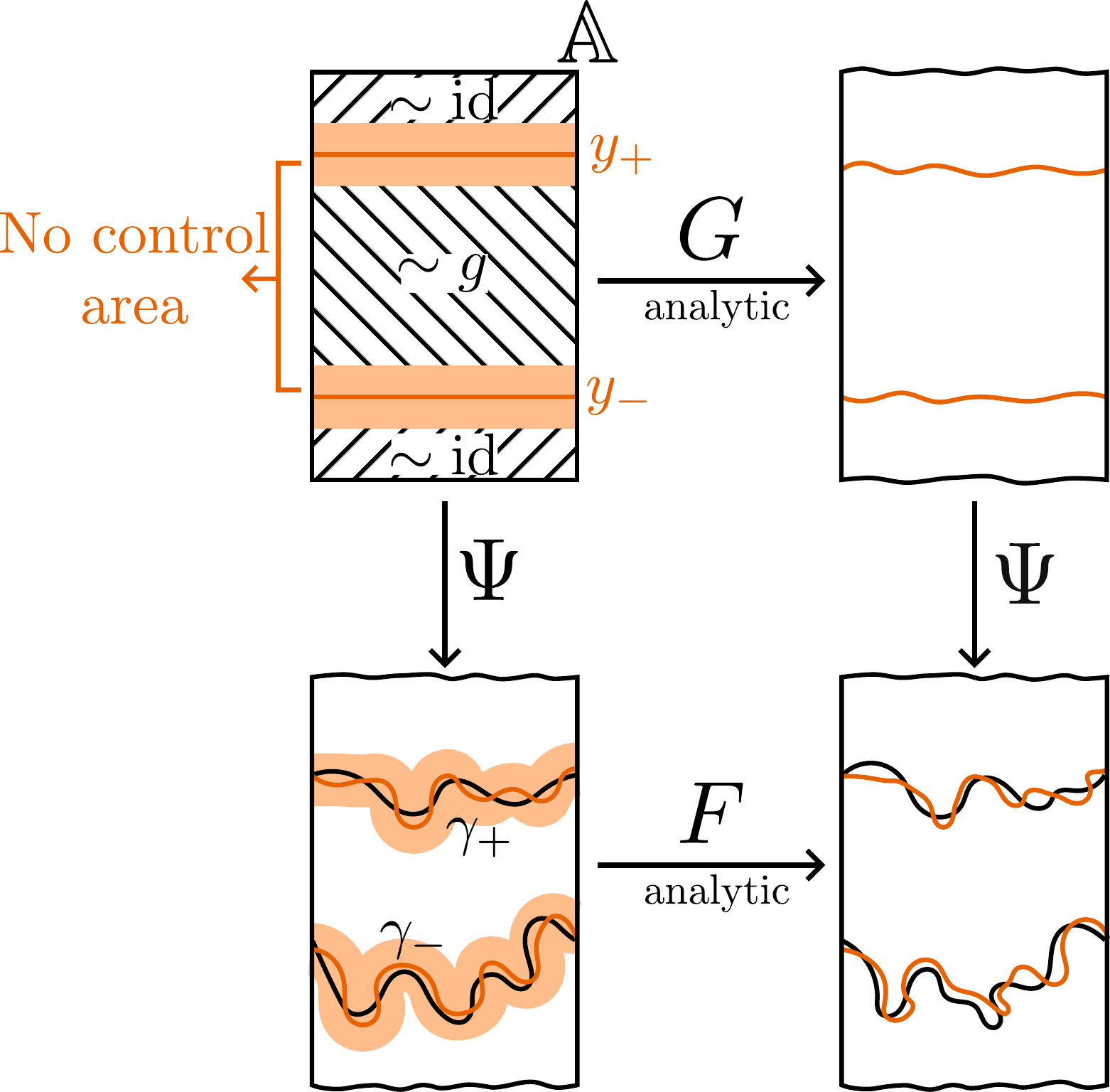}
    \caption{Analytic approximation of $g$ by $G$.}
    \label{fig:an_ap_conj}
\end{subfigure}
\end{figure}

Next, let us define the neighbourhood $\mathcal{W}$ of $\rest{g}{\A \setminus \gamma'}$ to apply \cref{an:thm_approx}.

Up to reducing the neighbourhood $\mathcal{V}$, we can assume that $\gamma'$ is close enough to $\Psi^{-1}(\gamma)$ so that there exists a neighbourhood $\mathcal{U}'$ of $\rest{g}{\T \times \R \setminus \gamma'}$ in the smooth compact-open topology of $C^\infty(\T \times \R \setminus \gamma',\T\times \R)$ such that the following holds:
\begin{equation}\label{eq:U'_thm}
\forall G: \T \times \R \to \T \times \R \, , \, \rest{G}{\T \times \R \setminus \gamma'} \in \mathcal{U}' \Longrightarrow  \rest{\iconj{\Psi}{G}}{\A \setminus \gamma} \in \mathcal{U}.
\end{equation}
Moreover, as $\rest{g}{\T \times \R \setminus \A} = \id$, there exists a neighbourhood $\mathcal{W}$ of $\rest{g}{\A \setminus \gamma'}$ in the smooth compact-open topology of $C^\infty(\A \setminus \gamma',\T\times \R)$ such that, for any $\rho_2>1$ sufficiently large, the following holds by Cauchy's estimates in $K_{\rho_2}$:
\begin{equation}\label{eq:W_thm}
\forall G\in \sympan{\rho_2}{\A_\infty} \, , \, \rest{G}{\A \setminus \gamma'} \in \mathcal{W} \Longrightarrow  \rest{G}{\T \times \R \setminus \gamma'} \in \mathcal{U}'.
\end{equation}

Now, because $g$ is compactly supported in $A(\gamma')$ and $\eta < d(\gamma',\partial\A)$, \cref{an:thm_approx} yields, for any $\rho_2 >1$ sufficiently large, a biholomorphism $G \in \sympan{\rho_2,\eta}{\A_\infty}$ such that $\rest{G}{\A\setminus \gamma^\eta}$ belongs to $\mathcal{W}$. Let us consider the biholomorphism
\begin{equation}\label{eq:def_F}
F := \iconj{\Psi}{G} \in \sympan{}{\A_\infty}.
\end{equation}
See \cref{fig:an_ap_conj}.\\

Since $\rest{G}{\A \setminus \gamma'}$ belongs to $\mathcal{W}$ and $G\in \sympan{\rho_2}{\A_\infty}$, $\rest{G}{\T \times \R \setminus \gamma'}$ belongs to $\mathcal{U}'$ by \cref{eq:W_thm}. Thus, by \cref{eq:U'_thm,eq:def_F}, $\rest{F}{\A \setminus \gamma}$ belongs to $\mathcal{U}$. It now remains to show that $F$ belongs to $\sympan{\rho}{\A_\infty}$ by using \cref{lem:conj_ham}.

Let $\rho_k' := 2^k \max \croset{2/\eta , \rho}$ for $k \in \croset{1,2,3}$. First, we choose $\rho_1$ sufficiently large such that $\Psi \in \sympan{\rho_1,\eta}{\A_\infty}$ implies that $\sup_{x \in K_{\rho, \eta/2}} d(\Psi(x),x) < 1/\rho_1'$ and $\sup_{x \in K_{\rho, \eta/2}} d(\Psi^{-1}(x),x) < 1/\rho_3'$ by continuity of the inversion. Next, we choose $\rho_2$  sufficiently large so that $G \in \sympan{\rho_2,\eta}{\A_\infty}$ implies that $\sup_{x \in K_{\rho, \eta/2}} d(G(x),x) < 1/\rho_2'$. Therefore, by \cref{lem:conj_ham}, the composition $F = \iconj{\Psi}{G}$ belongs to $\sympan{\rho}{\A_\infty}$. This concludes the proof in the case $\gamma \subset \A \setminus \partial \A$.\\

\textbf{\underline{Case $\gamma \cap \partial \A \neq \emptyset$ or $\supp(f) \subset A(\gamma)$:}}\\

In view of \cref{rk:gam_bord}, there exist $\kappa>0$ and a bicurve $\gamma'$ inside $B(\gamma,\kappa) \setminus (\partial\A \cup \supp(f))$ so that $\gamma' \subset \A \setminus \partial \A$, in particular $\supp(f) \subset A(\gamma')$. 
Moreover, there exists a neighbourhood $\mathcal{U}'$ of $\rest{f}{\A \setminus \gamma'}$ in the smooth compact-open topology of $C^\infty(\A \setminus \gamma',\T\times\R)$ such that the following holds:
$$\forall F : \T\times \R \to \T \times \R \, , \, \rest{F}{\A \setminus \gamma'} \in \mathcal{U}' \Longrightarrow \rest{F}{\A \setminus \gamma}\in\mathcal{U}.$$
We therefore apply the theorem to $\gamma'$ and this neighbourhood $\mathcal{U}'$ which yields the desired biholomorphism $F \in \sympan{\rho}{\A_\infty}$ such that $\rest{F}{\A \setminus \gamma}\in \mathcal{U}$.

\end{proof}

\subsection{Complexification and structures of $\M$}\label{sec:complex}

Before applying this approximation theorem to deform complex structures, we first introduce the complexification of our surfaces and clarify the notion of complex structure used in this setting.

\subsubsection{Complexification of $\M \in \croset{\A ,\Di, \Sp}$}

\paragraph{Complexification of the cylinder $\A$.}

We consider the following complexifications for $\rho >1$:
\newglossaryentry{07A1}{sort = {1Space}, name = {$\A_\rho$, $\A_\infty$, $\check{\A}_\rho$}, description = {Complexifications of $\A$ and $\check{\A}$}}
\glsadd{07A1}
\newglossaryentry{07A2}{sort = {1Space}, name = {$\tilde{\A}$}, description = {2-torus with a $\Z_2$-structure}}
\glsadd{07A2}
\newglossaryentry{07A3}{sort = {1Space}, name = {$\tilde{\A}_\rho$}, description = {Complexifications of $\tilde{\A}$}}
\glsadd{07A3}
$$ \A_\infty := \C/\Z \times \C, \hspace{1cm} \A_\rho := \croset{(\theta,y) \in \A_\infty, \lvert e^{i\theta} \rvert^2 + \lvert e^{-i\theta} \rvert^2 + \lvert y \rvert ^2 \leq 3\rho},$$
$$\check{\A}_\rho := \croset{(\theta,y) \in \A_\rho , \lvert \Re(y) \rvert < 1}.$$

Then, we recall that $\A$ is the quotient of $\tilde{\A} := \T \times \R/4\Z$ by the involution 
$$\tau : (\theta,y) \mapsto (\theta,2-y).$$
This involution extends holomorphically to $\C/\Z \times \C/4\Z$ and we define the complexification $\tilde{\A}_\rho := p(\A_\rho) \cup \tau \circ p(\A_\rho)$ where $p$ is the projection $\A_\infty \rightarrow \C/\Z \times \C/4\Z$. We notice that $\tilde{\A}_\rho$ is invariant under $\tau$.\\

Then, we endow $\A$ and $\tilde{\A}$ with the symplectic form $\Omega := \frac{1}{2}d\theta \wedge dy$. It extends holomorphically to a holomorphic form $\Omega_0$ on $\A_\infty$ and $\C/\Z \times \C/4\Z$. We observe that $\tau^* \Omega_0 = -\Omega_0$, 
therefore if we fix the group homomorphism $\Gamma : \tau^i \mapsto (-1)^i$, then for every $i$ we have $(\tau^i)^*\Omega_0 = \Gamma(\tau^i)\Omega_0$. This endows $(\tilde{\A}_\rho,\Omega_0)$ with a $\Z_2$-structure.\\

\paragraph{Complexification of the sphere $\Sp$.}

We consider the following complexifications for $\rho >1$:
\newglossaryentry{07Sp}{sort = {1Space}, name = {$\Sp_\rho$, $\Sp_\infty$, $\check{\Sp}_\rho$}, description = {Complexifications of $\Sp$ and $\check{\Sp}$}}
\glsadd{07Sp}
$$ \Sp_\infty := \croset{z \in \C^3, \isum{i=1}{3}z_i^2 =1}, \hspace{1cm} \Sp_\rho := \croset{z \in \Sp_\infty, \isum{i=1}{3}\lvert z_i \rvert \leq \rho},$$
$$\check{\Sp}_\rho := \croset{z \in \Sp_\rho , \lvert \Re(z_3) \rvert < 1}.$$

By analytically extending the square root to $\C\setminus \R^-$, we extend the axial projection of the cylinder onto the sphere into the following biholomorphism:
$$\begin{array}{rrcl}
\pi : &\check{\A}_\infty &\rightarrow &\check{\Sp}_\infty\\
&(\theta,y) &\mapsto &\left( \sqrt{1-y^2}\cos(2\pi \theta) , \sqrt{1-y^2}\sin (2\pi \theta ),y \right)
\end{array},$$
with inverse $z \mapsto \left( \frac{1}{2i\pi} \log \left( \frac{z_1 +iz_2}{\sqrt{1-z^2_3}} \right),z_3 \right)$.\\

However, since $\rest{\pi}{\check{\A}}$ is a symplectomorphism and the symplectic form of $\check{\A}$ extends holomorphically to a symplectic form $\Omega_0$ on $\check{\A}_\infty$, then $\pi$ extends the symplectic form of $\Sp$ to $\check{\Sp}_\infty$ holomorphically. By symmetry, it also extends holomorphically to any image of $\check{\Sp}_\infty$ by an element of $SO_3$. As $SO_3(\check{\Sp}_\infty)  = \Sp_\infty$  and $\Sp_\infty$ is simply connected, being diffeomorphic to $T\Sp$, the symplectic form $\Omega$ of $\Sp$ extends holomorphically to $\Sp_\infty$ as a holomorphic $2$-form denoted $\Omega_0$.\\

\paragraph{Complexification of the disk $\Di$.} 

We consider the following complexifications of the disk for $\rho>1$.
\newglossaryentry{08B}{sort = {1Space}, name = {$\B_\rho$}, description = {Complex ball in $\C^3$}}
\glsadd{08B}
\newglossaryentry{08D}{sort = {1Space}, name = {$\check{\Di}_\rho$}, description = {Complexifications of $\check{\Di}$}}
\glsadd{08D}
$$\B_\rho := \croset{ z \in \C^2 : \isum{i=1}{2} \lvert z_i \rvert^2 \leq \rho} \quad\text{and}\quad \check{\Di}_\rho := \croset{ z \in \B_\rho \ : \ 0 < \Re(z_1^2 + z_2^2 ) < 1}.$$

On these domains, the canonical symplectic form $\Omega := \frac{1}{\pi} dx_1 \wedge dx_2$ of $\Di$ extends holomorphically into $\Omega_0 := \frac{1}{\pi}dz_1 \wedge dz_2$.\\

We also holomorphically extend the polar coordinates as follows:

$$\begin{array}{rrcl}
\pi : &\croset{ (\theta,y) \in \A_\infty \, : \, \Re(y) > -1} &\rightarrow &\croset{ z \in \C^2 \, : \, \Re(z_1^2 + z_2^2) > 0}\\
&(\theta,y) &\mapsto & \sqrt{\frac{y+1}{2}}\left( \cos(2\pi \theta) , \sin (2\pi \theta ) \right)
\end{array}.$$

Whose inverse is the map $z \mapsto \left( \frac{1}{2i\pi} \log \left( \frac{z_1 + iz_2}{\sqrt{z_1^2 + z_2^2}} \right)	 , 2(z_1^2 + z_2^2) -1 \right)$.
In particular $\pi$ sends $\check{\A}_\infty$ onto $\check{\Di}_\infty$.\\

We now describe the $\Z_2$-structure on $\Di$. Let $U$ be the neighbourhood of $\partial \Di$ defined by:

$$U := \croset{x \in \R^2 \, : \, 0 < x_1^2 + x_2^2 < 2}.$$

We observe that the inverse image of $U$ by $\pi$ is $\T \times (-1,3)$, which is left invariant by the involution $\tau$. Thus conjugating $\tau$ by $\pi$ yields the following involution on $U$:

$$\psi := x \in U \mapsto \sqrt{2 - \Vert x \Vert^2}\frac{x}{\Vert x \Vert}.$$

Moreover, as $\pi$ is symplectic and $\tau$ anti-symplectic, it also satisfies $\psi^* \Omega = -\Omega$. Observe also that $\psi$ leaves $\partial \Di$ invariant.\\

Then, the real analytic symplectic structure on $\Di \cup U$ induces canonically a real analytic symplectic structure on the following quotient (see \cite[Cor. 2.4]{berger_coexistence_2023}):
\newglossaryentry{09D}{sort = {1Space}, name = {$\tilde{\Di}$}, description = {Symplectic sphere with a $\Z_2$-structure}}
\glsadd{09D}

$$\tilde{\Di} := (\Di \cup U) \times \Z_2 / \sim,$$

where $\sim$ is an equivalence relation defined such that for $x$ in $U$ we have $(x,i) \sim (\psi(x),i+1)$. We observe that $\tilde{\Di}$ is a symplectic sphere whose quotient by the involution $\tau : (x,i) \mapsto (x,i+1)$ is the disk $\Di$ (we use the same notation $\tau$ for the cylinder and the disk). As for the cylinder, we observe that $\tau^*\Omega = -\Omega$, this provides a $\Z_2$-structure to $(\tilde{\Di},\Omega)$.\\

Likewise, we propose a complexification of $\tilde{\Di}$. Once again we analytically extend $\sqrt{\cdot}$ to $\C \setminus \R_-$ in order to extend $\psi$ analytically as follows:
$$\Psi : z \in \croset{(z_1,z_2) \in \C^2 \, : \, 0 < \Re( z_1^2 + z_2^2) < 2} \mapsto z \frac{\sqrt{2 - z_1^2 -z_2^2}}{\sqrt{z_1^2 + z_2^2}}.$$

This map is still an involution. We define
\newglossaryentry{10D}{sort = {1Space}, name = {$\Di_\rho$, $\tilde{\Di}_\rho$}, description = {Complexifications of $\Di$ and $\tilde{\Di}$}}
\glsadd{10D}
$$\Di_\rho = \croset{z \in \B_\rho \, : \, \Re(z_1^2 + z_2^2) <2} \quad \text{and} \quad \U_\rho := \croset{z \in \B_\rho \, : \, 0<\Re(z_1^2 + z_2^2) <2}.$$

We obtain the following complex extension of $\tilde{\Di}$:
$$\tilde{\Di}_\rho := \Di_\rho \times \Z_2 / \sim \; \text{ with} \; (z,i) \sim (\Psi(z),i+1) \text{ when } z\in \U_\rho.$$
The involution $\tau$ also extends to $\tilde{\Di}_\rho$ and satisfies $\tau^* \Omega_0 = -\Omega_0$. Hence we have a $\Z_2$-structure on $(\tilde{\Di}_\rho , \Omega_0)$.

\subsubsection{Complex and symplectic structure}

In this section, we give definitions and examples of complexifications and structures on a complex manifold. Then, we state results from \cite[Section 5]{berger_analytic_2024} that are used to prove the AbC$^\star$ principle.\\

\newglossaryentry{01W}{sort = {5complex}, name = {$\V$, $\W$}, description = {Complex manifolds}}
\glsadd{01W}
Let $\V$ be a complex manifold, we denote by $T\V$ its complex tangent bundle and $T^\R \V$ its underlying real tangent bundle. 
The multiplication by $i$ on $T\V$ induces an automorphism $J_0$ from $T\V$ to itself satisfying $J_0^2 =-1$, this defines a \underline{\emph{complex structure}} of $\V$.

\begin{remark}
$\C^k$, $\A_\infty$, $\tilde{\A}_\infty$, $\Di_\infty$, $\tilde{\Di}_\infty$ and $\Sp_\infty$ have canonical structures of complex manifolds denoted indistinctly by $J_0$.
\end{remark}
\vspace{-0.6cm}
\newglossaryentry{02J}{sort = {5complex}, name = {$J$}, description = {Complex structure of a complex manifold}}
\glsadd{02J}
\begin{defin}
A smooth section $J$ of $T^\R\V^* \otimes T^\R\V$ such that $J^2 = -1$, as a linear map from $T^\R \V$ to itself, is called \underline{an almost complex structure} on $\V$.\\
Such a structure is called \underline{integrable} if there is a $C^\infty$ atlas $(\phi_\alpha)_\alpha$ of $\V$ such that $D\phi_\alpha \circ J = i D\phi_\alpha$. In this case the coordinate changes are holomorphic.
\end{defin}

\begin{remark}
Given an integrable complex structure on $\V$, we deduce a complex structure on $\V$ and vice versa. Therefore we identify the integrable complex structures with the complex structures.
\end{remark}

\begin{example}
For every embedding $h : \W \hookrightarrow \V$, the section $h^*J_0 = Dh^{-1} \circ J_0 \circ Dh$ is a complex structure on $\W$.
\end{example}

We endow the space of complex structures on $\V$ with the topology of smooth sections of $T^\R \V^* \otimes T^\R \V$.\\

We then define the notion of real analytic structure and complexification.
\newglossaryentry{03M_bul}{sort = {5complex}, name = {$(\M_\bullet ,J)$}, description = {Complexification of a real manifolds $\M$}}
\glsadd{03M_bul}
\newglossaryentry{04sigma}{sort = {5complex}, name = {$\sigma$}, description = {Real structure of a complex manifold $(\M_\bullet,J)$}}
\glsadd{04sigma}
\begin{defin}
A \underline{real analytic structure} on a manifold is a maximal atlas $\mathcal{A}^\omega = (\phi_\beta)_\beta	$ where coordinate changes are real analytic.\\

A \underline{complexification} of a real analytic manifold $\M$ is a complex manifold $(\M_\bullet ,J)$ such that $\M \subset \M_\bullet$ and for all $z$ in $\M$, there exists a $J$-holomorphic chart $\psi_\alpha$ of a neighbourhood $V_\alpha$ of $z$ in $\M_\bullet$ such that the restriction of $\psi_\alpha$ to $V_\alpha \cap \M$ is a real analytic chart for $\M$.\\

A \underline{real structure} on a complex manifold $(\M_\bullet , J)$ is an involution $\sigma$ from $\M_\bullet$ to itself which is anti-holomorphic with respect to $J$, i.e. $\sigma^* J = -J$.
\end{defin}

In our situation, we endow $\M_\bullet \in \croset{\A_\rho , \Sp_\rho, ...}$ with the canonical real structure:
$$\sigma : z =(z_j)_j \mapsto \bar{z} = (\bar{z_j})_j.$$

For $J$ a complex structure on $\M_\rho$ such that $\sigma$ is anti-holomorphic with respect to $J$, there exists a real analytic structure on $\M$ for which $(\M_\rho , J)$ is a complexification (see \cite[Proposition 5.4]{berger_analytic_2024} for a statement of this result). Hence we can define the following.

\begin{defin}\label{def:comp}
A complex structure $J$ on $\M_\rho$ such that $\sigma^*J = -J$ is called a \underline{ complexification of} \underline{a real analytic structure} on $\M$.
\end{defin}

We have the notable following example used for the principle.

\begin{example}
If $h : \M_\rho \hookrightarrow \M_\infty$ is a smooth embedding that commutes with $\sigma$, then $(\M_\rho, h^*J_0)$ is a complexification of a real analytic structure of $\M$.
\end{example}

We now incorporate symplectic forms.

\begin{defin}
Let $\M$ be a real surface endowed with a smooth symplectic form $\Omega$ and $(\M_\bullet,J)$ be a complexification of $\M$. We say that $\Omega$ \underline{extends $J$-holomorphically} to $\M_\bullet$ if there is a $J$-holomorphic extension $\Omega_\bullet$ of $\Omega$ to $\M_\bullet$. That is $\rest{\Omega_\bullet}{T\M} = \Omega$ and in any $J$-holomorphic charts, the form $\Omega_\bullet$ is of the form $\Psi(z_1,z_2) dz_1 \wedge dz_2$ for a holomorphic function $\Psi$.
\end{defin}

As discussed earlier in this section, when $\M \in \croset{\A , \Di , \tilde{\A}, \tilde{\Di} ,\Sp}$ we denote by $\Omega_0$ the canonical $J_0$-holomorphic extension to $\M_\infty$ of the symplectic form $\Omega$ on $\M$.\\

We endow the space of symplectic $2$-forms on a compact manifold with the smooth topology.

We can finally extend the last example by adding the symplectic form.

\begin{example}\label{ex:holo_pull}
If $h : \M_\rho \hookrightarrow \M_\infty$ is a smooth embedding commuting with $\sigma$ whose restriction to $\M$ is symplectic, then $\Omega$ extends holomorphically with respect to $h^*J_0$ to $\M_\rho$.
\end{example}

Now, let us introduce the main results that we will use. The first one is on the convergence of complex structures and symplectic forms and comes from \cite[Theorem 5.3 and 5.10]{berger_analytic_2024}.

\begin{theorem}\label{thm:lim_almost_comp}
Let $(\M, \Omega)$ be a symplectic compact real surface, and $(\M_\rho,J_n)_n$ be a sequence of complexification of $\M$ such that $(J_n)_{n\in \N}$ converges to an almost complex structure $J$. Then $J$ is integrable and $(\M_\rho ,J)$ is a complexification of $\M$.\\
Assume that $\Omega$ has a $J_n$-holomorphic extension $\Omega_n$ on $\M_\rho$ such that $(\Omega_n)_n$ converges to a 2-form $\Omega'$. Then $\Omega'$ is a $J$-holomorphic extension of $\Omega$ to $\M_\rho$.  
\end{theorem}

We finish this section by a compatibility result between $J$-holomorphy, symplecticity, and real-analytic symplectic structures.\\

We consider $\rho > \rho' >1$ such that $(\M_{\rho'},J)$ is a smaller complexification.

\begin{defin}
A smooth function $f$ from $\M_{\rho'}$ to $\M_{\rho}$ is said to be a \underline{$J$-holomorphic extension} of a real map if $f(\M) \subset \M$ and $f$ is $J$-holomorphic, i.e. $Df\circ J = J \circ Df$.
\end{defin}

Here comes the result from \cite[Proposition 5.13 and Corollary 5.14]{berger_analytic_2024}.

\begin{theorem}\label{prop:real_an_symp}
Let $\rho > \rho' > 1$, $(\M_\rho , J)$ be a complexification of $\M  \in \croset{\A, \Di, \Sp}$. Assume that $\Omega$ extends $J$-holomorphically to $\M_\rho$. Let $f\in C^\infty (\M_{\rho'},\M_\rho)$ be a $J$-holomorphic extension of a real symplectomorphism of $\M$.

If $\M = \Sp$, then $\rest{f}{\M}$ is a real analytic symplectomorphism for a $C^\infty$-compatible real analytic symplectic structure.

If $\M = \A$ or $\Di$, assume that $J$ coincides with $J_0$ on $\M_\rho \setminus \M_{\rho'}$ and $f$ coincides with a rotation $R_\alpha$ on $\M_\rho \setminus \M_{\rho'}$. Then $\rest{f}{\M}$ is a real analytic $\Z_2$-symplectomorphism for a $C^\infty$-compatible structure.
\end{theorem}

\subsection{Approximation modulo deformation}\label{sec:approx_mod}

Now that we have an approximation theorem outside bicurves and the notion of complex structures, we can use this theorem to deform the both of them.\\
Here is the Approximation Modulo Deformation Theorem used to obtain the principle.

\begin{theorem}\label{thm:app_mod_def}
Let $\gamma \in \Gamma_2(\A)$, let $h_0$ be in $\sympo{\infty}{\A}$ such that $\rest{h_0}{O(\gamma)} = \id$, and let $\mathcal{U}$ be a neighbourhood of $\rest{h_0}{\A\setminus \gamma}$ in the smooth compact-open topology of $C^\infty(\A \setminus \gamma,\T\times \R)$. There exists $\eta > 0$ such that for every $\sigma$-invariant compact set $\W \subset \A_\infty$  with smooth boundary, any $C^\infty$ neighbourhood $\mathcal{W}_J$ of $\rest{J_0}{\W}$ and $\mathcal{W}_\Omega$ of $\rest{\Omega_0}{\W}$, there exists $h : \W \rightarrow \A_\infty$, a smooth diffeomorphism onto its image, that satisfies:
\begin{itemize}
\item[($\mathcal{P}1$)] $h(\A) = \A$, $\det \rest{Dh}{\A} = 1$, and $\rest{h^{-1}}{\A \setminus \gamma} \in \mathcal{U}$.
\item[($\mathcal{P}2$)] $h$ is supported in $\croset{(\theta,y) \in \W \, : \, \lvert \Re(y) \rvert < 1-\eta }$.
\item[($\mathcal{P}3$)] $h^* J_0$ belongs to $\mathcal{W}_J$ and $h^* \Omega_0$ belongs to $\mathcal{W}_\Omega$.
\item[($\mathcal{P}4$)] $h$ commutes with $\sigma$.
\end{itemize}
\end{theorem}

We will prove this theorem in the Appendix \ref{an:approx_def} by adapting Berger's proof with our approximation theorem.

The definition below follows from this theorem.

\begin{defin}
Such an $h_0$ is said to be \underline{approximated modulo structural deformation}.
\end{defin}

Now, using a lifting $(\theta,y) \mapsto (q\cdot \theta ,y)$, the following corollary is an immediate consequence of \cref{thm:app_mod_def}.

\begin{coro}\label{coro:app_mod_def}
If $h_0$ commutes with a rotation $R_{p/q}$ and the bicurve $\gamma$ is invariant under $R_{p/q}$, then the approximation $h$ modulo structural deformation can be chosen commuting with $R_{p/q}$.
\end{coro}

\subsection{AbC$^\star$ realization on deformed real analytic surfaces}\label{sec:real}

The main goal of this section is to prove the following theorem from which the AbC$^\star$ principle will follow.

\begin{theorem}\label{thm:real_sphere}
Let $\rho > \rho' > 1$. For every $C^\infty$-AbC$^\star$ scheme $(U,\nu)$ on $\M \in \croset{\A, \Di, \Sp}$, there exists a complexification $(\M_\rho,J)$ on which $\Omega$ extends $J$-holomorphically and there exists a $J$-holomorphic diffeomorphism $f : \M_{\rho'} \hookrightarrow \M_\rho$ which extends a real symplectomorphism of $(\M, \Omega)$ realizing $(U,\nu)$.\\
Moreover, on the complement of $\check{\M}_\infty$, the map $f$ equals a rotation and $J=J_0$.
\end{theorem}

Firstly, we prove this theorem in the case of the cylinder. It is the simplest case since we can directly apply the Approximation modulo deformation which is stated on this surface.

\paragraph{AbC$^\star$ realizations on deformed real analytic symplectic cylinder}
\ \\

First, we recall that the space of smooth embeddings $ \emb{\infty}{\A_\rho,\A_\infty}$ is metrizable by a distance denoted $d$. So does the space of complex differentiable $2$-forms on $\A_\rho$ and the space of linear automorphisms of $T^\R \A_\rho$ by distances also denoted $d$.\\

Then, we will obtain the theorem by constructing the symplectomorphism $f$ with an induction. To that extent, we apply the Approximation modulo deformation at every step of the scheme. Therefore, the limit $f$ obtained is holomorphic for a complex structure, which is the limit of a sequence of complex structures that we have deformed at each step. In addition, $f$ is symplectic for a holomorphic symplectic form also obtained by deforming at each step our original form.\\

\newglossaryentry{4EmbA}{sort = {3Func Space}, name = {$\emb{\infty}{\A_\rho,\A_\infty}$}, description = {Space of smooth embedding from $\A_\rho$ to $\A_\infty$}}\glsadd{4EmbA}
Therefore, let us show the existence of sequences $\alpha_n \in \Q/\Z$, $h_n$ with $h_n^{-1} \in \emb{\infty}{\A_\rho,\A_\infty}$ and $f_n \in \emb{\infty}{\A_{\rho'},\A_\rho}$ -- where $\rho' < \rho$ -- such that:

$$\alpha_0 =0 \; , \; h_0 = \rest{\id}{\A_\rho} \quad\text{and}\quad f_0 = \rest{\id}{\A_{\rho'}}.$$
We also require that, for $n\geq 1$, $\alpha_n$, $h_n$ and $f_n$ satisfy:
\begin{enumerate}
\item[$(I_n)$] $h_n(\check{\A}) = \check{\A}$, $h_n$ is symplectic on $\check{\A}$ and $(\rest{h_n}{\A},\alpha_n)$ respects the AbC$^\star$ scheme $(U,\nu)$:
$$\rest{h_n}{\A} \in U(\rest{h_{n-1}}{\A},\alpha_{n-1}) \quad , \quad \iconj{\rest{h_{n}}{\A}}{R_{\alpha_{n-1}}} = \iconj{\rest{h_{n-1}}{\A}}{R_{\alpha_{n-1}}},$$
$$ \text{and} \quad 0 < \lvert \alpha_n - \alpha_{n-1} \rvert < \nu(\rest{h_n}{\A},\alpha_{n-1}).$$ 
\item[$(II_n)$] $R_{\alpha_n} \circ h_n^{-1} (\A_{\rho'}) \subset \operatorname{int} \left( h_n^{-1}(\A_\rho ) \right)$, so the composition $f_n := \rest{\iconj{h_n}{R_{\alpha_n}}}{\A_{\rho'}}$ is well defined. Furthermore, $d(f_n,f_{n-1}) < 2^{-n}$.
\item[$(III_n)$] $J_n = {h_n^{-1}}^* J_0$ and $\Omega_n = {h_n^{-1}}^* \Omega_0$ satisfy:
$$d(J_n,J_{n-1}) < 2^{-n} \quad\text{and}\quad d(\Omega_n,\Omega_{n-1}) < 2^{-n}.$$
\item[$(IV_n)$] $h_n$ commutes with $\sigma$.
\item[$(V_n)$] The support of $h_n^{-1}$ is included in $\check{\A}_\rho$.
\end{enumerate}

By applying the Approximation Modulo Deformation on a map $\widetilde{h}$ related to $h_n$ by the AbC$^\star$ scheme at a given step $n$, we obtain the following lemma which provides the existence of such sequences.

\begin{lemma}\label{lem:abc_real_cyl}
There exist sequences $(\alpha_n)_n$, $(h_n)_n$, and $(f_n)_n$ such that $(I_n \cdots V_n)$ is satisfied for every $n \geq 1$.
\end{lemma}

We provide a sketch of proof of this lemma below and prove it in the Appendix \ref{an:abc_real}.

\begin{proof}[Sketch of proof]
Let $\alpha_n$, $h_n$ and $f_n$ be such that they satisfy $(I_n \cdots V_n)$. 
Let us define $\alpha_{n+1}$, $h_{n+1}$ and $f_{n+1}$ so that they satisfy $(I_{n+1} \cdots V_{n+1})$. 
First, by \cref{def:abc_star} of an AbC$^\star$ scheme, there exists a symplectomorphism $\widetilde{h} \in \sympo{\infty}{\A}$ commuting with $R_{\alpha_n}$ and such that $h_n \circ \widetilde{h}$ belongs to $U(\rest{h_n}{\A},\alpha_n)$. 
Moreover, $U(\rest{h_n}{\A},\alpha_n)$ belongs to $\mathcal{T}^\infty_\gamma$ for a bicurve $\gamma \in \Gamma_2(\A)$ invariant under $R_{\alpha_n}$ and $\widetilde{h}$ satisfies $\rest{\widetilde{h}}{O(\gamma)} = \id$. 
Then, we approximate $\widetilde{h}$ modulo structural deformation by a map $h$ by applying \cref{thm:app_mod_def} with the bicurve $\gamma$, and well chosen neighbourhood $\mathcal{U}$, $\sigma$-invariant compact set $\W$ and neighbourhoods $\mathcal{W}_J$ and $\mathcal{W}_\Omega$.  
We now define $h_{n+1} := h_n \circ h^{-1}$, $\alpha_{n+1}$ close to $\alpha_n$, and $f_{n+1} = \rest{\iconj{h_{n+1}}{R_{\alpha_{n+1}}}}{\A_{\rho'}}$. 
Therefore, with  properties $(I_n)$, $(II_n)$, $(IV_n)$ and $(V_n)$ and with the properties satisfied by the approximation modulo deformation $h$, we obtain properties $(I_{n+1} \cdots V_{n+1})$ for $\alpha_{n+1}$, $h_{n+1}$ and $f_{n+1}$.
\end{proof}

Now we can obtain \cref{thm:real_sphere} on the cylinder.

\begin{proof}[Proof of \cref{thm:real_sphere} on $\A$]
Let us consider the sequence $(f_n)_n$ given by the latter lemma. By $(II_n)$, they converge toward a smooth embedding $f$ from $\A_{\rho'}$ to $\A_\rho$. Moreover, by $(I_n)$ its restriction to $\A$ is a symplectomorphism of $\A$ which is a realization of the AbC$^\star$ scheme $(U,\nu)$.

By $(III_n)$ and \cref{thm:lim_almost_comp}, the sequence $(J_n)_n$ converges to a complex structure $J$. Moreover, by symplecticity of the $h_n$, the forms $\Omega_n$ are $J_n$-holomorphic extensions of $\Omega$ to $\A_\rho$. Then, by $(III_n)$ and \cref{thm:lim_almost_comp}, the sequence $(\Omega_n)_n$ converges toward $\Omega'$, a $J$-holomorphic extension of $\Omega$ to $\A_\rho$.

Also, by $J_0$-holomorphy of rotations, the following holds on $\check{\A}_{\rho'}$:
$$f_n^* J_n = (h_n^{-1} \circ f_n)^*J_0 =(R_{\alpha_n} \circ h_n^{-1} )^* J_0 = (h_n^{-1})^* J_0 = J_n.$$

Hence, at the limit, $f^*J = J$ on $\A_{\rho'}$, since on $\A_{\rho'} \setminus \check{\A}_{\rho'}$ we have $J = J_0$ and $f = R_\alpha$ (where $R_\alpha$ is the limit of $(R_{\alpha_n})_n$ on $\A_{\rho'} \setminus \check{\A}_{\rho'}$). As a result, $f$ is a $J$-holomorphic extension to $\A_{\rho'}$ of the symplectomorphism $\rest{f}{\A}$.
\end{proof}

We now turn to the harder cases of the disk and the sphere.

\paragraph{AbC$^\star$ realizations on deformed real analytic symplectic sphere and disk}
\ \\

Let $\M = \Di$ or $\Sp$. We recall that we consider $\rho > \rho' >1$.\\

Here, we can not apply directly the Approximation Modulo Deformation and then conjugate with $\pi$ to transpose the previous demonstration to $\M$. 
Indeed, even if the approximation theorem ensures that the map $h: \A_\rho \to \A_\infty$ obtained is compactly supported in $\check{\A}_\rho$, it does not prevent its image from intersecting the complement of $\check{\A}_\infty$ and leave $\A_\rho$. 
Hence preventing us from conjugating $h$ by $\pi$ and obtain the approximation on $\M_\rho$.\\

Thus, the idea is to work with an embedding $h$ from $\check{\M}_\rho$ to $\A_\infty$ which coincides with $\pi^{-1}$ outside a compact subset of $\check{\M}_\rho$. 
Moreover, if such an $h$ satisfies $h(\check{\M}) = \check{\A}$ and is symplectic on $\check{\M}$, it provides a compactly supported symplectomorphism $\widehat{h}$ of $\M$. 
This yields the sequence $(\widehat{h_n})_n$ for the scheme. 
Then, the $f_n: \M_{\rho'} \rightarrow \M_\rho$ are defined by $R_{\alpha_n}$ outside a compact subset of $\check{\M}_{\rho'}$ and by $\conj{h_n}{R_{\alpha_n}}$ otherwise.\\

This seems to work since, given such a $h_n: \check{\M}_\rho \to \A_\infty$, we have by the scheme and the approximation theorem the existence of a map $\tilde{h} : \A_R \rightarrow \A_\infty$ such that $h_{n+1} := \tilde{h} \circ h_n$ is well defined by taking $R$ large enough. 
Moreover it coincides with $\pi^{-1}$ outside of a compact subset of $\check{\M}_\rho$ and provides a suitable $\widehat{h}_n$ on $\M$ respecting the scheme.\\

Besides this aspect, the overall idea is the same as before, we slightly deform the complex and symplectic structure on $\M_\rho$ at each step. 
Therefore, the limit $f$ constructed is holomorphic and symplectic for the limit structures.\\

\newglossaryentry{5Dilo}{sort = {3Func Space}, name = {$\dilo{c}{\infty}{\check{\M}_\rho}{\A_\infty}$}, description = {Space of local diffeomorphisms from $\check{\M}_\rho$ to $\A_\infty$ \enquote{compactly supported}}}\glsadd{5Dilo}
We now define $\dilo{c}{\infty}{\check{\M}_\rho}{\A_\infty}$ to be the space of local diffeomorphisms from $\check{\M}_\rho$ to $\A_\infty$ which coincide with $\pi^{-1}$ on the complement of a compact subset of $\check{\M}_\rho$. Such a diffeomorphism $g$ induces the following complex structure and $2$-form on $\M_\rho$:
$$g^\# J_0 :=\left\lbrace 
\begin{array}{ccc}
g^*J_0 &\text{on} &\check{\M}_\rho\\
J_0 &\text{on} &\M_\rho \setminus \check{\M}_\rho
\end{array} \right. ,$$
and
$$\label{g_dies}g^\# \Omega_0 :=\left\lbrace 
\begin{array}{ccc}
g^*\Omega_0 &\text{on} &\check{\M}_\rho\\
\Omega_0 &\text{on} &\M_\rho \setminus \check{\M}_\rho
\end{array} \right. .$$

From \cref{ex:holo_pull} we deduce the following:

\begin{fact}
The symplectic form $g^\# \Omega_0$ is $g^\#J_0$-holomorphic
\end{fact}

Suppose now that $g(\check{\M}) = \check{\A}$ and $\rest{g}{\check{\M}}$ is symplectic. We then define a symplectomorphism $\widehat{g} \in \symp{\infty}{\M}$ by:
$$\label{g_hat}\widehat{g} := \left\lbrace 
\begin{array}{ccc}
g^{-1}\circ \pi^{-1} &\text{on} &\check{\M}\\
\id &\text{on} &\M \setminus \check{\M}
\end{array} \right. .$$

Since the maps $h_n$ will belong to $\dilo{c}{\infty}{\check{\M}_\rho}{\A_\infty}$, they will coincide with $\pi^{-1}$ on the complement of a compact subset of $\check{\M}_\rho$. 
Hence, the associated maps $f_n$ will coincide with a rotation on the complement of a compact set. 
We therefore consider $\diro{c}{\infty}{\M_{\rho'}}{\M_\rho}$ to be the space of embeddings from $\M_{\rho'}$ to $\M_\rho$ which coincide with a rotation on the complement of a compact subset of $\check{\M}_{\rho'}$.\\

Let us show the existence of sequences $\alpha_n \in \Q/\Z$, $h_n \in \dilo{c}{\infty}{\check{\M}_\rho}{\A_\infty}$ and $f_n \in \diro{c}{\infty}{\M_{\rho'}}{\M_\rho}$ such that:
$$\alpha_0 =0 \; , \; h_0 = \rest{\pi^{-1}}{\check{\M}_\rho} \quad\text{and}\quad f_0 = \rest{\id}{\check{\M}_{\rho'}}.$$
We also require that, for $n\geq 1$, $\alpha_n$, $h_n$ and $f_n$ satisfy
\begin{enumerate}[label = $(\Roman*_n)$]
\item $h_n(\check{\M}) = \check{\A}$, $h_n$ is symplectic on $\check{\M}$ and $(\widehat{h_n},\alpha_n)_n$ respects the AbC$^\star$ scheme $(U,\nu)$.
\item $f_n (\check{\M}_{\rho'}) \subset \operatorname{int} \left( \check{\M}_\rho \right)$, $h_n \circ f_n = R_{\alpha_n} \circ \rest{h_n}{\check{\M}_{\rho'}}$, and $d(f_n,f_{n-1}) < 2^{-n}$ (in this case $d$ denotes a distance which metrizes $\emb{\infty}{\M_{\rho'},\M_\rho}$.
\item $J_n = h_n^\# J_0$ and $\Omega_n = h_n^\# \Omega_0$ satisfy:
$$d(J_n,J_{n-1}) < 2^{-n} \quad\text{and}\quad d(\Omega_n,\Omega_{n-1}) < 2^{-n}.$$
\item $h_n$ commutes with $\sigma$.
\end{enumerate}

Here is the lemma providing the sequences.

\begin{lemma}\label{lem:abc_real_ds}
There exist sequences $(\alpha_n)_n$, $(h_n)_n$, and $(f_n)_n$ such that $(I_n \cdots V_n)$ is satisfied for every $n \geq 1$.
\end{lemma}

\begin{proof}[Sketch of proof]
Let $\alpha_n$, $h_n$ and $f_n$ be such that they satisfy $(I_n \cdots IV_n)$. 
Let us define $\alpha_{n+1}$, $h_{n+1}$ and $f_{n+1}$ so that they satisfy $(I_{n+1} \cdots IV_{n+1})$. 
First, by \cref{def:abc_star} of an AbC$^\star$ scheme, there exists a symplectomorphism $h_\M \in \sympo{\infty}{\M}$ commuting with $R_{\alpha_n}$ and such that $\widehat{h_n} \circ h_\M$ belongs to $U(\widehat{h_n},\alpha_n)$. 
Moreover, $U(\widehat{h_n},\alpha_n)$ belongs to $\mathcal{T}^\infty_\gamma$ for a bicurve $\gamma_\M = \pi(\gamma_\A) \in \Gamma_2(\M)$ invariant under $R_{\alpha_n}$, and $h_\M$ satisfies $\rest{h_\M}{O(\gamma)} = \id$. 
We then blow up $h_\M$ on $\partial \check{\M}$ to obtain a symplectomorphism $h_\A \in \sympo{\infty}{\A}$ which satisfies $\pi \circ h_\A = h_\M \circ \pi$. 
We can now approximate $h_\A$ modulo structural deformation by a map $h$ by applying \cref{thm:app_mod_def} with the bicurve $\gamma_\A$, and well chosen neighbourhood $\mathcal{U}$, $\sigma$-invariant compact set $\W$ and neighbourhoods $\mathcal{W}_J$ and $\mathcal{W}_\Omega$. 
We therefore consider $h_{n+1}$ to be equal to $ h \circ h_n$ on a compact subset of $\check{\M}_\rho$ and $\pi^{-1}$ elsewhere, this defines a map in $\dilo{c}{\infty}{\check{\M}_\rho}{\A_\infty}$. 
We also consider $\alpha_{n+1}$ close to $\alpha_n$, and $f_{n+1}$ to be such that $h_{n+1} \circ f_{n+1} = R_{\alpha_{n+1}} \circ \rest{h_{n+1}}{\check{\M}_{\rho'}}$ and $f_{n+1}$ coincides with $R_{\alpha_{n+1}}$ outside $\check{\M}_{\rho'}$. 
Thus, with  properties $(I_n)$, $(II_n)$, and $(IV_n)$ and with the properties satisfied by the approximation modulo deformation $h$, we obtain properties $(I_{n+1} \cdots IV_{n+1})$ for $\alpha_{n+1}$, $h_{n+1}$ and $f_{n+1}$.
\end{proof}

A complete proof of this Lemma is given in Appendix \ref{an:abc_real}.

We can then finish the proof of \cref{thm:real_sphere} thanks to the lemma.

\begin{proof}[Proof of \cref{thm:real_sphere} on $\Di$ and $\Sp$ (The end)]
Let us consider the sequence $(f_n)_n$ given by the latter lemma. 
By $(II_n)$ it converges toward a smooth embedding $f$ from $\M_{\rho'}$ to $\M_\rho$. 
Moreover, by $(I_n)$ its restriction to $\M$ is a symplectomorphism of $\M$ which realizes the AbC$^\star$ scheme.

By $(III_n)$ and \cref{thm:lim_almost_comp}, the sequence $(J_n)_n$ converges to a complex structure $J$. 
Moreover, by symplecticity of the $h_n$ on $\check{\M}$, the forms $\Omega_n$ are $J_n$-holomorphic extensions of $\Omega$ to $\M_\rho$. 
Then, by $(III_n)$ and \cref{thm:lim_almost_comp}, the sequence $(\Omega_n)_n$ converges toward $\Omega'$, a $J$-holomorphic extension of $\Omega$ to $\M_\rho$.

Also, by $J_0$-holomorphy of rotations, the following holds on $\check{\M}_{\rho'}$:
$$f_n^* J_n = (h_n \circ f_n)^*J_0 = (R_{\alpha_n} \circ h_n )^* J_0 = h_n^* J_0 = J_n.$$

Hence, at the limit, $f^*J = J$ on $\M_{\rho'}$, since on $\M_{\rho'} \setminus \check{\M}_{\rho'}$ we have $J = J_0$ and $f = R_\alpha$ (where $R_\alpha$ is the limit of $(R_{\alpha_n})_n$ on $\M_{\rho'} \setminus \check{\M}_{\rho'}$). 
As a result, $f$ is a $J$-holomorphic extension to $\M_{\rho'}$ of the symplectomorphism $\rest{f}{\M}$.
\end{proof}

\subsection{Conclusion on the principle}\label{sec:conc_proof}

Finally, we can prove \cref{thm:princ_strctr,thm:princ_strctr_equi}, concluding the proof of the principle as stated at the beginning of the section.

\begin{proof}[Proof of \cref{thm:princ_strctr}]
By \cref{thm:real_sphere}, there exist $\rho > \rho' >1$ and a complexification $(\Sp_\rho,J)$ on which $\Omega$ extends $J$-holomorphically. Moreover, there exists a $J$-holomorphic diffeomorphism $f : \Sp_{\rho'} \hookrightarrow \Sp_\rho$ which extends a real symplectomorphism of $(\Sp , \Omega)$ constructed by the AbC$^\star$ scheme $(U,\nu)$. Then, by \cref{prop:real_an_symp} on the sphere, $\rest{f}{\Sp}$ is real analytic and symplectic for a $C^\infty$-compatible analytic symplectic structure $(\mathcal{A}^{\omega'})$, which concludes the proof.
\end{proof}

Now \cref{thm:princ_strctr_equi} with the $\Z_2$-action.

\begin{proof}[Proof of \cref{thm:princ_strctr_equi}]
Let $\M = \A$ or $\Di$.\\
By \cref{thm:real_sphere} we have $\rho > \rho' >1$ and a complexification $(\M_\rho,J)$ on which $\Omega$ extends $J$-holomorphically. Moreover, there exists a $J$-holomorphic diffeomorphism $f : \M_{\rho'} \hookrightarrow \M_\rho$ which extends a real symplectomorphism of $(\M , \Omega)$ constructed by the AbC$^\star$ scheme $(U,\nu)$. \cref{thm:real_sphere} also implies that $J$ coincides with $J_0$ on $\M_\rho\setminus \M_{\rho'}$ and $f$ coincides with a rotation on $\M_\rho\setminus \M_{\rho'}$. Then, by \cref{prop:real_an_symp} on the cylinder or the disk, $\rest{f}{\M}$ is a real analytic $\Z_2$-symplectomorphism for a $C^\infty$-compatible structure $(\mathcal{A}^{\omega'})$, concluding the proof.
\end{proof}

\appendix

\section{Results on the Kantorovich distance}\label{an:kanto}
In this appendix, we present several results on the Kantorovich distance which are used in \cref{sec:finerg_cyl}.\\
First, we have the following result from \cite[Prop. 1.3]{berger_analytic_2022} on the push forward of measures by $C^0$ maps.
\begin{prop}\label{prop:dk_c0}
Let $X$ and $Y$ be measurable compact metric spaces. For $f_1,f_2 \in C^0(X,Y)$ and $\mu \in \mathcal{M}(X)$, the Kantorovich distance satisfies:
$$d_K({f_1}_*\mu,{f_2}_*\mu) \leq d_{C^0}(f_1,f_2).$$
\end{prop} 

Next, let us present a result on the Kantorovich distance for bi-Lipschitz maps.

\begin{prop}\label{prop:kanto}
Let $(X,d_X)$ and $(Y,d_Y)$ be measurable metric spaces and let $\phi : X \rightarrow Y$ be an invertible $Q$-bi-Lipschitz map, that is $\phi$ and $\phi^{-1}$ are $Q$-Lipschitz. Then for $\mu_1$ and $\mu_2$ two measures on $X$ we have:
$$\frac{1}{Q} d_K(\mu_1,\mu_2) \leq d_K(\phi_*\mu_1,\phi_*\mu_2) \leq Q d_K(\mu_1,\mu_2).$$
\end{prop}

\begin{proof}
To obtain the right-hand inequality, we consider $f:Y \rightarrow \R$ a $1$-Lipschitz map, then $\tfrac{1}{Q}f\circ \phi$ is also $1$-Lipschitz and we have by Kantorovich's theorem:
$$\int_Yfd(\phi_*\mu_1-\phi_*\mu_2) = \int_X f\circ \phi d(\mu_1-\mu_2) = Q\int_X\frac{1}{Q}f\circ \phi d(\mu_1-\mu_2) \leq Q d_K(\mu_1,\mu_2).$$
This provides the right-hand inequality by Kantorovich's theorem. Then, we obtain the left-hand inequality with the right one by replacing $\phi$ by $\phi^{-1}$ and $\mu_i$ by $\phi_*\mu_i$ for $i\in \croset{1,2}$.
\end{proof}

From these two propositions we deduce the following result.

\begin{coro}\label{coro:kanto}
Let $h$ be in $\symp{0}{\M}$, then the map $\mu \in \mathcal{M}(\M) \mapsto h_*\mu$ is continuous.
\end{coro}

\begin{proof}
Let us show the continuity of the map by sequential characterization. Let $(\mu_n)_n$ be a sequence of probability measures on $\M$ converging to $\mu \in \mathcal{M}(\M)$. Let $\epsilon>0$, by density of $\symp{1}{\M}$ in $\symp{0}{\M}$ there exists a map $g\in \symp{1}{\M}$ which is $\epsilon$-$C^0$ close to $h$. Let $Q>0$ be such that $g$ is $Q$-bi-Lipschitz. For $n\in \N^*$ such that $d_K(\mu_n,\mu)\leq \tfrac{\epsilon}{Q}$, we have by triangle inequality and the two latter propositions:
$$d_K(h_*\mu_n,h_*\mu) \leq d_K(h_*\mu_n,g_*\mu_n) + d_K(g_*\mu_n,g_*\mu) + d_K(h_*\mu,g_*\mu) \leq \epsilon + Qd_K(\mu_n,\mu) + \epsilon \leq 3\epsilon.$$
This yields the continuity of the map $\mu \mapsto h_*\mu$.
\end{proof}

Next observe that the Kantorovich distance satisfies nice properties on convex combinations.

\begin{prop}\label{prop:kant_conv1}
Let $X$ be a compact metric space. Let $\nu_1,\nu_2, \mu_1$, and $\mu_2$ be probability measures in $\mathcal{M}(X)$ and $\alpha \in [0,1]$. We have:
$$d_K(\alpha \nu_1 + (1-\alpha)\nu_2,\alpha\mu_1 +(1-\alpha)\mu_2) \leq \alpha d_K(\nu_1,\mu_1) + (1-\alpha)d_K(\nu_2,\mu_2).$$
\end{prop}

\begin{proof} Let $f: X \rightarrow \R$ be a $1$-Lipschitz map, we have:
$$\int_X fd(\alpha\nu_1 + (1-\alpha)\nu_2 - \alpha\mu_1 -(1-\alpha)\mu_2) = \alpha\int_X fd(\nu_1 - \mu_1) + (1-\alpha)\int_X fd(\nu_2 - \mu_2).$$
Hence the result follows by Kantorovich's duality theorem.

\end{proof}

Likewise, using Kantorovich's duality theorem, we obtain the following proposition on convex combinations.
\begin{prop}\label{prop:kant_conv2}
Let $X$ be a compact metric space. Let $\mu_1$ and $\mu_2$ be probability measures in $\mathcal{M}(X)$ and $\alpha \in [0,1]$. We have:
$$d_K(\alpha \mu_1 + (1-\alpha)\mu_2,\mu_1) =(1-\alpha) d_K(\mu_1,\mu_2).$$
\end{prop}

Now, to obtain that a measure is close to the Lebesgue measure of a surface, we have the following proposition:

\begin{prop}\label{an:dK_unif}
Let $S$ be a compact surface, let $\epsilon>0$ and let $(D_1, \cdots, D_n)$ be disjoints subsets of $S$ such that:
\begin{enumerate}
\item $\Leb_S(\sqcup_i D_i) \geq 1-\frac{\epsilon}{\diam(S)}$,
\item $\Leb_S(D_i) = \Leb_S(D_j)$ for every $i,j$,
\item and $\diam(D_i) \leq \epsilon$ for every $i$.
\end{enumerate}
Then, for every family $(\mu_i)_{1 \le i \le N}$ such that $\mu_i \in \mathcal{M}(D_i)$, we have:
$$d_K(\Leb_S, \frac{1}{N} \isum{i=1}{N} \mu_i) \leq 2 \epsilon.$$
\end{prop}

\begin{proof}
First, because we have the decomposition $\Leb_S = \Leb_S(D) \rest{\Leb_S}{D} + (1-\Leb_S(D)) \rest{\Leb_S}{S\setminus D}$ where $D = \sqcup_i D_i$, it follows from \cref{prop:kant_conv2} and 1) that
\begin{equation}\label{eq:dk_SD}
d_K( \Leb_S, \rest{\Leb_S}{D}) \leq \frac{\epsilon}{\diam(S)} d_K(\rest{\Leb_S}{S \setminus D}, \rest{\Leb_S}{D}) \leq \epsilon.
\end{equation}
Then, by 2) we have $\rest{\Leb_S}{D} = \tfrac{1}{N}\isum{i=1}{N}\rest{\Leb_S}{D_i}$, i.e. by \cref{prop:kant_conv1} it follows:
$$d_K(\rest{\Leb_S}{D},\frac{1}{N} \isum{i=1}{N} \mu_i) \leq \frac{1}{N} \isum{i=1}{N} d_K(\rest{\Leb_S}{D_i},\mu_i).$$
Moreover, because $\diam(\mathcal{M}(D_i)) = \diam(D_i)$, it follows with the latter equation and 3) that:
$$d_K(\rest{\Leb_S}{D},\frac{1}{N} \isum{i=1}{N} \mu_i) \leq \epsilon.$$
This yields the results by using the triangular inequality and \cref{eq:dk_SD}.

\end{proof}

\section{Symplectomorphism between bicurves}\label{an:sec_bic}

In this appendix, we prove \cref{an:lemma_symp_bi} which, for a given bicurve $\gamma \subset \A \setminus \partial\A$, provides a symplectomorphism that sends a horizontal bicurve to $\gamma$. We recall that, by \cref{def:bicurve}, there exists a diffeotopy of the cylinder from a horizontal bicurve to $\gamma$. Therefore, by an application of Moser's trick and a slight modification of the horizontal bicurve, one obtains the symplectomorphism from the given diffeotopy. We recall the proposition.

\begin{prop-no}[\ref{an:lemma_symp_bi}]
Let $\gamma$ be in $\Gamma_2(\A)$, there exist a bicurve $\gamma' := \T \times \croset{y_-,y_+}$ and a symplectomorphism $\phi\in\mathrm{Symp}_c^\infty(\A,\partial\A)$ such that
$$\phi \left (\gamma'\right ) = \gamma.$$
\end{prop-no}

\begin{proof} 
First, let us prove the proposition in the case $\gamma \subset \A \setminus \partial \A$. Let us choose the appropriate bicurve $\gamma'$ with the following sub-lemma that we prove below.

\begin{lemma}\label{subl:diffeotopy}
There exists a bicurve $\gamma' := \T \times \croset{y_-,y_+}$ diffeotopic to $\gamma$ via a diffeotopy $(f_t)_{t\in [0,1]}$ compactly supported in $\A \setminus \partial\A$ such that $f_1$ preserves the areas of the components of $\A \setminus \gamma'$.
\end{lemma}

We therefore consider such a diffeotopy $(f_t)_{t\in [0,1]}$ compactly supported in $\A \setminus \partial\A$ between a bicurve $\gamma' := \T \times \croset{y_-,y_+}$ and $\gamma$. For any component $N$ of $\A \setminus \gamma'$, the diffeotopy satisfies:

\begin{equation}\label{eq:vol_N}
\Leb(f_1(N) ) = \Leb (N) \quad\text{and}\quad f_1 (\gamma'_\pm) = f_1( \T \times \croset{y_\pm} ) = \gamma_\pm.
\end{equation}

Now, let us use Moser's trick to compose $(f_t)_t$ with a smooth isotopy and obtain the desired symplectomorphism.
First, for any component $N$ of $\A \setminus \gamma'$, we can assume that $f_t$ preserves the canonical symplectic form  $\Omega$ near the boundary of $f_t(N)$ in $\A$. Indeed, we can smoothly change the size of the normal component of the differential of the $f_t$ near a tubular open neighbourhood of the boundary of $f_t(N)$ in $\A$. This is achieved by smoothly deforming $f_t$ on this neighbourhood.\\
  
Let $N$ be a component of $\A \setminus \gamma'$, and let $V$ be a tubular neighbourhood of $\partial N$ in $N$ such that $\rest{f_t}{V}$ is symplectic.

Let us define the family of symplectic forms $(\Omega_t)_t$ in order to apply Moser's trick.
First, by shrinking the neighbourhood $V$, we can assume that 
\begin{equation}\label{eq:VleN}
\Leb(V) < \Leb(N).
\end{equation}
Next, let $U$ be a neighbourhood of $\partial N$ relatively compact in $V$ and let $\chi : N \rightarrow [0,1]$ be a smooth function such that $\rest{\chi}{N\setminus V} \equiv 1$ and $\rest{\chi}{U} \equiv 0$. In particular, for every $t$, we have $\int_N \chi f_t^*\Omega \geq  \Leb(f_t(N\setminus V)) >0$ since $f_t$ is a diffeomorphism and $\Leb(N\setminus V) >0$ by \cref{eq:VleN}. Thus, we define for $t \in [0,1]$: 
$$\Omega_t := A_t\cdot\rest{(f_t^*\Omega)}{N}\quad \text{with} \quad A_t := 1 + \frac{\Leb(N)-\Leb(f_t(N))}{\int_N\chi f_t^*\Omega}\chi .$$

Note that $\int_N \Omega_t = \Leb(N)$. Moreover, since $\Leb(N) = \Leb(f_1(N))$ by \eqref{eq:vol_N}, it follows that $A_1 \equiv 1$ and so:
\begin{equation}\label{eq:f1_omega}
\rest{(f_1^*\Omega)}{N} = \Omega_1.
\end{equation}
In addition, observe that the $\Omega_t$ are symplectic forms by the following fact.
\begin{fact}
For every $t \in [0,1]$, $A_t$ is positive and $\Omega_t$ is a symplectic form.
\end{fact}
\begin{proof}[Proof of the Fact]
To prove that $A_t$ is positive, it suffices to show that $\frac{\Leb(N)-\Leb(f_t(N))}{\int_N\chi f_t^*\Omega}\chi > -1$. Then, since $\chi$ takes its values in $[0,1]$ and $\int_N\chi f_t^*\Omega$ is positive, it reduces to prove that:
$$\Leb(N)-\Leb(f_t(N)) > - \int_N\chi f_t^*\Omega $$
Now, since $f_t$ is symplectic on $V$, we have $\Leb(f_t(V)) = \Leb(V) < \Leb(N)$ by \cref{eq:VleN}. Therefore, we have: 
$$\int_N\chi f_t^*\Omega \geq \int_{N\setminus V} f_t^*\Omega = \Leb(f_t(N)) - \Leb(f_t(V)) > \Leb(f_t(N)) - \Leb(N).$$
This proves the positivity of $A_t$. Hence $\Omega_t$ is indeed a symplectic form.
\end{proof}

Next, as $\int_N \Omega_t = \Leb(N)$ for every $t$, the $(\Omega_t)_t$ are cohomologous symplectic forms with constant restriction to $U$ (because $\rest{f_t}{U}$ is symplectic and $\rest{A_t}{U} \equiv 1$). As a result, by the variant \cite[Ex. 3.2.6]{mcduff_introduction_2017} of the relative Darboux theorem using Moser's trick, there exists an isotopy $(\Psi_t)_t$ of smooth diffeomorphisms of $N$ such that $\Psi_t^*\Omega_t = \Omega_0 = \rest{\Omega}{N}$ for every $t$ and $\rest{\Psi_t}{U} \equiv \id$. Furthermore, \eqref{eq:f1_omega} leads to
$$(f_1 \circ \Psi_1)^*\Omega = \Psi_1^* \Omega_1 = \rest{\Omega}{N}.$$

Moreover, since $\Psi_1$ is the identity on $U$, we can smoothly glue together the maps $\Psi_1$ obtained from the different components of $\A \setminus \gamma'$ into a diffeomorphism $\widehat{\Psi}$ of $\A$ compactly supported in $\check{\A} \setminus \gamma'$. Therefore we can consider the smooth symplectomorphism $\phi := f_1 \circ \widehat{\Psi}$ compactly supported in $\check{\A}$.

Furthermore, because $\widehat{\Psi}$ is compactly supported in $\check \A \setminus \gamma'$ it is equal to the identity on a neighbourhood of $\gamma'$. Hence, the symplectomorphism $\phi$ satisfies:
$$\phi(\gamma') = f_1(\gamma') = \gamma.$$

This sorts the case $\gamma \subset \A \setminus \partial \A$.\\

Now, we consider $\gamma$ such that $\gamma \cap \partial \A \neq \emptyset$. If $\gamma = \partial \A$, the proposition is obvious by taking $\phi = \id$.
Otherwise, we have either $\gamma_+ = \T \times \croset{1}$ and $\gamma_- \subset \A \setminus \partial \A$, or $\gamma_- = \T \times \croset{-1}$ and $\gamma_+ \subset \A \setminus \partial \A$, and the latter case reduces to the first one upon conjugating $\phi$ by the involution $(\theta,y) \mapsto (\theta,-y)$ and applying the same transformation to $\gamma$. Now, by definition of a bicurve there exists $\tilde{\gamma} := \T \times \croset{\tilde y_-,1}$ that is diffeotopic to $\gamma$ via a diffeotopy $(h_t)_t$ compactly supported in $\A \setminus \partial \A$. 
Since the diffeotopy is compactly supported and $\tilde y_- <1$, there exists $\eta >0$ such that $\supp(h_t) \subset \A(2\eta)$ and $\tilde y_- < 1-\eta$. By considering $\bar{\gamma} := \gamma_- \sqcup \T \times \croset{1-\eta} \subset \A \setminus \partial \A$, it defines a bicurve since it is diffeotopic to $\T \times \croset{\tilde{y}_-,1-\eta}$ via $(h_t)_t$. 
Hence, by applying \cref{an:lemma_symp_bi} to $\bar{\gamma} \subset \A \setminus \partial \A$, it yields a symplectomorphism $\phi \in\mathrm{Symp}_c^\infty(\A,\partial\A)$ which sends a bicurve $\T \times \croset{\bar{y}_-,\bar{y}_+}$ to $\bar{\gamma}$, in particular it sends $\T \times \croset{\bar{y}_-,1}$ to $\gamma$ since it is compactly supported in $\A \setminus \partial \A$. 
This gives the proposition in the case $\gamma \not\subset \A \setminus \partial \A$.\\
\end{proof}

\begin{proof}[Proof of \cref{subl:diffeotopy}]
By the definition of a bicurve $\gamma$, there exists a bicurve $\tilde{\gamma} := \T \times \croset{\tilde y_-,\tilde y_+}$ that is diffeotopic to $\gamma$ via a diffeotopy $(h_t)_t$, compactly supported in $\A \setminus \partial \A$. Let us construct a diffeotopy $(g_t)_t$ from a bicurve $\gamma'$ to $\tilde{\gamma}$, then concatenate $(h_t)_t$ and $(g_t)_t$ to obtain a diffeotopy $(f_t)_t$. With an appropriate choice of $\gamma'$, the diffeotopy $(f_t)_t$ satisfies the property that $f_1$ preserves the areas of the components of $\A \setminus \gamma'$.

Let $O_+$ and $O_-$ be the components of $O(\gamma)$ such that $\gamma_\pm$ lies in the boundary of $O_\pm$. Then, we choose the bicurve $\gamma' := \T \times \croset{y_-,y_+}$ so that 
\begin{equation}\label{eq:y_O}
\Leb(O_+)= \tfrac{(1-y_+)}{2} \quad \text{and} \quad \Leb(O_-)= \tfrac{(y_- + 1)}{2},
\end{equation}
note that, since $\Leb(O(\gamma))<1$ as $A(\gamma)$ is never empty, we have $y_+ > y_-$.

Assume that we have constructed a diffeotopy $(g_t)_t$ between $\gamma'$ and $\tilde \gamma$ and we denote by $(f_t)_t$ its concatenation with $(h_t)_t$. Therefore, $f_1$ sends the components of $O(\gamma')$ to $O_+$ and $O_-$, hence it preserves their areas by \cref{eq:y_O}. Moreover, $f_1$ also preserves the area of $A(\gamma')$, which is sent to $A(\gamma)$, because 
$$\Leb(\A) = \Leb(A(\gamma')) + \Leb(O(\gamma')) = \Leb(A(\gamma)) + \Leb(O(\gamma)).$$
Thus, such a $(f_t)_t$ is the desired diffeotopy between $\gamma'$ and $\gamma$.

It remains to construct this diffeotopy $(g_t)_t$ to conclude. Let us construct $(g_t)_t$ by using the following standard fact.

\begin{fact}
Let $y_{+}^1$, $y_{-}^1$, $y_{+}^2$, and $y_{-}^2$ be points of $(-1,1)$ such that $y_{+}^i > y_{-}^i$. There exists a diffeomorphism $G$ of $[-1,1]$ compactly supported in $(-1,1)$ such that $G(y_+^1) = y_+^2$ and $G(y_-^1)=y_-^2$.
\end{fact}

Now, because $\gamma \subset \A\setminus \partial\A$ and $(h_t)_t$ is a diffeotopy from $\tilde \gamma$ to $\gamma$, it follows that the points $\tilde{y}_+$ and $\tilde y_-$ belong to $(-1,1)$. Moreover, $\gamma \subset \A\setminus \partial\A$ also implies that $\Leb(O_+)$ and $\Leb(O_-)$ are nonzero, hence that $y_+$ and $y_-$ also belongs to $(-1,1)$ by \cref{eq:y_O}. Thus, we can apply the above fact to the points $y_+$, $y_-$, $\tilde{y}_+$, and $\tilde y_-$. This yields a diffeomorphism $G$ of $[-1,1]$, compactly supported in $(-1,1)$, such that $G(y_+) = \tilde y_+$ and $G(y_-)=\tilde y_-$. Now, as $G$ is a diffeomorphism, we consider the path 
$$G_t : y\in \I \mapsto tG(y) + (1-t)y = -1 + \int_{-1}^y tG'(u) +(1-t)du,$$
defined for $t \in [0,1]$. It is an isotopy between $\id$ and $G$ in the space of diffeomorphisms of $[-1,1]$. Since $G$ is compactly supported in $(-1,1)$, the isotopy $(G_t)_t$ is aswell.
We finally define our diffeotopy between $\gamma'$ and $\tilde{\gamma}$ compactly supported in $\A \setminus \partial\A$ by 
$$g_t : (\theta,y) \in \A \mapsto (\theta,G_t(y)) \in \A,$$
for $t\in [0,1]$.\\
Concatenating $(h_t)_t$ and $(g_t)_t$ yields the desired diffeotopy $(f_t)_t$.
\end{proof}

\section{Proof of the Theorem of approximation modulo deformation}\label{an:approx_def}

In this Appendix we prove \cref{thm:app_mod_def} from \cref{sec:approx_mod}, following the idea of Berger's proof.\\

Before we prove the theorem, we recall the main approximation theorem from the Subsection \ref{an:approx} and derive a corollary.

\begin{theorem-no}[\ref{an:thm_main_approx}]
Let $\gamma \in \Gamma_2(\A)$ and $f \in \symps{\infty}{\A}$ be such that $\rest{f}{O(\gamma)} = \id$. Then, for any $\rho >1$ and any neighbourhood $\mathcal{U}$ of $\rest{f}{\A \setminus \gamma}$in the smooth compact-open topology of $C^\infty(\A \setminus \gamma,\T\times\R)$, there exists a map $F \in \sympan{\rho}{\A_\infty}$ such that $\rest{F}{\A \setminus \gamma}\in \mathcal{U}$.
\end{theorem-no}

\begin{coro}
Let $h_0$ be in $\symps{\infty}{\A}$, $\gamma$ in $\Gamma_2(\A)$ such that $\rest{h_0}{O(\gamma)} = id$ and $\mathcal{U}$ be a neighbourhood of $\rest{h_0}{\A \setminus \gamma}$ in the smooth compact-open topology of $C^\infty(\A\setminus\gamma,\T\times\R)$.

There exist $\epsilon >0$ and a family of biholomorphism $(h_r)_{r \geq 1}$ in $\sympan{}{\A_\infty}$ such that, for every $r$, $\rest{h_r^{-1}}{\A\setminus\gamma}$ belongs to $\mathcal{U}$, commutes with $\sigma$, and satisfies
$$\Vert \rest{(h_r -\id)}{K_{r,\epsilon}} \Vert_{C^r} \leq r^{-1}.$$
\end{coro}

\begin{proof}
First, by \cref{rk:gam_bord}, we can assume that $\gamma \subset \A \setminus \partial \A$ and $\supp(h_0) \subset A(\gamma)$.

Then, as $\gamma \cap \partial\A = \emptyset$, there exists $\epsilon>0$ such that $\gamma$ lies in $\A_{2\epsilon} = \T \times (-1+2\epsilon,1-2\epsilon)$, in particular $h_0$ is compactly supported in $\A_{2\epsilon}$. 
We now consider the affine map:
$$g : (\theta,y) \in \A_\infty \mapsto (\theta,\frac{y}{1-2\epsilon}) \in \A_\infty.$$
This map sends $\A_{2\epsilon}$ onto $\A_0$. Therefore the symplectomorphism $f:=\iconj{g}{h_0}$ belongs to $\symps{\infty}{\A}$ and equals the identity on $O(g(\gamma))$. 
Moreover, $g$ conjugates the neighbourhood $\mathcal{U}$ onto a neighbourhood $\mathcal{V}$ of $\rest{f}{\T \times \R \setminus g(\gamma)}$ in the smooth compact-open topology of $C^\infty(\T \times \R \setminus g(\gamma),\T\times\R)$. 
Yet, as $\rest{f}{\T \times \R \setminus A(g(\gamma))} = \id$, there exists a neighbourhood $\mathcal{U}'$ in the smooth compact-open topology of $C^\infty(\A\setminus g(\gamma),\T\times\R)$ such that the following holds by Cauchy's estimates in $K_{\rho_r}$:
\begin{equation}\label{eq:U'_an}
\forall F \in \sympan{\rho_r}{\A_\infty} \, , \, \rest{F}{\A \setminus g(\gamma)} \in \mathcal{U}' \Longrightarrow  \rest{F}{\T \times \R \setminus g(\gamma)} \in \mathcal{V}.
\end{equation}

Thus, we apply \cref{an:thm_main_approx} to $f=\iconj{g}{h_0}$, the bicurve $g(\gamma)$, the neighbourhood $\mathcal{U}'$ and a large $\rho_r >1$ to obtain a biholomorphism $F_r \in \sympan{\rho_r}{\A_\infty}$ such that $\rest{F_r}{\A \setminus g(\gamma)}$ belongs to $\mathcal{U}'$. 
In particular, by \cref{eq:U'_an} and since $\mathcal{V}$ is the conjugacy of $\mathcal{U}$ by $g$, the restriction $\rest{\conj{g}{F_r}}{\A \setminus \gamma}$ belongs to $\mathcal{U}$.\\

Next, let us consider the biholomorphism $h_r := \conj{g}{F_r^{-1}} \in \sympan{}{\A_\infty}$. 
It follows that $\rest{h_r^{-1}}{\A\setminus\gamma} = \rest{\conj{g}{F_r}}{\A\setminus\gamma}$ belongs to $\mathcal{U}$ as desired. By definition of $\sympan{}{\A_\infty}$, the map $h_r$ restricts to a real symplectomorphism of $\T\times \R$, therefore it commutes with $\sigma$.

Finally, observe that with $\rho'_r = (1-2\epsilon)\rho_r-4\epsilon$ we have $g(K_{\rho'_r,2\epsilon}) \subset K_{\rho_r}$. 
Therefore $h_r^{-1} = \conj{g}{F_r}$ is $\rho^{-1}$-close to the identity on $K_{\rho'_r,2\epsilon}$ for $\rho = \rho_r/(1-2\epsilon)$. Thus, as $\rho_r$ is large, the continuity of the inversion on compact sets together with Cauchy's inequality yields the result:
$$\Vert \rest{(h_r -\id)}{K_{r,\epsilon}} \Vert_{C^r} \leq r^{-1}.$$
\end{proof}

Now let us prove the theorem of approximation modulo deformation.

\begin{proof}[Proof of Theorem \ref{thm:app_mod_def}]
Firstly, by the latter corollary, there exists $\epsilon >0$ and a family $(h_r)_{r\geq 1}$ of biholomorphisms of $\A_\infty$ preserving $\Omega_0$, commuting with $\sigma$, such that $\rest{h_r^{-1}}{\A\setminus\gamma} \in \mathcal{U}$, and satisfying:
\begin{align}\label{eq:r_conv}
 \Vert \rest{(h_r -\id)}{K_{r,\epsilon}} \Vert_{C^r} \rconv{r}{+\infty} 0. 
\end{align}

Let $\eta \in (0,\epsilon)$. We fix $\beta \in C^\infty(\R,\R_+)$ supported in $(-1+\eta,1-\eta)$ such that $\beta \equiv 1$ on $(-1+\epsilon, 1-\epsilon)$. 
For $\W'$ a compact neighbourhood of $\W$ with smooth boundary we define:
$$\widehat{h_r} : (\theta,y) \in \W' \mapsto \beta(\Re(y)) h_r(\theta,y) + \left( 1 - \beta(\Re(y))\right) (\theta,y).$$

Then, since $\widehat{h}_r$ is either close to $\id$ or equal to $h_r$, and $h_r$ is a biholomorphism, $\widehat{h}_r$ becomes a smooth diffeomorphism onto its image for $r$ sufficiently large (see \cite[Fact 6.11]{berger_analytic_2024}). 
Moreover, $\widehat{h_r}$ satisfies the following properties:
\begin{fact}\label{fact:hr_hat}
When $r$ is large it holds:
\begin{itemize}
\item[($\mathcal{P}'1$)] $\widehat{h_r}$ leaves $\A$ invariant and $\rest{\widehat{h_r}^{-1}}{\A\setminus\gamma}$ belongs to $\mathcal{U}$,
\item[($\mathcal{P}'2$)] $\widehat{h_r}$ is supported in $\W' \setminus K_{r,\eta}$,
\item[($\mathcal{P}'3$)] $\widehat{h_r}^* J_0 - J_0$ and $\widehat{h_r}^*\Omega_0 - \Omega_0$ are $C^\infty$-small when $r$ is large with support in $K_{r,\epsilon} \setminus K_{r,\eta}$,
\item[($\mathcal{P}'4$)] $\widehat{h_r}$ commutes with $\sigma$.
\end{itemize}
\end{fact}
\begin{proof}[Proof of the fact]
First, since $\beta$ is supported in $(-1+\eta,1-\eta)$ and $\widehat{h_r}$ is a diffeomorphism onto its image, it follows that $\widehat{h_r}$ leaves $\A$ invariant. 
Moreover, since $\beta \equiv 1$ on $(-1+\epsilon, 1-\epsilon)$, then $\widehat{h_r}$ is $C^\infty$-close to $h_r$ by \cref{eq:r_conv} when $r$ is large, hence $\rest{\widehat{h_r}^{-1}}{\A\setminus\gamma}$ belongs to $\mathcal{U}$ and property ($\mathcal{P}'1$) follows.
Next, property ($\mathcal{P}'2$) is a direct consequence of $\beta$ being supported in $(-1+\eta,1-\eta)$. 
Property ($\mathcal{P}'3$) follows from the facts that $\widehat{h_r}$ is equal to the symplectic biholomorphism $h_r$ outside $K_{r,\epsilon}$, it is equal to the identity in $K_{r,\eta}$, and it is $C^\infty$-close to the inclusion in $K_{r,\epsilon} \setminus K_{r,\eta}$. 
Property ($\mathcal{P}'4$) follows from the fact that $h_r$ commutes with $\sigma$.
\end{proof}

Therefore $\widehat{h_r}$ satisfies the properties of \cref{thm:app_mod_def} for $r$ large, except the property of being symplectic on $\A$.

Hence, the last step of the proof is to obtain a symplectic diffeomorphism on $\A$ by composing $\widehat{h_r}$ with the inverse of a diffeomorphism $\phi_r$. This diffeomorphism is provided by the following result from \cite[Lemma 6.13]{berger_analytic_2024}.

\begin{lemma}\label{lem:phi_r}
When $r$ is large, there exists a diffeomorphism $\phi_r : \W' \to \A_\infty$ that is $C^\infty$-close to the canonical inclusion, leaves $\A$ invariant, commutes with $\sigma$, with support in $\W' \setminus K_{r,\eta}\subset \check{\A}_\infty \cap \W'$ and that satisfies:
$${\phi_r}^* \rest{\Omega_0}{T\A} = \widehat{h_r}^* \rest{\Omega_0}{T\A}.$$
\end{lemma}

We therefore consider such a $\phi_r$. Since $r$ is large and $\W'$ is a neighbourhood of $\W$, $\phi_r$ is close enough to the inclusion so that its range contains $\W$. Hence, we can define the following map:
$$h := \widehat{h_r} \circ \phi_r^{-1} : \W \rightarrow \A_\infty.$$
Let us show that $h$ is the desired smooth diffeomorphism when $r$ is large. Since $\phi_r$ leaves $\A$ invariant and is $C^\infty$ close to the inclusion when $r$ is large, $h$ also satisfies property ($\mathcal{P}'1$) from \cref{fact:hr_hat}. 
It remains to prove that $h$ is symplectic on $\A$ to obtain property ($\mathcal{P}1$) for $h$. Observe that by \cref{lem:phi_r} we have
$$h^*\rest{\Omega_0}{T\A} = (\phi_r^{-1})^* ( \widehat{h_r}^* \rest{\Omega_0}{T\A}) = \rest{\Omega_0}{T\A}.$$
This yields property ($\mathcal{P}1$). Property ($\mathcal{P}2$) follows from property ($\mathcal{P}'2$) and the fact that $\phi_r$ is supported in $\W' \setminus K_{r,\eta}$ which is equal to $\croset{(\theta,y) \in \W': \left| \Re(y) \right| <1-\eta}$ when $r$ is large. 
Property ($\mathcal{P}3$) is a consequence of property ($\mathcal{P}'3$) and the fact that $\phi_r$ is $C^\infty$-close to the inclusion when $r$ is large. 
Finally, property ($\mathcal{P}4$) follows from property ($\mathcal{P}'4$) and the fact that $\phi_r$ also commutes with $\sigma$.

Therefore, when $r$ is large, the diffeomorphism $h$ satisfies \cref{thm:app_mod_def}. This concludes the proof.
\end{proof}

\section{Proof of the AbC$^\star$ realization modulo deformations}\label{an:abc_real}

\paragraph{Proof of \cref{lem:abc_real_cyl} on the cylinder}

In this first part we prove the lemma used to obtain the AbC$^\star$ realization on the cylinder. This is done by using the Approximation modulo deformation Theorem in an AbC$^\star$ scheme.

We recall that proving the lemma is to prove the existence of sequences $\alpha_n \in \Q/\Z$, $h_n$ with $h_n^{-1} \in \emb{\infty}{\A_\rho,\A_\infty}$ and $f_n \in \emb{\infty}{\A_{\rho'},\A_\rho}$ -- where $\rho' < \rho$ -- such that:

$$\alpha_0 =0 \; , \; h_0 = \rest{\id}{\A_\rho} \quad\text{and}\quad f_0 = \rest{\id}{\A_{\rho'}}.$$
It is also required that, for $n\geq 1$, $\alpha_n$, $h_n$ and $f_n$ satisfy:
\begin{enumerate}
\item[$(I_n)$] $h_n(\check{\A}) = \check{\A}$, $h_n$ is symplectic on $\check{\A}$ and $(\rest{h_n}{\A},\alpha_n)$ respects the AbC$^\star$ scheme $(U,\nu)$:
$$\rest{h_n}{\A} \in U(\rest{h_{n-1}}{\A},\alpha_{n-1}) \quad , \quad \iconj{\rest{h_{n}}{\A}}{R_{\alpha_{n-1}}} = \iconj{\rest{h_{n-1}}{\A}}{R_{\alpha_{n-1}}},$$
$$ \text{and} \quad 0 < \lvert \alpha_n - \alpha_{n-1} \rvert < \nu(\rest{h_n}{\A},\alpha_{n-1}).$$ 
\item[$(II_n)$] $R_{\alpha_n} \circ h_n^{-1} (\A_{\rho'}) \subset \operatorname{int} \left( h_n^{-1}(\A_\rho ) \right)$, so the composition $f_n := \rest{\iconj{h_n}{R_{\alpha_n}}}{\A_{\rho'}}$ is well defined. Furthermore, $d(f_n,f_{n-1}) < 2^{-n}$.
\item[$(III_n)$] $J_n = {h_n^{-1}}^* J_0$ and $\Omega_n = {h_n^{-1}}^* \Omega_0$ satisfy:
$$d(J_n,J_{n-1}) < 2^{-n} \quad\text{and}\quad d(\Omega_n,\Omega_{n-1}) < 2^{-n}.$$
\item[$(IV_n)$] $h_n$ commutes with $\sigma$.
\item[$(V_n)$] The support of $h_n^{-1}$ is included in $\check{\A}_\rho$.
\end{enumerate}

\begin{proof}[Proof of \cref{lem:abc_real_cyl}]

Let $n\geq 0$ be such that $h_n$, $\alpha_n$ and $f_n$ are defined and satisfy $(I_n \cdots V_n)$\footnote{If $n=0$, properties $(I_n \cdots V_n)$ are not defined. However, the only properties used to obtain the induction in the proof are the first parts of properties $(I_n)$ and $(II_n)$ and properties $(IV_n)$ and $(V_n)$, which are obviously satisfied for $n=0$. The induction is therefore naturally initiated at $n=0$.}. Let us build $h_{n+1}$, $\alpha_{n+1}$ and $f_{n+1}$ such that they satisfy $(I_{n+1} \cdots V_{n+1})$. To this end, we will approximate modulo deformation a symplectomorphism $\widetilde{h}$ for neighbourhoods $\mathcal{U}$, $\mathcal{W}_J$, and $\mathcal{W}_\Omega$, and a $\sigma$-invariant compact set $\W$. 

First we define this symplectomorphism and the neighbourhood $\mathcal{U}$. Since $(U,\nu)$ is an AbC$^\star$ scheme, by a) of \cref{def:abc_star}, there exists a symplectomorphism $\widetilde{h} \in \sympo{\infty}{\A}$ commuting with $R_{\alpha_n}$ and such that $h_n \circ \widetilde{h}$ belongs to $U(\rest{h_n}{\A},\alpha_n)$. Moreover, by b) of \cref{def:abc_star}, $U(\rest{h_n}{\A},\alpha_n)$ belongs to $\mathcal{T}^\infty_\gamma$ for a bicurve $\gamma \in \Gamma_2(\A)$ invariant under $R_{\alpha_n}$ and $\widetilde{h}$ satisfies $\rest{\widetilde{h}}{O(\gamma)} = \id$. Hence, if we consider 
\begin{equation}\label{eq:def_U0}
\mathcal{U} : =h_n^{-1} \circ U(\rest{h_n}{\A},\alpha_n),
\end{equation}
it defines a neighbourhood of $\rest{\widetilde{h}}{\A \setminus \gamma}$ in the smooth compact-open topology of $C^\infty(\A \setminus \gamma, \T \times \R)$. 

We now define the $\sigma$-invariant compact set by $\W := \A_R$, with $R$ large enough such that $\W$ contains $h_n^{-1}(\A_\rho)$. Then, since $\W$ is invariant under $R_{\alpha_n}$, we have by $(II_n)$:
\begin{equation}\label{eq:inc_W}
h_n^{-1}(\A_\rho) \Subset \W \quad \text{and} \quad R_{\alpha_n} \circ h_n^{-1}(\A_\rho) \Subset \W,
\end{equation}
this will be useful later.

Then, we chose the $C^\infty$ neighbourhoods $\mathcal{W}_J$ of $\rest{J_0}{\W}$ and $\mathcal{W}_\Omega$ of $\rest{\Omega}{\W}$ such that for all $J'$ in $\mathcal{W}_J$ the structure ${h_n^{-1}}^* J'$ is at a distance $< 2^{-n-1}$ from $J_n$, and for all $\Omega'$ in $\mathcal{W}_\Omega$  the $2$-form ${h_n^{-1}}^* \Omega'$ is at a distance $< 2^{-n-1}$  from $\Omega_n$.\\

Let $\eta_0>0$ be such that the support of $h_n$ is included in $\A_\infty(\eta_0) = \croset{ (\theta,y) \in \A_\infty \; ; \; \lvert \Re(y) \rvert < 1-\eta_0}$. By \cref{thm:app_mod_def}, $\widetilde{h}$ is approximable modulo structural deformation. Hence there exist $\eta \in (0, \eta_0)$ and a smooth diffeomorphism $h : \W \hookrightarrow \A_\infty$ such that:
\begin{itemize}
\item[($\mathcal{P}0$)] $h$ commutes with $R_{\alpha_n}$ (see \cref{coro:app_mod_def}).
\item[($\mathcal{P}1$)] $h(\A) = \A$, $\det \rest{Dh}{\A} =1$ and $\rest{h^{-1}}{\A \setminus \gamma} \in \mathcal{U}$.
\item[($\mathcal{P}2$)] $h$ is supported in $\croset{(\theta,y) \in \W \, : \, \lvert \Re(y) \rvert < 1-\eta }$.
\item[($\mathcal{P}3$)] $h^* J_0$ belongs to $\mathcal{W}_J$ and $h^* \Omega_0$ belongs to $\mathcal{W}_\Omega$.
\item[($\mathcal{P}4$)] $h$ commutes with $\sigma$.
\end{itemize}

We now define the wished diffeomorphism $h_{n+1}$ through its inverse:

$$h_{n+1}^{-1} := h \circ h_n^{-1} \in \emb{\infty}{\A_\rho, \A_\infty}.$$

Observe that $h_{n+1}^{-1}$ is well-defined on $\A_\rho$ because $h_n^{-1}$ belongs to $\emb{\infty}{\A_\rho, \A_\infty}$ and $h$ is defined on $\W$ which contains $h_n^{-1}(\A_\rho)$ by \cref{eq:inc_W}. Then $h_{n+1}^{-1}$ defines indeed an element of $\emb{\infty}{\A_\rho, \A_\infty}$ and is compactly supported in $\check{\A}_\rho$ by poperties $(\mathcal{P}2)$ and $(V_n)$, giving property $\underline{(V_{n+1})}$. Also property $\underline{(III_{n+1})}$ is satisfied by property $(\mathcal{P}3)$ and construction of $\mathcal{W}_J$ and $\mathcal{W}_\Omega$. Property $\underline{(IV_{n+1})}$ follows from properties $(IV_n)$ and $(\mathcal{P}4)$.\\

It remains to define $\alpha_{n+1}$ to obtain properties $(I_{n+1})$ and $(II_{n+1})$. Let $\alpha_{n+1}$ be close to $\alpha_n$.

First, by properties $(I_n)$ and $(\mathcal{P}1)$, $\rest{h_{n+1}}{\check{\A}}$ is symplectic on $\check{\A}$ and leaves $\check{\A}$ invariant. Then, by property $(\mathcal{P}1)$, and \cref{eq:def_U0}, $h_{n+1}$ belongs to $U(\rest{h_n}{\A},\alpha_n)$. Moreover, by choosing $\alpha_{n+1}$ sufficiently close to $\alpha_n$ and by property $(\mathcal{P}0)$, $(\rest{h_{n+1}}{\A},\alpha_{n+1})$ satisfies the AbC$^\star$ scheme $(U,\nu)$ and property \underline{$(I_{n+1})$} is satisfied.

Next, let us show property $(II_{n+1})$. Since $h$ commutes with $R_{\alpha_n}$ by property $(\mathcal{P}0)$, we obtain:

$$ R_{\alpha_n} \circ h_{n+1}^{-1} (\A_\rho) = h \circ R_{\alpha_n} \circ h_n^{-1}(\A_\rho).$$

Therefore, up to taking $\alpha_{n+1}$ closer to $\alpha_n$, the compact set $R_{\alpha_{n+1}} \circ h_{n+1}^{-1} (\A_{\rho'})$ is close to $h \circ R_{\alpha_n} \circ h_n^{-1}(\A_{\rho'})$ in the Hausdorff metric on compact subsets of $\A_\infty$. 

By property $(II_n)$, the compact set $h \circ R_{\alpha_n} \circ h_n^{-1}(\A_{\rho'})$ is included in $h\left(\operatorname{int}(h_{n}^{-1}(\A_\rho))\right) = \operatorname{int}(h_{n+1}^{-1}(\A_\rho)) $. 
Then, because $R_{\alpha_{n+1}} \circ h_{n+1}^{-1} (\A_{\rho'})$ is close to $h \circ R_{\alpha_n} \circ h_n^{-1}(\A_{\rho'})$, we obtain $R_{\alpha_{n+1}} \circ h_{n+1}^{-1}(\A_{\rho'})\subset h_{n+1}^{-1}(\A_{\rho})$. 
This prove the first requirement of property $\underline{(II_{n+1})}$. To obtain the remaining estimate (that $f_{n+1}$ is $2^{-(n+1)}$-close to $f_n$), note that commutation of $h$ with $R_{\alpha_n}$ guarantees that the conjugacy relation defining $f_{n+1}$ depends continuously on $\alpha_{n+1}$. 
Thus, by taking $\alpha_{n+1}$ sufficiently close to $\alpha_n$, we do have the condition $d(f_n,f_{n+1}) < 2^{-n-1}$.

\end{proof}

\paragraph{Proof of \cref{lem:abc_real_ds} on the disk and sphere}

In this part, we prove the lemma for the AbC$^\star$ realization on the sphere and the disk using the idea presented in the dedicated paragraph from \cref{sec:real} and the Approximation modulo deformation Theorem.\\

Recall that $g^\#$ and $\widehat{g} \in \symp{\infty}{\M}$ were defined on page \pageref{g_dies} for $g \in \dilo{c}{\infty}{\check{\M}_\rho}{\A_\infty}$. We aim to prove the existences of sequences $\alpha_n \in \Q/\Z$, $h_n \in \dilo{c}{\infty}{\check{\M}_\rho}{\A_\infty}$ and $f_n \in \diro{c}{\infty}{\M_{\rho'}}{\M_\rho}$ such that:
$$\alpha_0 =0 \; , \; h_0 = \rest{\pi^{-1}}{\check{\M}_\rho} \quad\text{and}\quad f_0 = \rest{\id}{\check{\M}_{\rho'}}.$$
It is also required that, for $n\geq 1$, $\alpha_n$, $h_n$ and $f_n$ satisfy
\begin{enumerate}[label = $(\Roman*_n)$]
\item $h_n(\check{\M}) = \check{\A}$, $h_n$ is symplectic on $\check{\M}$ and $(\widehat{h_n},\alpha_n)_n$ respects the AbC$^\star$ scheme $(U,\nu)$.
\item $f_n (\check{\M}_{\rho'}) \subset \operatorname{int} \left( \check{\M}_\rho \right)$ and $h_n \circ f_n = R_{\alpha_n} \circ h_n$ on $\check{\M}_{\rho'}$. Furthermore we have $d(f_n,f_{n-1}) < 2^{-n}$.
\item $J_n = h_n^\# J_0$ and $\Omega_n = h_n^\# \Omega_0$ satisfy:
$$d(J_n,J_{n-1}) < 2^{-n} \quad\text{and}\quad d(\Omega_n,\Omega_{n-1}) < 2^{-n}.$$
\item $h_n$ commutes with $\sigma$.
\end{enumerate}

\begin{proof}[Proof of Lemma \ref{lem:abc_real_ds}]

Let $n\geq 0$ be such that $h_n$, $\alpha_n$ and $f_n$ are defined and satisfy $(I_n \cdots IV_n)$\footnote{Once again, if $n=0$, properties $(I_n \cdots IV_n)$ are not defined. 
However, the only properties used to obtain the induction in the proof are the first parts of properties $(I_n)$ and $(II_n)$ and property $(IV_n)$, which are obviously satisfied for $n=0$. The induction is therefore initiated at $n=0$ again.}. 
Let us build $h_{n+1}$, $\alpha_{n+1}$ and $f_{n+1}$ such that they satisfy $(I_{n+1} \cdots IV_{n+1})$. 
To this end, likewise the previous proof, we will approximate modulo deformation a symplectomorphism $h_\A$ for neighbourhoods $\mathcal{U}$, $\mathcal{W}_J$, and $\mathcal{W}_\Omega$, and a $\sigma$-invariant compact set $\W$.\\

First we define this symplectomorphism $h_\A$ and the neighbourhood $\mathcal{U}$. 
Since $(U,\nu)$ is an AbC$^\star$ scheme, by a) of \cref{def:abc_star}, there exists a map $h_\M \in \sympo{\infty}{\M}$ commuting with $R_{\alpha_n}$ and such that $\widehat{h_n} \circ h_\M$ belongs to $U(\widehat{h_n},\alpha_n)$. 
Moreover, by b) of \cref{def:abc_star}, $U(\widehat{h_n},\alpha_n)$ belongs to $\mathcal{T}^\infty_{\gamma_\M}$ for a bicurve $\gamma_\M = \pi(\gamma_\A) \in \Gamma_2(\M)$ invariant under $R_{\alpha_n}$ and $h_\M$ satisfies $\rest{h_\M}{O(\gamma_\M)} = \id$.
In addition, because $A(\gamma_\M)$ is an open cylinder invariant under $R_{\alpha_n}$, we can assume by \cref{fact:dens} that $h_\M$ is compaclty supported in $A(\gamma_\M)$, in particular that $h_\M$ belongs to $\sympc{\infty}{\M}$.
We now consider the symplectomorphism $h_\A$ in $\sympc{\infty}{\A}$ obtained from $h_\M$ by blowing up the isolated points of $\partial\check{\M}$. 
Observe that $h_\M \circ \pi = \pi \circ h_\A$ on $\check{\A}$, hence $h_\A$ commutes with $R_{\alpha_n}$ and $\rest{h_\A}{O(\gamma_\A)} = \id$. 
Moreover this yields $\widehat{h_n} \circ h_\M = \widehat{h_n} \circ \pi \circ h_\A \circ \pi^{-1}$ on $\check{\M}$.

Now, by definition of the topology $\mathcal{T}^\infty_\gamma$, we consider $\mathcal{U}$ to be a $C^\infty$ neighbourhood of $\rest{h_\A}{\A \setminus \gamma_\A}$ in the smooth compact-open topology of $C^\infty(\A\setminus\gamma_\A,\T\times\R)$ such that the following holds: 
\begin{equation}\label{eq:def_U}
\forall \tilde{h} \in \sympc{\infty}{\A} \, , \, \rest{\tilde{h}}{\A\setminus\gamma_\A} \in \mathcal{U} \Longrightarrow  \widehat{h_n} \circ \pi \circ \tilde{h} \circ \pi^{-1} \in U(\widehat{h_n},\alpha_n).
\end{equation}
Now, let us define the $\sigma$-invariant compact set $\W$. First, for any $\eta \in (0,1)$ and $\tilde{\rho} >1$ we consider the set 
$$\check{\M}_{\tilde{\rho}}(\eta) := \check{\M}_{\tilde{\rho}} \cap \pi(\A_\infty(\eta)),$$
where $\A_\infty(\eta) := \croset{(\theta,y) \in \A_\infty, \, \left| \Re(y) \right| <1-\eta}$. Then, let $\eta_0>0$ be small enough so that:
\begin{equation}\label{eq:eta0_supp}
\croset{z \in \check{\M}_\rho \, : \, h_n(z) \neq \pi^{-1}(z)} \subset \check{\M}_\rho(\eta_0).
\end{equation}
 Since $\check{\M}_\rho(\eta_0)$ is a relatively compact subset of $\check{\M}_\rho$, there exists $R>1$ such that $h_n(\check{\M}_\rho(\eta_0)) \subset \operatorname{int}(\A_R)$. We therefore fix the $\sigma$-invariant compact set:
$$\W := \A_R.$$

We now choose the $C^\infty$-neighbourhoods $\mathcal{W}_J$ of $\rest{J_0}{\W}$ and $\mathcal{W}_\Omega$ of $\rest{\Omega_0}{\W}$ so that, for every $J \in \mathcal{W}_J$ equals $J_0$ outside $\A_\infty(\eta)$, $h_n^\# J$ is at distance $< 2^{-n-1}$ from $J_n = h_n^\# J_0$; and for every $\Omega\in \mathcal{W}_\Omega$ equals $\Omega_0$ outside $\A_\infty(\eta)$, $h_n^\# \Omega$ is at distance $< 2^{-n-1}$  from $\Omega_n = h_n^\# \Omega_0$.\\

By \cref{thm:app_mod_def}, $h_\A$ is approximable modulo structural deformation. Hence there exist $\eta \in (0, \eta_0)$ and a smooth diffeomorphism $h : \W \hookrightarrow \A_\infty$ such that:
\begin{enumerate}[label = ($\mathcal{P}$\arabic*), start = 0]
\item $h$ commutes with $R_{\alpha_n}$ (see \cref{coro:app_mod_def}).
\item $h(\A) = \A$, $\det \rest{Dh}{\A} =1$, and $\rest{h^{-1}}{\A\setminus\gamma_\A}$ belongs to $\mathcal{U}$
\item $h$ is supported in $\croset{(\theta,y) \in \W \, : \, \lvert \Re(y) \rvert < 1-\eta }$.
\item $h^* J_0$ belongs to $\mathcal{W}_J$ and $h^* \Omega_0$ belongs to $\mathcal{W}_\Omega$.
\item $h$ commutes with $\sigma$.
\end{enumerate}

We can now define the desired map $h_{n+1}$:

$$h_{n+1} : z\in \check{\M}_\rho \mapsto \left\lbrace
\begin{array}{clc}
h\circ h_n (z) &\text{if} &z \in \check{\M}_\rho(\eta)\\
\pi^{-1}(z) &\text{otherwise} &
\end{array}\right. .$$

Observe that $h_{n+1}$ is well defined and belongs to $\dilo{c}{\infty}{\check{\M}_\rho}{\A_\infty}$. Indeed, $\W$ contains $h_n(\check{\M}_\rho(\eta))$ by construction and $h_n$ equals to $\pi^{-1}$ outside $\check{\M}_\rho(\eta_0)$ which is included in $\check{\M}_\rho(\eta)$. 
Hence, by property $(\mathcal{P}2)$, $h_{n+1}$ is defined and belongs to $\dilo{c}{\infty}{\check{\M}_\rho}{\A_\infty}$. In addition, since $h$ leaves $\A$ invariant and $\rest{h}{\A}$ is compactly supported in $\check{\A}$, observe that $\widehat{h_{n+1}}$ is defined and satisfies:
\begin{equation}\label{eq:wd_n+1}
\widehat{h_{n+1}}= \widehat{h_n} \circ \iconj{\pi}{h^{-1}}.
\end{equation}

Let $\alpha_{n+1}$ be close to $\alpha_n$ and let us show that the properties $(I_{n+1} \cdots IV_{n+1})$ are satisfied.

First, by definition of $\mathcal{W}_J$ and $\mathcal{W}_\Omega$, property \underline{$(III_{n+1})$} comes immediately from properties $(\mathcal{P}2)$ and $(\mathcal{P}3)$. And by properties $(\mathcal{P}4)$ and $(IV_n)$ the property \underline{$(IV_{n+1})$} is satisfied.

Then, by properties $(I_n)$ and $(\mathcal{P}1)$, $\rest{h_{n+1}}{\check{\M}}$ is symplectic on $\check{\M}$ and leaves $\M$ invariant. Then, by property $(\mathcal{P}1)$, and \cref{eq:def_U,eq:wd_n+1}, $\widehat{h_{n+1}}$ belongs to $U(\widehat{h_n},\alpha_n)$. Moreover, since $\alpha_{n+1}$ is close to $\alpha_n$ and by property $(\mathcal{P}0)$, $(\widehat{h_{n+1}},\alpha_{n+1})$ satisfies the AbC$^\star$ scheme $(U,\nu)$ and property \underline{$(I_{n+1})$} is satisfied.\\

It remains to defines $f_{n+1}$ and obtain property $(II_{n+1})$. 
First, because $h$ commutes with $R_{\alpha_n}$ and $f_n$ satisfies property $(II_n)$, observe that we have on $\check{\M}_{\rho'}(\eta) = \check{\M}_{\rho'} \cap \pi(\A_\infty(\eta))$:
$$R_{\alpha_n} \circ h_{n+1} = h\circ R_{\alpha_n} \circ h_n = h_{n+1} \circ f_n,$$
and on $\check{\M}_{\rho'} \setminus \check{\M}_{\rho'}(\eta)$, where $h_{n+1}$ is equal to $\pi^{-1}$ and $f_n$ is equal to $R_{\alpha_n}$, we have:
$$R_{\alpha_n} \circ h_{n+1} = R_{\alpha_n} \circ \pi^{-1} = \pi^{-1} \circ R_{\alpha_n} =  h_{n+1} \circ f_n.$$
Hence we have on $\check{\M}_{\rho'}$:
\begin{equation}\label{eq:hf_n}
R_{\alpha_n} \circ h_{n+1} = h_{n+1} \circ f_n.
\end{equation}
Now, by property $(II_n)$, $f_n$ sends $\check{\M}_{\rho'}$ into the interior of $\check{\M}_\rho$, hence $f_n(\check{\M}_{\rho'}(\eta))$ is relatively compact in the interior of $\check{\M}_\rho$. The map $h_{n+1}$ is therefore defined on a neighbourhood of $f_n(\check{\M}_{\rho'}(\eta))$. Thus, for any small angle $\iota \in \T$, there exists a unique diffeomorphism $g_\iota : f_n(\check{\M}_{\rho'}(\eta)) \rightarrow \M_\infty$ which is $C^\infty$-close to the identity on $f_n(\check{\M}_{\rho'}(\eta))$ and satisfies:
$$h_{n+1} \circ g_\iota = R_\iota \circ \rest{h_{n+1}}{f_n(\check{\M}_{\rho'}(\eta))}.$$

Then, observe that $f_n(\check{\M}_{\rho'} \setminus \check{\M}_{\rho'}(\eta)) = \check{\M}_{\rho'} \setminus \check{\M}_{\rho'}(\eta)$. 
Indeed, by \cref{eq:eta0_supp} and since $\eta \in (0,\eta_0)$, we have that $h_n$ coincides with $\pi^{-1}$ on $\check{\M}_{\rho'} \setminus \check{\M}_{\rho'}(\eta)$. Therefore, by $(II_n)$,  $f_n$ is equal to $R_{\alpha_n}$ on $\check{\M}_{\rho'} \setminus \check{\M}_{\rho'}(\eta)$ which is left invariant by rotations.

Next, by uniqueness of $g_\iota$ and since $h_{n+1}$ is equal to $\pi^{-1}$ on $f_n(\check{\M}_{\rho'} \setminus \check{\M}_{\rho'}(\eta)) = \check{\M}_{\rho'} \setminus \check{\M}_{\rho'}(\eta)$, we can extend $g_\iota$ smoothly by setting it equals to the rotation $R_\iota$ on $f_n( \M_{\rho'} \setminus \check{\M}_{\rho'}(\eta))$. Then we have on $f_n(\check{\M}_{\rho'})$:
\begin{equation}\label{eq:iota}
h_{n+1} \circ g_\iota = R_\iota \circ h_{n+1}.
\end{equation} 

We now fix $\iota = \alpha_{n+1} - \alpha_n$, which is small because $\alpha_{n+1}$ is close to $\alpha_n$, and we define $f_{n+1} = g_\iota \circ f_n$. First, as $g_\iota$ and $f_n$ are rotations on the complement of compact subsets of $\check{\M}_{\rho'}$, $f_{n+1}$ is an element of $\diro{c}{\infty}{\M_{\rho'}}{\check{\M}_\rho}$. Next, because $\iota$ is small $g_\iota$ is $C^\infty$-close to the identity, then $d(f_{n+1},f_n) < 2^{-n-1}$. It also implies with property $(II_n)$ that $f_{n+1}(\check{\M}_{\rho'}) \subset \operatorname{int}(\check{\M}_\rho)$. Finally, it holds on $\check{\M}_{\rho'}$ by \cref{eq:hf_n,eq:iota}:
$$h_{n+1} \circ f_{n+1} = h_{n+1} \circ g_\iota \circ f_n = R_\iota \circ h_{n+1} \circ f_n = R_\iota \circ R_{\alpha_n} \circ h_{n+1} = R_{\alpha_{n+1}} \circ h_{n+1}.$$

This yields \underline{$(II_{n+1})$} and concludes the proof.
\end{proof}

\newpage
\label{index}
\begin{multicols}{2}
\printglossary[title = {Index of notations}]
\end{multicols}
\bibliographystyle{alpha} 
\bibliography{biblio}
\addcontentsline{toc}{section}{\bfseries{Bibliography}}
\end{document}